\documentclass[article,10pt,reqno]{amsart}
\usepackage{amscd,amssymb,graphicx,color,epigraph}

\usepackage[shortalphabetic,initials]{amsrefs}

\newcommand{\red}[1]{{\color{red} #1}}
\newcommand{\blue}[1]{{\color{blue} #1}}
\usepackage[dvipsnames]{xcolor}
\usepackage[T1]{fontenc}
\usepackage{hyperref}
\hypersetup{
colorlinks = True,
linkcolor=Bittersweet,
citecolor=red,
}
\usepackage{enumitem}
\setlist[itemize]{leftmargin=18pt}
\setlist[enumerate]{leftmargin=18pt}
\usepackage{comment}

\usepackage{amssymb}
\usepackage{amsthm}
\usepackage{bm}
\usepackage{latexsym}
\usepackage{amsmath}
\usepackage{hyperref}
\usepackage{eufrak}
\usepackage{mathrsfs}
\usepackage{amscd}
\usepackage[all,cmtip]{xy}
\usepackage{leftindex}
\usepackage{color}
\usepackage{amscd}
\usepackage{listings}
\usepackage[center]{caption}
\theoremstyle{plain}
\usepackage{dirtytalk}

 \numberwithin{equation}{section}
 \newcommand{\eni}{k1A}
\newcommand{\enii}{k2A}

\newtheorem{theorem}{Theorem}[section]
\newtheorem{proposition}[theorem]{Proposition}

\newtheorem{lemma}[theorem]{Lemma}

\newtheorem{corollary}[theorem]{Corollary}

\newtheorem{conjecture}[theorem]{Conjecture}

\theoremstyle{definition}

\newcommand{\appsection}[1]{\let\oldthesection\thesection
\renewcommand{\thesection}{Appendix \oldthesection}
\section{#1}\let\thesection\oldthesection}

\newtheorem{definition}[theorem]{Definition}
\newtheorem{propdef}[theorem]{Proposition-Definition}

\newtheorem{remark}[theorem]{Remark}
\newtheorem{example}[theorem]{Example}

\def\D{{\mathbb{D}}}

\def\Z{{\mathbb{Z}}}
\def\F{{\mathbb{F}}}
\def\Q{{\mathbb{Q}}}
\def\C{{\mathbb{C}}}
\def\P{{\mathbb{P}}}

\def\O{{\mathcal{O}}}
\def\Wa{{\texttt{W}}}
\def\W{{\mathcal{W}}}
\def\bW{{\overline{W}}}
\def\bbW{{\overline{\mathcal{W}}}}

\usepackage{xcolor}

\makeatletter
\providecommand{\leftsquigarrow}{%
  \mathrel{\mathpalette\reflect@squig\relax}%
}
\newcommand{\reflect@squig}[2]{%
  \reflectbox{$\m@th#1\rightsquigarrow$}%
}
\makeatother

\title{The birational geometry of Markov numbers}

\author{Giancarlo Urz\'ua}
\address{Facultad de Matem\'aticas,
Pontificia \allowbreak Universidad \allowbreak{Cat\'olica} de Chile, Santiago, Chile.}
\email{gianurzua@gmail.com}

\author{Juan Pablo Z\'u\~niga}
\address{Facultad de Matem\'aticas,
Pontificia Universidad Cat\'olica de Chile, Santiago, Chile.}
\email{jpzuniga3@uc.cl}

\date{\today}

\begin{document}

\begin{abstract}
It is known that all degenerations of the complex projective plane into a surface with only quotient singularities are controlled by the positive integer solutions $(a,b,c)$ of the Markov equation $$x^2+y^2+z^2=3xyz.$$ It turns out that these degenerations are all connected through finite sequences of other simpler degenerations by means of birational geometry. In this paper, we explicitly describe these birational sequences and show how they are bridged among all Markov solutions. For a given Markov triple $(a,b,c)$, the number of birational modifications depends on the number of branches that it needs to cross in the Markov tree to reach the Fibonacci branch. We show that each of these branches corresponds exactly to a Mori train of the flipping universal family of a particular cyclic quotient singularity defined by $(a,b,c)$. As a byproduct, we obtain new numerical/combinatorial data for each Markov number, and new connections with the Markov conjecture (Frobenius Uniqueness Conjecture), which rely on Hirzebruch-Jung continued fractions of Wahl singularities. 
\end{abstract}

\dedicatory{To the memory of Martin Aigner}


\maketitle

\tableofcontents



\section{Introduction} \label{s0}

The complex projective plane $\P^2$ is rigid. Although it may degenerate into singular surfaces. After the work of B\v{a}descu \cite{B86}, Manetti \cite{Ma91} and Hacking \cite{H04},  Hacking-Prokhorov \cite{HP10} classified all possible degenerations of $\P^2$ into normal projective surfaces with only quotient singularities. They proved that every degeneration is a $\Q$-Gorenstein partial smoothing of $\P(a^2,b^2,c^2)$, where $(a,b,c)$ satisfies the Markov equation $$x^2+y^2+z^2=3xyz.$$ 

We recall that the positive integer solutions of this equation are called \textit{Markov triples}, and the set of all coordinates are the \textit{Markov numbers}. Any permutation of a Markov triple is a solution again, and so we typically order them from smaller to bigger. If $(a,b,c)$ is a solution, then its \textit{mutation} $(a,b,3ab-c)$ is a solution as well. Every Markov triple is obtained from $(1,1,1)$ by permuting and mutating some number of times. Markov triples form an infinite tree of valency $3$ at all vertices except for $(1,1,1)$ and $(1,1,2)$. (We briefly review all basics on Markov triples in Section \ref{s2}.) The old and famous \textit{Markov conjecture} \cite{Aig13} (known also as Frobenius Uniqueness Conjecture \cite{Fro13}) states that in a Markov triple $(a,b<c)$ the integer $c$ determines the integers $a,b$. Markov conjecture has been checked for Markov numbers up to $10^{15000}$ \cite{Per22}. In this paper, we find some new equivalences to this conjecture.

\begin{figure}[htbp]
\centering
\includegraphics[width=12.7cm]{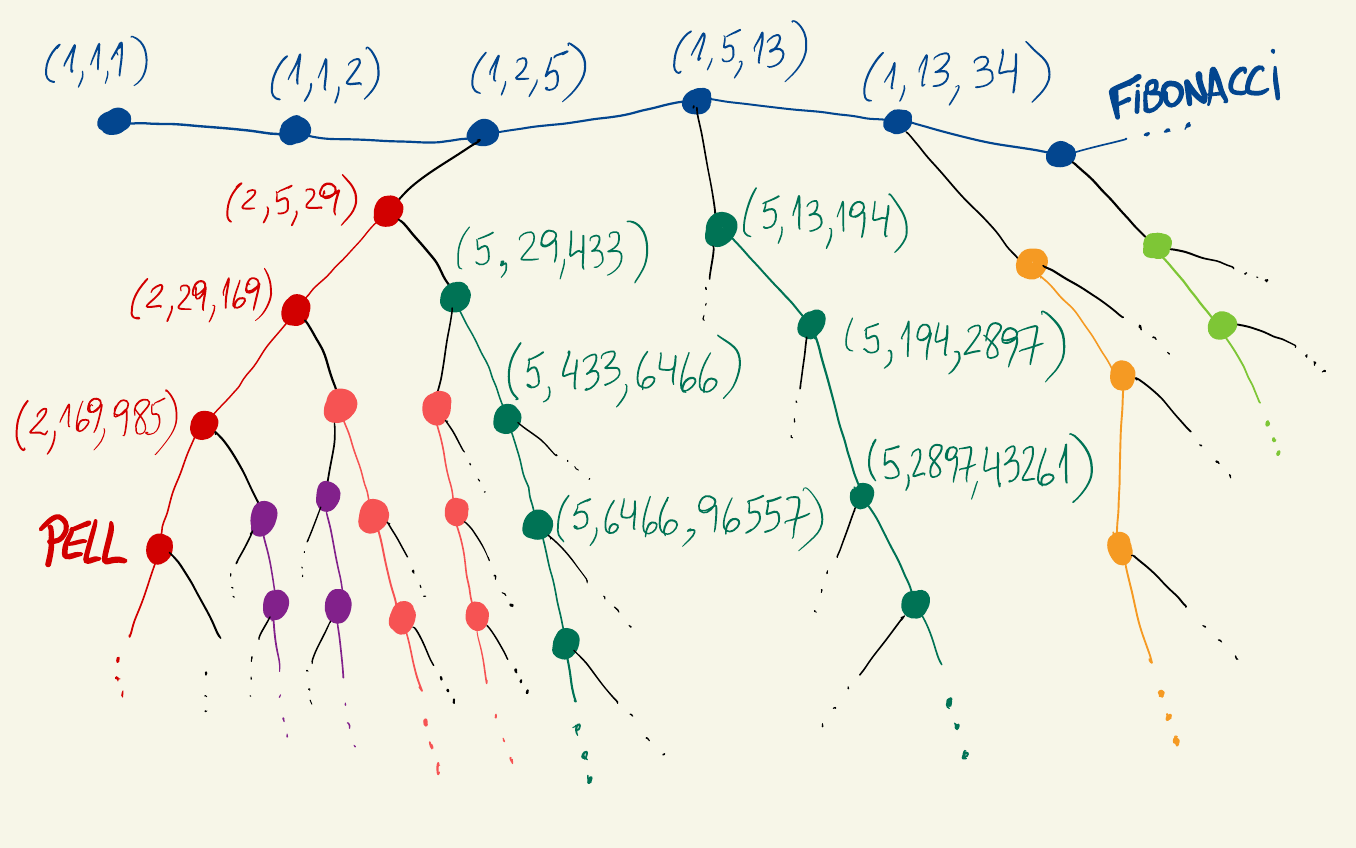}
\caption{Markov tree and some branches} 
\label{f0}
\end{figure}


The Hacking-Prokhorov degenerations of $\P^2$ are part of the bigger picture of $\Q$-Gorenstein deformations of surfaces. They are relevant to understand arbitrary degenerations of surfaces, particularly in the Koll\'ar--Shepherd-Barron--Alexeev compactification of the moduli space of surfaces of general type \cite{KSB88,Ale94}. They have been used in various ways. For example, construction of a compact moduli space for plane curves \cite{H01,H04}, geography of surfaces of general type (e.g. \cite{LP07,PPS09a, PPS09b,RU21}), construction of exotic 4-manifolds and/or diffeomorphism  type of surfaces (e.g. \cite{FS97,P97,Ma01,RU22}), and exceptional collections of vector bundles (e.g. \cite{H12,H13,K21,TU22}).

To be precise, consider a projective surface $W_0$ with only Wahl singularities, and assume we have a $\Q$-Gorenstein smoothing $W_t$ over a disk $\D$ (for definitions see Section \ref{s4}). We summarize this with the symbol $W_t \rightsquigarrow W_0$. It turns out that for this type of deformation, we can run an explicit \textit{Minimal Model Program (MMP)} for the canonical class relative to the base $\D$ \cite{M02,HTU17,Urz16a, Urz16b}, which ends with either surfaces with nef canonical class, or nonsingular deformations of ruled surfaces, or degenerations of $\P^2$ with quotient singularities. 

\begin{definition}
A \textit{Markovian plane} is a degeneration of $\P^2$ with quotient singularities, i.e. a $\P^2 \rightsquigarrow W$ where $W$ is a normal projective surface with only quotient singularities. 
\label{markovian}
\end{definition}

Although Markovian planes are the ends of the MMP, we can still run the MMP on them \cite[Section 3]{Urz16a}. Consider a Markovian plane $\P^2 \rightsquigarrow W$. We cannot run MMP here. Instead, we blow up a general section over $\D$. (We could also take more special blow-ups, even over the singularities of $W$, see \cite[Section 2]{Urz16a}.) Then we have $$\F_1 \rightsquigarrow \text{Bl}_{\text{pt}}(W)=:W_0,$$ where $\F_m$ is the Hirzebruch surface of index $m$. Now we have two options for running MMP on $\F_1 \rightsquigarrow W_0$: we can take either a divisorial contraction, returning to the Markovian plane, or a flip \cite[Section 3]{Urz16a}. And so we do the flip. After that, we obtain a $\F_1 \rightsquigarrow W_1$, and if $W_1$ is singular, then we have a flip again and so on, until we reach a $\F_1 \rightsquigarrow W_{\mu}$ with $W_{\mu}$ nonsingular, i.e. a Hirzebruch surface. If $W=\P(a^2,b^2,c^2)$ has $r$ singularities, then $W_{\mu}=\F_{2r+1}$.

Since a general Markovian plane is governed by the toric surface $\mathbb{P}(a^2, b^2, c^2)$, we restrict our focus to the case $W=\mathbb{P}(a^2,b^2,c^2)$. In this setting, the flips from $W_0$ to $W_{\mu}$ give particular numerical/combinatorial data to the Markov triple $(a,b,c)$. In this process, we find connections between Markov triples via particular properties of cyclic quotient singularities. For example, the Markov conjecture is about singularities that admit extremal P-resolutions of a special kind. There is a well-established machinery to study them (see e.g. \cite[Section 4]{HTU17}). In general, when we perform a flip from $W_t\rightsquigarrow W^-$ to $W_t\rightsquigarrow W^+$, an extremal P-resolution is precisely a 2-dimensional neighborhood of the flipped curve in $W^+$.

\begin{definition}
An \textit{extremal P-resolution} $f^+ \colon W^+ \to \bW$ of a cyclic quotient singularity germ $(P \in \bW)$ is a partial resolution with only Wahl singularities such that ${f^+}^{-1}(P)$ is a nonsingular rational curve $\Gamma^+$, and $\Gamma^+ \cdot K_{W^+}>0$. Thus, $W^+$ has at most two singularities (\cite[Lemma 3.14]{KSB88}). 
\label{extremalPres}
\end{definition}

At this point, one could ask: \textit{How does the number of flips $\mu$ depend on $(a,b,c)$? Is it possible to express the chain of flips explicitly?} The main purpose of this paper is to completely describe this MMP for any Markov triple, and reinterpreting Markov conjecture in some new ways as a byproduct.
\vspace{0.3cm}

Our first theorem shows how Markov triples are connected (through deformations and MMP) on a \textit{branch} of the Markov tree. Given a Markov triple $(a<b<c)$, we define its two branches as the set of $(c<m_k<m_{k+1})$ with $k\geq 0$ in one of the two chains
{\scriptsize $$ (a<b<c)-(a<c<3ac-b)-(c<m_0<m_1)-(c<m_1<m_2)-\ldots-(c<m_{k}<m_{k+1})-\ldots$$} where $m_0=3ac-b$, $m_1=3(3ac-b)c-a$, and $m_{k+1}=3m_{k-1}c-m_k$ for $k \geq 1$.

{\scriptsize $$ (a<b<c)-(b<c<3bc-a)-(c<m_0<m_1)-(c<m_1<m_2)-\ldots-(c<m_{k}<m_{k+1})-\ldots$$} where $m_0=3bc-a$, $m_1=3(3bc-a)c-b$, and $m_{k+1}=3m_{k-1}c-m_k$ for $k \geq 1$. 

For the triples $(1,1,1)$ and $(1,1,2)$ we have only one branch (see Section \ref{s2}), we name them as the \textit{Fibonacci branch} and the \textit{Pell branch} respectively. In Figure \ref{f0} we have different colors for branches of different Markov numbers $a$ of Markov triples $(a<b<c)$. For example, the triple $(1<2<5)$ defines the two green branches of $5$ in Figure \ref{f0}. 

\vspace{0.1cm}

On the other hand, any cyclic quotient singularity that admits an extremal P-resolution defines a universal antiflipping family \cite{HTU17} and the corresponding Mori trains (see Section \ref{s4} for more details). 

\begin{definition}
Let $\frac{1}{\Delta}(1,\Omega)$ be a cyclic quotient singularity. A \textit{Mori train} is the combinatorial data to construct all divisorial contractions or flips over $\frac{1}{\Delta}(1,\Omega)$. This will be precisely described in Section \ref{s4}, but it involves the full combinatorial data of all the k1A and k2A extremal neighborhoods, which are deformation equivalent, over $\frac{1}{\Delta}(1,\Omega)$. The train wagons are all the Hirzebruch-Jung continued fractions of the Wahl singularities in these k1A and k2A extremal neighborhoods.      
\label{MoriTrainI}
\end{definition}  


\begin{definition}
Let $n\geq 0$. We say that a collection of $\Q$-Gorenstein smoothings $\{\F_n \rightsquigarrow W_i\}$ \textit{stabilizes at} the $k$th flip if for every $i$ we can apply $k$ consecutive flips on $\F_n \rightsquigarrow W_i$ to obtain a $\F_n \rightsquigarrow W$ for some fixed $W$. On the other hand, we say that a given $\F_n \rightsquigarrow W_0$ \textit{stabilizes into} $\F_n \rightsquigarrow W_1$ at the $k$th flip if after $k$ flips on $\F_n \rightsquigarrow W_0$ and $i$ flips on $\F_n \rightsquigarrow W_1$ for some $0 \leq i \leq k$, we obtain $\F_n \rightsquigarrow W$ with a common $W$.
\label{stabilization}
\end{definition}

\begin{theorem}
Let $(a,b<c)$ be a Markov triple, and let $i>0$. Then the MMP on $$\F_1 \rightsquigarrow W_0=\text{Bl}_{\text{pt}}\big(\P(c^2,m_i^2,m_{i+1}^2) \big)$$ corresponding to each of the two branches defined by $(a,b<c)$ stabilizes at the $3$rd flip (see Definition \ref{stabilization}). Moreover, for each fixed branch the antiflips at the 3rd flip are k2A neighborhoods of one Mori train over the cyclic quotient singularity $\frac{1}{\Delta}(1,\Omega)$, where $$\Delta=c^2(c^2D-(c-1)^2), \ \ \ \ \ \ \Omega=c^2+(c \zeta +1)(c^2D-(c-1)^2),$$ $\zeta=w_c,c-w_c$ (one value for each branch), and $D=9c^2-4$. (The integer $w_c$ is the T-weight of $c$ as in Definition \ref{r,w}.) We have that the $\delta$ invariant for both Mori trains is equal to $3c$. 
\label{branches-trains}
\end{theorem}

Hence, one can think of a Markov branch as a Mori train. It turns out that the particular singularity $\frac{1}{\Delta}(1,\Omega)$ can be reduced to the singularity $\frac{1}{\Delta_0}(1,\Omega_0)$, where $$\Delta_0=(4c+w_c)(5c-w_c)-9   \ \ \ \text{and} \ \ \ \Omega_0=c(4c+w_c)-1,$$ in the sense that one singularity admits an extremal P-resolution if and only if the other does, and with the same $\delta=3c$. We analyze the connection of this singularity with the Markov conjecture in Section \ref{s3}. Another thing to highlight at this point is that the Mori trains over these singularities depend on the infinite continued fraction $$\frac{3c+ \sqrt{9c^2-4}}{2}=3c  - \frac{1}{3c - \frac{1}{\ddots}},$$ and so the appearance of the discriminant $D=9c^2-4$. This is the discriminant of the Markov quadratic form associated with $(a,b<c)$ \cite[Definition 2.5]{Aig13}.    


\begin{definition}
We say that we \textit{change branches} when we move from a branch to a different adjacent branch in the Markov tree. Precisely it means that we move from the initial Markov triple of a branch $(c<m_0<m_1)$ to the adjacent vertex $(3cm_0-m_1<c<m_0)$ of another branch. 
\label{changebranches}
\end{definition}

It turns out that a change in branches is reflected in a stabilization of the MMP of the corresponding degenerations.

\begin{theorem}
Let $(c<m_0<m_1)$ be the initial Markov triple of a branch with $1<3cm_0-m_1$. Then $\F_1  \rightsquigarrow \text{Bl}_{\text{pt}}\P(c_0^2,m_0^2,m_1^2) $ stabilizes into $$\F_1 \rightsquigarrow \text{Bl}_{\text{pt}} \ \P((3cm_0-m_1)^2,c^2,m_0^2)$$ at the $k$th flip for some $k\leq 12$.
\label{changes-branches}
\end{theorem}

With these two theorems, we ensemble all the corresponding Mori trains of ``minimal" Markov numbers $a$ (i.e. depending on triples $(a<b<c)$) in a decreasing order, until we arrive at the Fibonacci branch, which has its own Mori train over the singularity $\frac{1}{5}(1,1)$. In particular the first antiflip over $\frac{1}{7}(1,1)$ (i.e. the general case with $3$ singularities which arrives to $\F_7$) has always the same central singular surface with one Wahl singularity $\frac{1}{36}(1,5)$. We summarize it all in the next theorem.   


\begin{theorem}
Let $(1<a<b<c)$ be a Markov triple, and consider the shortest connected path from $(a<b<c)$ to the Fibonacci branch. Let $\nu$ be the number of branches it needs to cross in the Markov tree to become $(a_{\nu}=1<b_\nu<c_\nu)$. 

Then, the MMP on $\F_1 \rightsquigarrow W_0$ for the Markovian plane corresponding to $\P(a^2,b^2,c^2)$ needs at most $6\nu+3$ flips to reach the smooth deformation $\F_1 \rightsquigarrow \F_7$. (For the particular case $(a=1<b<c)$ we need $3$ flips to reach $\F_1 \rightsquigarrow \F_5$, and for $(1,1,2)$ we need only $1$ flip to reach $\F_1 \rightsquigarrow \F_3$.) The upper bound $6\nu+3$ is optimal. 
\label{complete-MMP}
\end{theorem} 

Theorem \ref{complete-MMP} implies that the amount of flips tends to infinite if and only if the amount of changes in branches does. At each step the flips are unique, and so this gives a unique numerical data associated to each Markov triple. 


As an example, we reproduce here the numerical data associated to 
$(5,29,433)$. The surface $\mathbb{P}(5^2,29^2,433^2)$ has $3$ Wahl singularities with associated Hirzebruch-Jung continued fractions:

\begin{itemize}
    \item $\frac{1}{433^2}(1,433\cdot104-1), \frac{433^2}{433\cdot 104-1}=[5, 2, 2, 2, 2, 2, 10, 5, 2, 2, 2, 2, 2, 2, 2, 8, 2, 2, 2]$
    \item $\frac{1}{29^2}(1,29\cdot 7 -1), \frac{29^2}{29\cdot 7-1}=[5, 2, 2, 2, 2, 2, 10, 2, 2, 2]$
    \item $\frac{1}{5^2}(5\cdot 1 -1), \frac{5^2}{5\cdot 1-1}=[7, 2, 2, 2]$.
\end{itemize}
In this way, the following chain of numbers represents the self-intersections in its minimal resolution and connecting $(-1)$-curves.

{\scriptsize $$[7, 2, 2, 2]-(1)-[5, 2, 2, 2, 2, 2, 10, 5, 2, 2, 2, 2, 2, 2, 2, 8, 2, 2, 2]-(1)-[5, 2, 2, 2, 2, 2, 10, 2, 2, 2]$$}

After that, we consider $\F_1  \rightsquigarrow \text{Bl}_{\text{pt}}\P(5^2,29^2,433^2)$, and we run MMP as described in Section \ref{s4}.  The MMP stops after $9$ flips with a smooth deformation $\F_1  \rightsquigarrow \F_7$. At each step, the flipping curves are $(-1)$-curves in the corresponding minimal resolutions, and so they are showed as $(1)_{-}$, and the flipped curves are showed as $(c)_{+}$. Then the numerical data from this MMP for $(5,29,433)$ is: 

\vspace{0.1cm}
Flip 1:

{\scriptsize [7, 2, 2, 2]-(1)-[5, 2, 2, 2, 2, 2, 9, 5, 2, 2, 2, 2, 2, 2, 8, 2, 2, 2]-$(1)_{+}$-[5, 2, 2, 2, 2, 9, 2, 2, 2]-$(1)_{-}$-[5, 2, 2, 2, 2, 2, 10, 2, 2, 2]}

Flip 2:  

{\scriptsize [7, 2, 2, 2]-(1)-[5, 2, 2, 2, 2, 2, 9, 5, 2, 2, 2, 2, 2, 2, 8, 2, 2, 2]-$(1)_{-}$-[6, 2, 2]-$(1)_{+}$-[5, 2, 2, 2, 2, 9, 2, 2, 2]}

Flip 3:  

{\scriptsize [7, 2, 2, 2]-$(1)_{-}$-[6, 2, 2]-$(1)_{+}$-[5, 2, 2, 2, 2, 9, 5, 2, 2, 2, 2, 2, 2, 7, 2, 2, 2]-(1)-[5, 2, 2, 2, 2, 9, 2, 2, 2]}

Flip 4:  

{\scriptsize [6, 2, 2]-$(2)_{+}$-$(1)_{-}$-[5, 2, 2, 2, 2, 9, 5, 2, 2, 2, 2, 2, 2, 7, 2, 2, 2]-(1)-[5, 2, 2, 2, 2, 9, 2, 2, 2]}

Flip 5:

{\scriptsize  [6, 2, 2]-(1)-[4, 2, 2, 2, 2, 9, 5, 2, 2, 2, 2, 2, 2, 7, 2, 2]-$(2)_{+}$-$(1)_{-}$-[5, 2, 2, 2, 2, 9, 2, 2, 2]}

Flip 6:   

{\scriptsize [6, 2, 2]-(1)-[4, 2, 2, 2, 2, 9, 5, 2, 2, 2, 2, 2, 2, 7, 2, 2]-$(1)_{-}$-[4, 2, 2, 2, 2, 9, 2, 2]-$(2)_{+}$} 

Flip 7:

{\scriptsize [6, 2, 2]-$(1)_{-}$-$(4)_{+}$-[2, 2, 2, 2, 8]-(0)}

Flip 8: 

{\scriptsize $(5)_{+}$-$(1)_{-}$-[2, 2, 2, 2, 8]-(0)}

Flip 9: 

{\scriptsize (0)-$(7)_{+}$-(0)}
\bigskip 

The complete numerical computation of the MMP for a general Markov triple is detailed in Section \ref{s6}. For more examples, we refer to the computer program \cite{programa}.

We highlight that this is a small part of the bigger picture of antiflips of smooth deformations of rational surfaces. In particular, in a subsequent work, we will describe the situation for deformations of Hirzebruch surfaces, and how Markovian planes sit in the general picture for $\F_1$. 

\vspace{0.3cm} 

Let us finish this introduction with various combinatorial characterizations of Markov numbers among all integers, and equivalences to the Markov conjecture. The statements use the notation of Hirzebruch-Jung continued fractions, Wahl chains (Definition \ref{WahlChain}), and weights of a Markov triple, which are reviewed in Sections \ref{s1} and \ref{s2}. Reinterpretations of the bijection between the Farey tree (reduced version of the Stern-Brocot tree) and the Markov tree via the Wahl tree, and Cohn words via Wahl-2 chains (Definition \ref{Wahl2Chain}) are included. These results are discussed in Sections \ref{s2} and \ref{s3}. The initial sections do not require any knowledge of birational geometry.

\begin{theorem} [Proposition \ref{markovrc}]
Let $\frac{a}{r_a}, \frac{b}{r_b}, \frac{c}{r_c}$ be the fractions of Wahl-2 chains. We have $\frac{c}{r_c}=[\frac{a}{r_a},4,\frac{b}{r_b}]$ if and only if $(a,b<c)$ is a Markov triple.
\label{markovrcintro}
\end{theorem}

Let $0<q<m$ be coprime integers. Consider the Hirzebruch-Jung continued fractions $$\frac{m}{q}=[x_1,\ldots,x_r] \ \ \text{and} \ \ \frac{m}{m-q}=[y_1,\ldots,y_s].$$ One can check that $\frac{m^2}{mq-1}=[x_1,\ldots,x_{r}+y_{s},\ldots,y_1]$.

\begin{theorem} [Proposition \ref{combi1}]
We have $\frac{m^2}{mq-1}=[\Wa_0^{\vee},10,\Wa_1^{\vee}]$ for some non-empty Wahl chains $\Wa_i$ (where $\Wa_i^{\vee}$ are dual Wahl chains) if and only if $m$ is a Markov number. In fact, if $\frac{n_0^2}{n_0a_0-1}=\Wa_0$ and $\frac{n_1^2}{n_1 a_1-1}=\Wa_1$, then $n_0^2+n_1^2+m^2=3 n_0n_1m$. 
\label{combi1intro}
\end{theorem}

\begin{theorem} [Proposition \ref{combi3}]
We have $$[5,x_1,\ldots,x_r,2,y_s,\ldots,y_1,5]=[\Wa_0,2,\Wa_1]$$ if and only if there is a Markov triple $(1<a,b<c=:m)$.
\label{combi3intro}
\end{theorem}

The last theorem comes from the singularity $\frac{1}{\Delta_0}(1,\Omega_0)$ explained above, and precisely shows an extremal P-resolution with Wahl singularities corresponding to $\Wa_0$ and $\Wa_1$, the $(-2)$-curve $\Gamma^+$ (in the minimal resolution), and $\delta=3c$. In general, by \cite[Section 4]{HTU17}, there are at most two such extremal P-resolutions for a fixed pair $m,q$. But Markov conjecture is known when we fix $m$ and $q$, and so there is only one extremal P-resolution \cite[Proposition 3.15]{Aig13}.
\vspace{0.1cm} 

Let us recall some equivalences to the Markov conjecture, including the connections made in this paper. Our choice is oriented towards algebraic geometry, for more equivalences see \cite{Aig13}. 

\begin{itemize}
    \item[(I)] Let $c>2$ be a Markov number. Then the equation $x^2 \equiv -1 ($mod $c)$ has only two solutions as weights $r_c$ of some Markov triple $(a,b,c)$. The other is $c-r_c$. 
    

    \item[(II)] Consider a degeneration of $\P^2$ into a projective surface with one quotient singularity of some given order. Then the singularity is unique. In that case, the order is the square of a Markov number $c$, and the singularity is $\frac{1}{c^2}(1,c w_c-1)$. 
    
    \item[(III)]  Up to dualizing and tensoring by line bundles, an exceptional vector bundle in $\P^2$ is uniquely determined by its rank (conjectured by A. N. Tyurin \cite{Rud88}). This rank is always a Markov number, and the statement is about the uniqueness of the slope of the vector bundle in $[0,1/2]$.

    \item[(IV)] Given an integer $m$, there are at most two $0<q<m$ coprime  such that $$\frac{m}{q}=\left[\frac{m_0}{q_0},4,\frac{m_1}{q_1}\right]$$ where $\frac{m}{q},$ $\frac{m_0}{q_0}$ and $\frac{m_1}{q_1}$ are the fractions of Wahl-2 chains (see Theorem \ref{markovrcintro}).

    \item[(V)]  Given an integer $m$, there are at most two $0<q<m$ coprime such that $$\frac{m^2}{mq-1}=[\Wa_0^{\vee},10,\Wa_1^{\vee}]$$ for some Wahl chains $\Wa_i$ (see Theorem \ref{combi1intro}).

    \item[(VI)]  Given an integer $m$, then there are at most two $0<q<m$ coprime such that $$\left[5,\frac{m}{q},2,\frac{m}{m-q'},5\right] =[\Wa_0,2,\Wa_1]$$ for some Wahl chains $\Wa_0$ and $\Wa_1$, where $0<q'<m$ and $q q' \equiv 1$ (mod $m$) (see Theorem \ref{combi3intro}).


\end{itemize}

\subsubsection*{Acknowledgments} We thank Jonny Evans, Markus Perling, and Nicol\'as Vilches for useful discussions, and to the anonymous referee for a very helpful report. The paper was partially written while the first author was at the Freiburg Institute for Advanced Studies under a Marie S. Curie FCFP fellowship. He thanks the institute, Stefan Kebekus, and Adrian Langer for their hospitality. The first author was supported by the FONDECYT regular grant 1230065. The second author was supported by ANID-Subdirecci\'on de Capital Humano/Doctorado Nacional/ 2022-21221224.

\section{Hirzebruch-Jung continued fractions and Wahl chains} \label{s1}

\begin{definition}
Let $\{e_1,\ldots,e_r \}$ be a sequence of positive integers. We say that it admits a \textit{Hirzebruch-Jung continued fraction (HJ continued fraction)} if $$[e_i,\ldots, e_r]:=  e_i - \frac{1}{e_{i+1} - \frac{1}{\ddots - \frac{1}{e_r}}}$$ is positive for all $i\geq 2$. Its value is $[e_1,\ldots, e_r]$.  
\end{definition}

If $e_i \geq 2$ for a given $\{e_1,\ldots,e_r \}$, then the sequence admits a HJ continued fraction and $[e_1,\ldots,e_r] >1$. In fact, there is a one-to-one correspondence between $[e_1,\ldots, e_r]$ with $e_i\geq 2$ and rational numbers greater than $1$. Hence for any coprime integers $0<q<m$ we associate a unique HJ continued fraction $$\frac{m}{q}=[e_1,\ldots,e_r]$$ with $e_i\geq 2$ for all $i$. The presence of $1$s in an admissible sequence $\{e_1,\ldots,e_r \}$ produces non-uniqueness of the HJ continued fractions for the same value $[e_1,\ldots,e_r]$, and if this value is a rational number smaller than or equal to $1$, then we are forced to have $1$s for some $ e_i$s. This non-uniqueness is derived from the ``arithmetic blowing-up" identity $$u- \frac{1}{v} = u+1 - \frac{1}{1-\frac{1}{v+1}}.$$

For example, the HJ continued fractions associated with the value $0$, which will be called \textit{zero continued fractions}, are:

$[1,1]$,

$[1,2,1]$, $[2,1,2]$, 

$[1,2,2,1]$, $[2,1,3,1]$, $[1,3,1,2]$, $[3,1,2,2]$, $[2,2,1,3]$,

etc.

\begin{remark}
There is a well-known one-to-one correspondence between the previous list of zero continued fractions and triangulations of polygons \cite{C91,S91,HTU17}. A \textit{triangulation of a convex polygon} $P_0P_1 \dots P_s$ is given by drawing some non intersecting diagonals on it which divide the polygon into triangles. For a fixed triangulation, one defines $v_i$ as the number of triangles that have $P_i$ as one of its vertices. Note that $$v_0+v_1+\ldots+v_s = 3(s-1).$$
Via an easy induction on $s$, one can show that $[k_1, \dots, k_s]$ is a zero continued fraction if and only if there exists a triangulation of $P_0P_1\dots P_s$ such that $v_i=k_i$ for every $1 \leq i \leq s$. In this way, the number of zero continued fractions of length $s$ is the \textit{Catalan number} $$\frac{1}{s}\binom{2(s-1)}{s-1}.$$
\label{catalan}
\end{remark}

\begin{propdef}
Let $0<q<m$ be coprime integers, and let $\frac{m}{q}=[x_1,\ldots,x_r]$ be its HJ continued fraction (with $x_i\geq 2$). Then we have the following:
\begin{itemize}
    \item[(1)] $\frac{m}{q'}=[x_r,\ldots,x_1]$ where $0<q'<m$ satisfies $qq' \equiv 1($mod $m)$.
    \item[(2)] The \textit{dual HJ continued fraction} $\frac{m}{m-q}:=[y_1,\ldots,y_s]$ satisfies $$[x_1,\ldots,x_r,1,y_s,\ldots,y_1]=0.$$
    \item[(3)] $\frac{m^2}{mq-1}=[x_1,\ldots,x_r+y_s,\ldots,y_1]$, and $\frac{m^2}{m(m-q)+1}=[y_1,\ldots,y_s,2,x_r,\ldots,x_1].$    
\end{itemize}
\label{basico}
\end{propdef}

\begin{proof}
These are well-known facts on HJ continued fractions. See \cite[Remark 2.8]{HP10} for (1), \cite[Section 2.1]{UV22} for (2), and \cite[Lemma 8.5]{HP10} for (3).
\end{proof}

\begin{definition}
Let $0<a<n$ be coprime integers. A \textit{Wahl chain} is the collection of numbers corresponding to the HJ continued fraction of $\frac{n^2}{na-1}$. 
\label{WahlChain}
\end{definition}

Every Wahl chain can be obtained via the following algorithm due to J. Wahl (see \cite[Prop.3.11]{KSB88}): 

\begin{itemize}
    \item[(i)] $[4]$ is the Wahl chain for $n=2$ and $a=1$.
    \item[(ii)] If $[e_1,\ldots,e_r]$ is a Wahl chain, then $[e_1+1,e_2,\ldots,e_r,2]$ and $[2,e_1,\ldots,e_{r-1},e_r+1]$ are Wahl chains.
    \item[(iii)] Every Wahl chain is obtained by starting with (i) and iterating the steps in (ii).
\end{itemize}

As we saw in Proposition \ref{basico} part (iii), the Wahl chain of $\frac{n^2}{na-1}$ can be constructed from the HJ continued fraction of $\frac{n}{a}$ and its dual.

\begin{proposition}
Let $[b_1,\ldots,b_s]$ be a HJ continued fraction with $b_i\geq 2$ for all $i$. Assume that there is $i$ such that $[b_1,\ldots, b_i-1, \ldots,b_s]=0$. Then this $i$ is unique and $[b_1,\ldots,b_s]$ is the HJ continued fraction of the dual of a Wahl chain.
\label{onepos}
\end{proposition}

\begin{proof}
We note that $b_i=2$. Therefore $[b_1,\ldots,b_s]=[x_1,\ldots,x_{i-1},2,y_{s-i},\ldots,y_1]$ and so it is the dual of a Wahl chain by Proposition \ref{basico}. 
If there is another index $j$ such that $[b_1,\ldots, b_j-1, \ldots,b_s]=0$, then $[b_1,\ldots,b_s]=[x'_1,\ldots,x'_{j-1},2,y'_{s-j},\ldots,y'_1]$, where there exist some coprime integers $0<a'<n'$ such that $\frac{n'}{a'}=[x'_1,\ldots,x'_{j-1}]$ and $\frac{n'}{n-a'}=[y'_1,\ldots,y'_{s-j}]$.
By Proposition, \ref{basico} we obtain $\frac{n'^2}{n'a'-1}= \frac{n^2}{na-1}$. Since gcd$(n,a)=1$ and gcd$(n',a')=1$, it follows that $n=n'$ and $a=a'$. Therefore, the index $i$ is unique.
\end{proof}

\begin{remark}
Let $[b_1,\ldots,b_s]$ be a HJ continued fraction with $b_i\geq 2$ for all $i$. Let us assume we have a pair of indices $i<j$ such that $[b_1,\ldots, b_i-1, \ldots, b_j-1, \ldots,b_s]=0$. These HJ continued fractions are precisely the ones associated with extremal P-resolutions. They were studied in \cite[Section 4]{HTU17}. As we will see, they are important to study the birational geometry involved in this article. For now, we can say that these particular HJ continued fractions may admit at most two pairs $i_k < j_k$ $k=1,2$ of indices so that $[b_1,\ldots, b_{i_k}-1, \ldots, b_{j_k}-1, \ldots,b_s]=0$ \cite[Theorem 4.3]{HTU17}, and when that happens we have the wormhole cyclic quotient singularities studied in \cite{UV22}. A classification of these wormhole singularities is not known.   
\end{remark}

In Proposition \ref{basico} we have HJ continued fraction of $\frac{n^2}{n^2-na+1}$ (dual to $\frac{n^2}{na-1}$). They are also generated by the steps (ii), and (iii) above, but starting with $[2,2,2]$. We call them \textit{dual Wahl chains}. In what follows we will use the notation $\Wa$ to represent a Wahl chain (or its sequence of integers), and $\Wa^{\vee}$ for dual Wahl chains. For example, the HJ continued fraction $[2,5,3,7,2,2,3,2,2,4]$ has the form $[\Wa_1,7,\Wa_2^{\vee}]$ where $\Wa_1=[2,5,3]$ and $\Wa_2^{\vee}=[2,2,3,2,2,4]$. In this case, we may also say that $\Wa_2=[4,5,2,2]$ as it is dual to $[2,2,3,2,2,4]$.   

\section{Markov numbers and HJ continued fractions} \label{s2}

As in the introduction, \textit{Markov triples} are the positive integer solutions $(a,b,c)$ of the Markov equation $$ x^2+y^2+z^2=3xyz.$$ The integers that appear in Markov triples are called \textit{Markov numbers}. Any permutation of a solution $(a,b,c)$ is a solution. A \textit{mutation} of a Markov triple $(a,b,c)$ is $(a,b,3ab-c)$, which is again a solution. In fact, any Markov triple can be obtained from $(1,1,1)$ by applying finitely many permutations and mutations. Solutions can be seen as vertices of an infinite connected tree, where edges represent mutations. The triples $(1,1,1)$ and $(1,1,2)$ are the only solutions with repeated Markov numbers, and so we will typically order Markov triples as $(a<b<c)$.


\begin{remark}
In the study of Markov triples, it is common to construct a function between the vertices of the Farey tree, which are in bijection with the rationals in $[0,1]$, and the Markov tree (see the Farey table in \cite[3.2]{Aig13}). This tree has vertices $(\frac{a}{b},\frac{a+a'}{b+b'},\frac{a'}{b'})$, and it is constructed by levels via the operation $\frac{x}{y} \oplus \frac{w}{z}= \frac{x+w}{y+z}$ on consecutive entries $\frac{x}{y},\frac{w}{z}$. The function sends the middle entry $t$ in the vertex to the corresponding Markov number $m_t$ in the same position. On the other hand, one can define the \textit{Wahl tree} using the Wahl chain algorithm in the previous section, starting with the vertex $[4]$. At each level, we have all the Wahl chains of a given length. A Wahl chain $[e_1,\ldots,e_r]=\frac{m^2}{mq-1}$ depends on coprime integers $0<q<m$, and $\frac{m^2}{m(m-q)-1}=[e_r,\ldots,e_1]$. When we increase its length by one through the algorithm, we obtain two Wahl chains, one for $0<q+m-q=m<m+m-q=2m-q$ and the other for $0<q<m+q$. If we think of each Wahl vertex as a pair $(\frac{q}{m},\frac{m-q}{m})$, then we obtain an obvious correspondence with the Farey tree via sending  $(\frac{q}{m},\frac{m-q}{m})$ to the vertices $\frac{q}{m}$ and $\frac{m-q}{m}$. Hence each Wahl chain has two associated Markov numbers, which are ``opposite" in the Markov tree.
\label{trees}
\end{remark}

\begin{definition}
Given a Markov triple $(a<b<c)$ and $x\in\{a,b,c\}$, we define integers $0<r_x,w_x<x$ as follows:
\begin{itemize}
    \item $r_a\equiv b^{-1}c \ (\text{mod} \ a)$, $r_b\equiv c^{-1}a \ (\text{mod} \ b)$, and $r_c\equiv a^{-1}b \ (\text{mod} \ c)$.
    \item $w_a\equiv 3b^{-1}c \ (\text{mod} \ a)$, $w_b\equiv 3c^{-1}a \ (\text{mod} \ b)$, and $w_c\equiv 3a^{-1}b \ (\text{mod} \ c)$.
\end{itemize}
\label{r,w}
\end{definition}

Markov conjecture essentially says that the $r_x$ depends only on the Markov number $x$ and not on a Markov triple that contains $x$. In \cite[Section 3.3]{Aig13}, the numbers $r_x$ are called characteristic numbers. In \cite{Per22} they are called weights, and the $w_x$ are called T-weights. Let us summarize basic properties of Markov numbers and their weights (see e.g. \cite{Aig13}, \cite{Rud88}, \cite[Cor.5.4]{Per22}).  

\begin{proposition}
Let $x>1$ be part of a Markov triple $(a<b<c)$, then $x+w_x=3r_x$, $r_x^2 \equiv -1 (\text{mod} \ x)$, $r_ca-r_a c=b$, $cr_b-br_c=a$, and $ar_b-br_a=3ab-c$. 
\label{prop}
\end{proposition}

T-weights will be soon important in the Wahl chains of the birational picture for Markov numbers. Weights produce particular HJ continued fractions for the $x/r_x$ in a Markov triple $(a,b,c)$. The reason is $r_x^2 \equiv -1 (\text{mod} \ x)$. Let us briefly describe that. 

\begin{definition}
A HJ continued fraction is a \textit{Wahl-2 chains} if it corresponds to $\frac{m}{r}$ where $0<r<m$ be integers such that $r^2 \equiv -1 (\text{mod} \ m)$.
\label{Wahl2Chain}  
\end{definition}

Wahl-2 chains obey the Wahl chain rule of formation but starting with $[2]$. If $\frac{m}{r}=[x_1,\ldots,x_p]$ is a Wahl-2 chain, then its dual is $[x_p,\ldots,x_1]$, because $r(m-r)\equiv 1 (\text{mod} \ m)$, and so $[x_1,\ldots,x_p,1,x_1,\ldots,x_p]=0$. Let $r^2+1=:f(r,m)m$, and so we have a triple $(m,r,f(r,m))$. Then 

\vspace{0.1cm}
$[x_1,\ldots,x_p] \mapsto [x_1+1,x_2,\ldots,x_p,2]$ gives $(m,r,f) \mapsto (m+2r+f,r+f,f)$, and

$[x_1,\ldots,x_p] \mapsto [2,x_1,x_2,\ldots,x_p+1]$ gives $(m,r,f) \mapsto (4m-4r+f,2m-r,m).$
\vspace{0.1cm}

For example: $[2] \mapsto [3,2] \mapsto [2,3,3] \mapsto [3,3,3,2] \ldots$ gives the Wahl-2 chains for the Markov triples $(1<b<c)$.

\bigskip 

The Wahl-2 chains of Markov triples have a special form which characterizes them. 

\begin{proposition}
    Let $\frac{a}{r_a}, \frac{b}{r_b}, \frac{c}{r_c}$ be the fractions of Wahl-2 chains. Then $\frac{c}{r_c}=[\frac{a}{r_a},4,\frac{b}{r_b}]$ if and only if $(a,b<c)$ is a Markov triple. 
    \label{markovrc}
\end{proposition}

\begin{proof}
First, we revisit the well-known fact outlined in \cite[Section 2]{UV22} that for a HJ continued fraction $\frac{n}{p}=[e_1,...,e_r]$, the following equality holds:

\[
\left(\begin{array}{cc}
e_1 & -1 \\
1 & 0
\end{array}\right) \cdots\left(\begin{array}{cc}
e_r & -1 \\
1 & 0
\end{array}\right)=\left(\begin{array}{cc}
n & -p^{-1} \\
p & \frac{1-p p^{-1}}{n}
\end{array}\right),
\]
where $0<p^{-1}<n$ denotes the inverse modulo $n$ of $p$. Specifically, for a Wahl-2 chain $\frac{m}{r}$, the matrix on the right side simplifies to $\left(\begin{array}{cc}
m & r-m \\
r & f(r,m)-r
\end{array}\right)$. 

Now let us suppose that  $\frac{c}{r_c}=[\frac{a}{r_a},4,\frac{b}{r_b}]$ where $\frac{a}{r_a}, \frac{b}{r_b}, \frac{c}{r_c}$ are fractions of Wahl-2 chains. We denote by $f_a,f_b,f_c$ the respective $f$ numbers. From the previous assertion, we derive the matrix equation:

\[
\left(\begin{array}{cc}
a & r_a-a \\
r_a & f_a-r_a
\end{array}\right)\left(\begin{array}{cc}
4 & -1 \\
1 & 0
\end{array}\right)\left(\begin{array}{cc}
b & r_b-b \\
r_b & f_b-r_b
\end{array}\right)=\left(\begin{array}{cc}
c & r_c-c \\
r_c & f_c-r_c
\end{array}\right),
\] 
This leads us to the following equations:
\begin{align}
    c &= 3ab+br_a-ar_b\\
    r_c &=3br_a+bf_a-r_ar_b \\
    r_c-c &= 3ar_b-3ab+r_ar_b-br_a-af_b+ar_b.
\end{align}

By subtracting equation (3.1) from equation (3.2) and comparing it to equation (3.3), we obtain the relation $3(ar_b-ar_b)=bf_a+af_b-2r_ar_b$. Multiplying both sides by $ab$ gives us $3ab(ar_b-br_a)=(ar_b-br_a)^2+a^2+b^2$. This simplifies to $$c(ar_b-br_a)=(3c+br_a-ar_b)(ar_b-br_a)=a^2+b^2.$$ Multiplying equation (3.1) by $c$ yields $a^2+b^2+c^2=3abc.$

\vspace{0.1cm}
Conversely, suppose $(a<b<c)$ is a Markov triple. By Proposition \ref{prop}, the equations $b=r_ca-r_a c$, $a=cr_b-br_c$, $c=3ab+br_ar-ar_b$ hold. To prove the assertion it suffices to show that equation (3.2) holds.  This is automatic, since $a(3br_a+bf_a-r_ar_b)=r_ac+b$.
\end{proof}

\begin{corollary}
Let $(a,b<c)$ be a Markov triple, and let $\frac{a}{r_a}=[x_1,\ldots,x_p]$ and $\frac{b}{r_b}=[y_1,\ldots,y_q]$. Then $$\left[\frac{b}{r_b},1,\frac{c}{r_c},1,\frac{a}{r_a} \right]=0,$$ and $\frac{c}{r_c}=\left[y_1,\ldots,y_q+1,2,x_1+1,\ldots,x_p \right]=[x_1,\ldots,x_p,4,y_1,\ldots,y_q]$.  
\end{corollary}

\begin{proof}
By Proposition \ref{markovrc} we have $\frac{c}{r_c}=[\frac{a}{r_a},4,\frac{b}{r_b}]$, and so, as the $\frac{a}{r_a},\frac{b}{r_b}, \frac{c}{r_c}$ are Wahl-2 chains, we have $$[\frac{a}{r_a},4,\frac{b}{r_b},1,y_1,\ldots,y_q+1,2,x_1+1,\ldots,x_p]=0.$$ From that we obtain $\left[\frac{b}{r_b},1,\frac{c}{r_c},1,\frac{a}{r_a} \right]=0$.
\end{proof}

\begin{remark}
Proposition \ref{markovrc} has a geometric meaning, which will be clarified in the coming sections. Given a Markov triple $(a<b<c)$, the equation $$c=3ab+br_a-ar_b$$ defines an extremal P-resolution with a $(-4)$-curve $\Gamma^+$ in the minimal resolution, and $\delta=c$. The corresponding dual HJ continued fraction of the cyclic quotient singularity $\frac{1}{\Delta'}(1,\Omega')$ is $$\frac{\Delta'}{\Delta'-\Omega'}=\left[\frac{a}{a-r_a},2,\frac{c}{c-r_c},2,\frac{b}{b-r_b} \right],$$ where the $2$s are the positions where we subtract $1$ to get the extremal P-resolution. If we apply the triangulation in Remark \ref{catalan}, then $v_0=2$. Thus this is a beautiful ``circular continued fraction".  
\label{P-reswithc}
\end{remark}

\vspace{0.2cm}

Let us express the Wahl-2 chains for Markov triples $(a<b<c)$ after mutations $(a<b<c) \mapsto (a<c<c^\prime=3ac-b)$ and $(a<b<c)\mapsto (b<c<c^\prime=3bc-a)$. The characteristic numbers change as

\begin{itemize}
    \item  $(a<b<c)\mapsto (a<c<c^\prime=3ac-b)$. If $(a<c<c^\prime)$ has characteristic numbers $(r_a,r_c,r_{c^\prime})$, then we get the same $r_a$ and $r_c$. From Proposition \ref{prop}, it follows that $c=3ab+br_a-ar_b=a(3ar_c-r_b)-c^\prime r_a$. By comparing to $c=ar_{c^\prime}-c^\prime r_a$, it follows that $r_{c^\prime}=3ar_c-r_b$.
    
    \item $(a<b<c)\mapsto (b<c<c^\prime=3bc-a)$. If $(b<c<c^\prime)$ has characteristic numbers $(r_b,r_c,r_{c^\prime})$, then we get $b-r_b$ for the new $r_b$, and $c-r_c$ for the new $r_c$. Similarly to the previous case, we obtain $c=b(r_a-3br_c)+c^\prime r_b$ and $c=b(r_{c^\prime}-c^\prime)+c^\prime r_b$. It follows that $c^\prime-r_{c^\prime}=3br_c-r_a$.
\end{itemize}

The mutations produce the following Wahl-2 chains:

$\bullet$ For $(1<a<b<c)\mapsto (a<c<c^\prime=3ac-b)$,

$$\frac{c^\prime}{r_{c^\prime}}=\left[\frac{a}{r_a},4,\frac{c}{r_c}\right]=\left[\frac{a}{r_a},4,\frac{a}{r_a},4,\frac{b}{r_b}\right].$$

$\bullet$ For $(1<a<b<c)\mapsto (b<c<c^\prime=3bc-a)$,

$$\frac{c^\prime}{r_{c^\prime}}=\left[\frac{b}{b-r_b},4,\frac{c}{c-r_c}\right]=\left[\frac{b}{b-r_b},4,\frac{b}{b-r_b},4,\frac{a}{a-r_a}\right].$$

From this, we establish a one-to-one correspondence between the Cohn words and the Wahl-2 chains $\frac{c}{r_c}$.

\vspace{0.2cm}
Let $(1<a<b)$ be a Markov triple and take the mutation $(a<b<c=3ab-1)$ and characteristic numbers $(r_a,r_b,r_c)$. We observe that $\frac{b}{b-r_b}=[3,\frac{a}{a-r_a}]$ and consequently that
$$\frac{c}{r_c}=\left[\frac{a}{r_a},4,\frac{a}{r_a},3\right].$$

\vspace{0.1cm}
Let us define $A:=[\frac{a}{r_a},4]$ and $B:=[\frac{a}{r_a},3]$ and define the product $AB$ as the concatenation of HJ continued fractions. Using the previous computations we derive that

\bigskip

\begin{itemize}
    \item  $(a<b<c)\mapsto (a<c<c^\prime=3ac-b)$, $\frac{c^\prime}{r_{c^\prime}}=\left[\frac{a}{r_a},4,\frac{a}{r_a},4,\frac{a}{r_a},3\right]=A^2B$.

    \item $(a<b<c)\mapsto (b<c<c^\prime=3bc-a)$, $\frac{c^\prime}{c^\prime-r_{c^\prime}}=\left[\frac{a}{r_a},4,\frac{a}{r_a},3,\frac{a}{r_a},3\right]=AB^2$.
\end{itemize}

Therefore, the triple $(\frac{a}{r_a},\frac{c^{\prime}}{r_{c^\prime}},\frac{c}{r_c})$ determines $(A,A^2B,AB)$ and $(\frac{c}{r_c},\frac{c^{\prime}}{c^\prime-r_{c^\prime}},\frac{b}{r_b})$ does it for $(AB,AB^2,B)$. Note that the inversion of the continued fraction is reflected through the reverse order of the triple $(\frac{b}{b-r_b},\frac{c^{\prime}}{r_{c^\prime}},\frac{c}{c-r_c})$. 
Now, by making this process inductively, we construct a binary tree which indeed agrees to the tree of words generated by a pair of Cohn matrices $(R,M)$ in a Cohn triple $(R,RM,M)$ as in \cite[Theorem 6.12]{Aig13}.

\bigskip
From a given Markov triple $(a<b<c)$, we observe that the inductive application of mutations of the form $(a,m_i,m_{i+1})\mapsto (a,m_{i+1},m_{i+2})$ gives rise to the linear recurrence $m_{i+2}=3am_{i+1}-m_i$, where $m_0=b$ and $m_1=c$. Analogously, if $(r_a,r_i,r_{i+1})$ are the corresponding weights for $i\geq 0$, then we have the recurrence $r_{i+2}=3ar_{i+1}-r_i$, where $r_0=r_b$ and $r_1=r_c$. As both sequences of numbers $(m_i)_i, (r_i)_i$ reflect similar growth, in the following elementary proposition we proceed to compute the limit of $(\frac{r_i}{m_i})_i$.


\begin{proposition}
Given a Markov triple $(a<b=m_0<c=m_1)$ with characteristic numbers $(r_a<r_b=r_0<r_c=r_1)$, the sequence defined above satisfies
\[
\lim_{i\to \infty}\frac{r_i}{m_i}=\frac{\phi r_1-r_0}{\phi m_1-m_0},
\]
where $\phi=\frac{3a+\sqrt{9a^2-4}}{2}$.
\end{proposition}

\begin{proof}
Both recurrences share the characteristic polynomial $x^2-3ax+1$, which has solutions $x=\frac{3a\pm \sqrt{9a^2-4}}{2}$. We denote these solutions as $\phi_{\pm}$, where $\phi_+=\phi$. Consequently, there exist real numbers $\alpha_{\pm},\beta_{\pm}$ such that $r_i=\alpha_+ \phi_+^{i-1}+\alpha_{-}\phi_{-}^{i-1}$ and $m_i=\beta_+ \phi_+^{i-1}+\beta_{-}\phi_{-}^{i-1}$ for $i\geq 0$. Upon computing the limit, we find that $\lim_{i\to \infty}\frac{r_i}{m_i}=\frac{\alpha_+}{\beta_+}$. From the initial conditions, we derive that $\alpha_+=\frac{r_2-r_1\phi_{-}}{\phi_{+}-\phi_{-}}=\frac{r_1\phi_{+}-r_0}{\phi_{+}-\phi_{-}}$ and $\beta_{+}=\frac{m_2-m_1\phi_{-}}{\phi_{+}-\phi_{-}}=\frac{m_1\phi_{+}-m_0}{\phi_{+}-\phi_{-}}$. The proposition follows by substituting $\phi_+=\phi$.
\end{proof}

Any Markov triple $(a,b,c)$ produces two branches of solutions in the Markov tree, except for the triples $(1,1,1)$ and $(1,1,2)$, which produce only one branch. They are defined as the set of $(c<m_k<m_{k+1})$ with $k\geq 0$:
\bigskip

\begin{itemize}
\item The triple $(1,1,1)$ defines the \textit{Fibonacci branch}:
{\scriptsize  $$(1,1,1)-(1,1,2)-(1,m_0=2,m_1=5)-(1,5,13)-\ldots-(1<m_{k}<m_{k+1})-\ldots,$$} where $m_{k+1}=3m_{k}-m_{k-1}$ for all $k\geq 0$. These $m_k$ are the Fibonacci numbers in odd positions.

\item The triple $(1,1,2)$ defines the \textit{Pell branch}:  
{\scriptsize  $$(1,1,2)-(1,2,5)-(2,m_0=5,m_1=29)-(2,29,169)-\ldots-(2<m_{k}<m_{k+1})-\ldots,$$} where $m_{k+1}=6m_{k}-m_{k-1}$ for all $k\geq 0$. These $m_k$ are the Pell numbers in odd positions.

\item Given a Markov triple $(a<b<c)$, we define the branches:

{\scriptsize  $$ (a<b<c)-(a<c<3ac-b)-(c<m_0<m_1)-(c<m_1<m_2)-\ldots-(c<m_{k}<m_{k+1})-\ldots$$} where $m_0=3ac-b$, $m_1=3(3ac-b)c-a$, and $m_{k+1}=3m_{k-1}c-m_k$ for $k \geq 1$. 

{\scriptsize  $$(a<b<c)-(b<c<3bc-a)-(c<m_0<m_1)-(c<m_1<m_2)-\ldots-(c<m_{k}<m_{k+1})-\ldots$$} where $m_0=3bc-a$, $m_1=3(3bc-a)c-b$, and $m_{k+1}=3m_{k-1}c-m_k$ for $k \geq 1$. 
\end{itemize}

For $(a<b<c)$, we have that the $r_c$ of $(a<b<c)$ changes as $r_c \mapsto r_c \mapsto c-r_c \mapsto c-r_c \mapsto \ldots$ in the first branch, and $r_c \mapsto c-r_c \mapsto r_c \mapsto r_c \mapsto \ldots $ in the second branch. Same for $w_c$ of course.

\begin{conjecture} [Markov uniqueness conjecture \cite{Aig13}] If $(a',b'<c)$ and $(a,b<c)$ are Markov triples, then $a'=a$ and $b'=b$. In other words, the greatest number in a Markov triple determines the triple.     
\end{conjecture}

In particular, if the conjecture is true, then the associated numbers $r_c$ and $w_c$ depend only on $c$. A baby case of this conjecture, which is useful for geometric interpretations, is the next theorem \cite[Proposition 3.15]{Aig13} (see \cite[Proposition 6.2]{H13} for the geometric version).

\begin{theorem}
Let us consider two Markov triples $(a',b'<c)$ and $(a,b<c)$ with the corresponding $r_c$ and $r'_c$. If $r_c=r'_c$, then $a'=a$ and $b'=b$.
\label{markitos}
\end{theorem}

\section{Some characterizations of Markov numbers} \label{s3}

Let $0<q<m$ be coprime integers. Consider the HJ continued fractions $$\frac{m}{q}=[x_1,\ldots,x_r] \ \ \text{and} \ \ \frac{m}{m-q}=[y_1,\ldots,y_s].$$ We have the Wahl chain $\frac{m^2}{mq-1}=[x_1,\ldots,x_{r}+y_{s},\ldots,y_1]$ by Proposition \ref{basico} (iii), and so $$\sum_{i=1}^r x_i + \sum_{j=1}^s y_j=3r+3s-2.$$

\begin{proposition}
We have $$\frac{m^2}{mq-1}=[\Wa_0^{\vee},\alpha,\Wa_1^{\vee}]$$ for some Wahl chains $\Wa_i$ and some $\alpha \geq 2$ if and only if $m$ is a Markov number. In fact, if $\frac{n_0^2}{n_0a_0-1}=\Wa_0$ and $\frac{n_1^2}{n_1 a_1-1}=\Wa_1$, then $$n_0^2+n_1^2+m^2=3 n_0n_1m.$$ We have $3$ possibilities for $\alpha$:
\begin{itemize}
    \item If both $\Wa_i$ are nonempty, then $\alpha=10$.
    \item If only one $\Wa_i$ is empty, then $\alpha=7$.
    \item If both $\Wa_i$ are empty, then $\alpha=4=m^2$.
\end{itemize}
\label{combi1}
\end{proposition}

\begin{proof}
Let us suppose the fraction decomposes as $\frac{m^2}{mq-1}=[\Wa_0^{\vee},\alpha,\Wa_1^{\vee}]$. From the relation mentioned above, if we denote by $S_{\Wa}$ the sum of the entries of a Wahl chain $\Wa$, and $l_{\Wa}$ to its length, then $S_{\Wa}=3l_{\Wa}+1$. Under the same reasoning, we get that $S_{\Wa^\vee}=3l_{\Wa^\vee}-3$ for duals. Let $\Wa = \frac{m^2}{mq-1}$. Then $l_{\Wa}=1+l_{\Wa_0^\vee}+l_{\Wa_1^\vee}$ and $S_{\Wa}=\alpha+S_{\Wa_0^\vee}+S_{\Wa_1^\vee}$. In this way, we compute $\alpha=10,7,3$ as in the statement. 

For convenience, we take $w_i:=n_i-a_i$. We note that $\Wa_i^{\vee}=\frac{n_i^2}{n_iw_i+1}$. By the characterization of HJ continued fractions as a product of matrices (as in Proposition \ref{markovrc}), we obtain

$$\left(\begin{array}{cc}
n_0^2 & -n_0(n_0-w_0)-1 \\
n_0w_0+1 & -w_0(n_0-w_0)-1
\end{array}\right)\left(\begin{array}{cc}
10 & -1 \\
1 & 0
\end{array}\right)\left(\begin{array}{cc}
n_1^2 & -n_1(n_1-w_1)-1 \\
n_1w_1+1 & -w_1(n_1-w_1)-1
\end{array}\right).$$ 
 
$$=\left(\begin{array}{cc}
m^2 & -m(m-q)+1 \\
mq-1 & -q(m-q)+1
\end{array}\right)$$
This gives us the following relations
\begin{equation}
    m^2=-n_0^2-n_1^2+9n_0^2n_1^2+n_0n_1^2w_0-n_0^2n_1w_1
\end{equation}
\begin{multline}
    m^2-2=(mq-1)-(1-m(m-q)) =8n_0^2+8n_1^2+9n_0^2n_1^2+ \\
    10n_0n_1^2w_0-10n_0^2n_1w_1+n_0^2w_0^2+n_1^2w_0^2-2n_0n_1w_0w_1-2
\end{multline}
By substituting the term $8(n_0^2+n_1^2)$ from (4.1) into (4.2), this lead us to
\begin{multline*}
    9m^2=81n_0^2n_1^2+18n_0n_1^2w_0-18n_0^2n_1w_1\\
    +n_0^2w_0^2+n_1^2w_0^2-2n_0n_1w_0w_1=(9n_0n_1+n_1w_0-n_0w_1)^2
\end{multline*}
and consequently to $3m=9n_0n_1+n_1w_0-n_0w_1$, since $n_iw_j<n_in_j$. From the latter equation and (4.1), it follows that
\begin{equation*}
    n_0^2+n_1^2+m^2=n_0n_1(9n_0n_1+n_1w_0-n_0w_1)=3n_0n_1m.
\end{equation*}

\vspace{0.1cm}
Assume that $m$ is a Markov number, and so it is in some Markov triple $(n_0<n_1<m)$. As \cite[Proposition 4.1]{HP10} asserts that $\mathbb{P}(n_0^2,n_1^2,m^2)$ admits a \textit{$\Q$-Gorenstein smoothing} to $\mathbb{P}^2$. Then \cite[Theorem 17]{Ma91} and \cite[Theorem 18]{Ma91} imply that $\frac{m^2}{m w_m-1}=[\Wa_0^{\vee},10,\Wa_1^{\vee}]$ when $1<n_0<n_1<m$, where $\Wa_i=\frac{n_i^2}{n_i(n_i-w_i)-1}$ and the $(w_0,w_1,w_m)$ are the respective T-weights. If $n_0=1$ we obtain the other cases easily.
\end{proof}

In this way, a Markov triple $a<b<c$ is the same as the data $$\frac{c^2}{cw_c-1}=\left[\frac{a^2}{aw_a+1},10,\frac{b^2}{bw_b+1} \right] \ \ \ \ \text{or} \ \ \ \ \frac{c^2}{cw_c-1}=\left[7,\frac{b^2}{bw_b+1} \right],$$ the last one being the case $a=1$. Theorem \ref{markitos} says that if one fixes the left side of any of these two equations, then the right side is determined. 

We now express the HJ continued fractions after the mutations $(a<b<c)\mapsto (a<c<c^\prime=3ac-b)$ and $(a<b<c)\mapsto (b<c<c^\prime=3bc-a)$. First, we easily see how the T-weights change:

\begin{itemize}
    \item  $(a<b<c)\mapsto (a<c<c^\prime=3ac-b)$.  If $(a<c<c^\prime)$ has T-weights  $(w_a<w_c<w_{c^\prime})$, then we get $w_a=w_a$ and $w_c=w_c$.

    \item $(a<b<c)\mapsto (b<c<c^\prime=3bc-a)$. If $(b<c<c^\prime)$ has T-weights $(w_b,w_c,w_{c^\prime})$, then we get $w_b=b-w_b$ and $w_c=c-w_c$.
\end{itemize}

In terms of HJ continuous fractions, the mutations are:
\bigskip

$\bullet$ For $(1<a<b<c)\to (a<c<c^\prime=3ac-b)$ we have

$$\frac{(c^\prime)^2}{c^\prime w_{c^\prime}-1}=\left[\frac{a^2} {aw_a+1},10,\frac{c^2}{cw_c+1}\right]$$

$$=\left[\frac{a^2} {aw_a+1},10,\Big(\frac{b^2}{bw_b-1} \Big)_{+1},2,2,2,2,2,2,2,\leftindex_{1+}{\Big(\frac{a^2}{aw_a-1}\Big)}\right],$$
where the $+1$ and $1+$ correspond to adding one in the last and initial position respectively.

\bigskip
$\bullet$ For $(1<a<b<c)\to (b<c<c^\prime=3bc-a)$ we have

$$\frac{(c^\prime)^2}{c^\prime w_{c^\prime}-1}=\left[\frac{b^2} {b(b-w_b)+1},10,\frac{c^2}{c(c-w_c)+1}\right]$$

$$=\left[\frac{b^2} {b(b-w_b)+1},10,\Big( \frac{a^2}{a(a-w_a)-1} \Big)_{+1},2,2,2,2,2,2,2, \leftindex_{1+} {\Big( \frac{b^2}{b(b-w_b)-1} \Big)} \right].$$

\bigskip
$\bullet$ For $(1<b<c)\to (1<c<c^\prime=3c-b)$ we have
$$\frac{(c^\prime)^2}{c^\prime w_{c^\prime}-1}=\left[7,\frac{c^2}{cw_c+1} \right]=\left[7,\Big(\frac{b^2}{bw_b-1}\Big)_{+1},2,2,2,2,2\right].$$

\bigskip
$\bullet$ For $(1<b<c)\to (b<c<c^\prime=3bc-1)$ we have

$$\frac{(c^\prime)^2}{c^\prime w_{c^\prime}-1}=\left[\frac{b^2} {b(b-w_b)+1},10,\frac{c^2}{c(c-w_c)+1}\right] $$ $$=\left[\frac{b^2} {b(b-w_b)+1},10,
2,2,2,2,2,\leftindex_{1+}{\Big(\frac{b^2}{b(b-w_b)-1}\Big)}\right].$$

\bigskip 
Next, we go for a situation related to birational geometry. Let us consider $$\frac{m^2}{m q-1} =[e_1,\ldots,e_{\beta}]=[x_1,\ldots,x_r+y_s,\ldots,y_1],$$ where $\beta:=r+s-1$. Assume that $\alpha=10$. We recall that by Proposition \ref{onepos}, we have that a dual Wahl chain has exactly one index where we subtract $1$ to obtain a zero continued fraction. Then Proposition \ref{combi1} gives the following characterization. We say that a sequence of numbers blows down to another sequence of numbers, if that happens by applying the arithmetical blowing-down several times. For example $\{4,1,2,2,2,10,2,2,2,2,2,5,1,2,2,2,7 \}$ blows-down to $\{0,10,0\}$. Notice that it also blows-down to $\{0,8,0\}$. 

\begin{proposition}
There are $i<j$ such that $$\{ e_1,\ldots,e_i-1,\ldots,e_j-1,\ldots,e_{\beta} \}$$ blows-down to $\{ 0,10,0 \}$ if and only if $$[2,2,2,e_1+1,e_2,\ldots,e_{i-1},e_i-1,e_{i+1},\ldots,e_{j-1},e_j-1,e_{j+1},\ldots,e_{\beta-1},e_{\beta}+1,2,2,2]=0.$$ If that is the case, then $\frac{a}{w_a}:=[e_1,\ldots,e_{i-1}]$, $\frac{b}{b-w_b}:=[e_{\beta},\ldots,e_{j+1}]$, and $(a,b,c)$ is a Markov triple with $a,b<c:=m$. 
\label{combi2}
\end{proposition}

\begin{proof}
First, we show that blowing-down to $\{0,10,0\}$ implies the zero continued fraction. 

As $\{ e_1,\ldots,e_i-1,\ldots,e_j-1,\ldots,e_{\beta} \}$ blows-down to $\{0,10,0\}$, we previously reach $\{1,1,10,1,1\}$. On the other hand, the sequence of numbers $$\{2,2,2,2,1,10, 1, 2, 2, 2, 2\}$$ blows-down to $0$, but using the previous remark, we see that then $$\{2,2,2,e_1+1,e_2,\ldots,e_{i-1},e_i-1,e_{i+1},\ldots,e_{j-1},e_j-1,e_{j+1},\ldots,e_{\beta-1},e_{\beta}+1,2,2,2\}$$ blows-down to $0$.

On the other hand, if $$[2,2,2,e_1+1,e_2,\ldots,e_{i-1},e_i-1,e_{i+1},\ldots,e_{j-1},e_j-1,e_{j+1},\ldots,e_{\beta-1},e_{\beta}+1,2,2,2]=0,$$ then, as in Remark \ref{catalan}, we have the corresponding triangulation on a convex polygon with $\beta+7$ vertices $P_i$. The number assigned to the vertex $P_0$ is $v_0=2$ (see the definition of $v_i$ in Remark \ref{catalan}). Therefore $$v_0=v_1=v_2=v_3=v_{\beta+1}=v_{\beta+2}=v_{\beta+3}=2.$$
Then, there is a vertex $P_k$ such that $v_k=8+x+y$, and there are two new polygons with vertices $P_4,\ldots,P_k$ and $P_k,\ldots,P_{\beta-4}$ such that $v_k=x$ for the first and $v_k=y$ for the second. But each of these polygons must have a $1$ at $P_k$, since for the rest of the entries we have only one $1$. Therefore $x=y=1$. Thus we have the equivalence.  

Thus we have that $\{ e_1,\ldots,e_i-1,\ldots,e_j-1,\ldots,e_{\beta} \}$ blows-down to $\{0,10,0\}$. Now we can use Proposition \ref{onepos} to conclude that $[e_1,...,e_{\beta}]=[\frac{a^2}{aw_a+1},10,\frac{b^2}{bw_b+1}]$ for some $a,b$. In this way, by Proposition \ref{combi1}, we have that $(a,b<c:=m)$ is a Markov triple, and the formulas for $\frac{a}{w_a}$ and $\frac{b}{b-w_b}$ are deduced via Proposition \ref{basico}.
\end{proof}

This brings us to the dual version due to the birational geometry viewpoint. 

\begin{proposition}
We have $$[5,x_1,\ldots,x_r,2,y_s,\ldots,y_1,5]=[\Wa_0,2,\Wa_1]$$ if and only if there is a Markov triple $(1<a,b<c=:m)$.
\label{combi3}
\end{proposition}

\begin{proof}

First assume $[5,x_1,\ldots,x_r,2,y_s,\ldots,y_1,5]=[\Wa_0,2,\Wa_1]$, and so it admits an extremal P-resolution (Definition \ref{extremalPres}) as in \cite[Section 4]{HTU17}. Its dual HJ continued fraction is $[2,2,2,y_1+1,y_2,\ldots,y_s+x_r,\ldots,x_2,x_1+1,2,2,2]$. By \cite[\S 4.1]{HTU17}, we obtain exactly the statement in Proposition \ref{combi2}, i.e. two indices $i<j$ such that we subtract $1$ at those positions and obtain a zero continued fraction. So, we obtain a Markov triple.

On the other hand, a Markov triple $(1<a,b<c)$ defines the first part of Proposition \ref{combi2}, and so the zero continued fraction $$[2,2,2,e_1+1,e_2,\ldots,e_{i-1},e_i-1,e_{i+1},\ldots,e_{j-1},e_j-1,e_{j+1},\ldots,e_{\beta-1},e_{\beta}+1,2,2,2]=0.$$ The dual of $$[2,2,2,e_1+1,e_2,\ldots,e_{i-1},e_i,e_{i+1},\ldots,e_{j-1},e_j,e_{j+1},\ldots,e_{\beta-1},e_{\beta}+1,2,2,2]$$ is $[5,x_1,\ldots,x_r,2,y_s,\ldots,y_1,5]$, and so it admits an extremal P-resolution by \cite[\S 4.1]{HTU17}. As we are not subtracting $1$'s in the ends, we know it has two Wahl singularities. We also have the formula $\sum_i x_i + \sum_j y_j=3r+3s-2$, and so by \cite[Theorem 2.5]{UV22} we compute that the ``middle" number in the extremal P-resolution is $2$. Therefore $$[5,x_1,\ldots,x_r,2,y_s,\ldots,y_1,5]=[\Wa_0,2,\Wa_1]$$ for some Wahl chains $\Wa_i$.   
\end{proof}

This last Proposition \ref{combi3} is strongly related to the MMP that will be developed in the subsequent sections. We recall that given a Markov triple $(a<b<c)$ and $x \in \{a,b,c\}$, we have integers $r_x$ and $w_x=3r_x-x$ (see Definition \ref{r,w}). The relevant HJ continued fraction is $$\frac{\Delta_0}{\Omega_0}=[5,x_1,\ldots,x_r,2,y_s,\ldots,y_1,5],$$ and it will have a unique extremal P-resolution (Definition \ref{extremalPres}) with $\delta=3m$, $n_0=5b-w_b=3(2b-r_b)$, $a_0=b$, $n_1=4a+w_a=3(a+r_a)$, and $a_1=3a+w_a=2a+3r_a$ (Proposition \ref{prop}). We have the formulas

\bigskip 
$\Delta_0=(4c+w_c)(5c-w_c)-9=9((c+r_c)(2c-r_c)-1)=n_0^2 + n_1^2 +3c n_0 n_1$

$3c=n_0n_1+n_1 a_0-n_0a_1$

$\Omega_0=c(4c+w_c)-1=3c(c+r_c)-1=n_1^2 a_0^2 +(n_0 a_0 -1)(n_1^2-n_1 a_1 +1)$

$\Omega_0^{-1}=c(5c-w_c)-1=3c(2c-r_c)-1$, the inverse of $\Omega_0$ modulo $\Delta_0$.

\begin{remark}
Note that $\Omega_0+\Omega_0^{-1}=9c^2-2$. In this way, we have the particular classical Dedekind sum $$ 12 s(\Omega_0,\Delta_0):= 1 + \frac{9c^2-2}{(4c+w_c)(5c-w_c)-9} = \frac{29c^2+cw_c-w_c^2-11}{20c^2+cw_c-w_c^2-9}.$$ It has a minimum at $\frac{c}{2}$, whose minimum value is $\frac{117c^2-44}{81c^2-36}$. If we think of $w_c$ as variable and $c$ fixed, \textit{Is there a characterization of the values of the Dedekind sums that admits an extremal P-resolution?} Connections between Markov triples and Dedekind sums can be found in \cite[Ch. 2 \S 8]{HiZa74}.
\label{dedekind}
\end{remark}


Note that $$\frac{\Delta_0}{\Delta_0- \Omega_0}=[2,2,2,y_1+1,y_2,\ldots,y_s+x_r,\ldots,x_2,x_1+1,2,2,2].$$ The ``circular continued fraction" has a $2$ at the zero vertex. These seven $2$s connect with a $10$ which is in the chain $[y_2,\ldots,y_s+x_r,\ldots,x_2]$, and it is uniquely located by Theorem \ref{markitos}. This $10$ splits the triangulation into two triangulations, each corresponding to duals of Wahl chains. This was used to prove Proposition \ref{combi2}.

\section{Birational geometry and Mori trains} \label{s4}

We now start with geometry. We will only encounter 2-dimensional \textit{cyclic quotient singularities} (c.q.s.), and the most relevant will be \textit{Wahl singularities}. We recall that a c.q.s. $\frac{1}{m}(1,q)$ is the surface germ at $(0,0)$ of the quotient of $\C^2$ by $(x,y) \mapsto (\zeta x, \zeta^q y)$, where $\zeta$ is an $m$-th primitive root of $1$, and $0<q<m$ are coprime integers. A Wahl singularity is a c.q.s. $\frac{1}{n^2}(1,na-1)$, where $0<a<n$ are coprime integers. We include smooth points setting $n=1$.

A singularity $\frac{1}{m}(1,q)$ can be minimally resolved by a chain of nonsingular rational curves $E_1,\ldots,E_r$ where $E_i^2=-e_i \leq -2$ and $\frac{m}{q}=[e_1,\ldots,e_r]$. This last part is the direct connection to the previous sections. These singularities do not have parameters involved, so c.q.s. are the same as the HJ continued fractions of rational numbers greater than $1$. The symbol $[e_1,\ldots,e_r]$ will also refer to these chains of curves. Wahl singularities are minimally resolved by Wahl chains.

To operate with birational geometry ``on chains with singularities", we will need the following definition.

\begin{figure}[htbp]
\centering
\includegraphics[width=12cm]{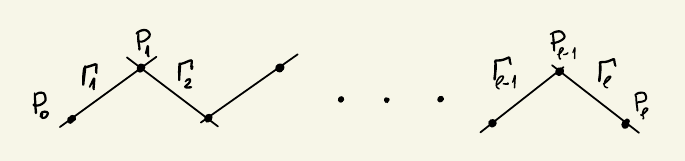}
\caption{Chain of Wahl singularities} 
\label{f1}
\end{figure}

\begin{definition}
A \textit{chain of Wahl singularities} is a collection of nonsingular rational curves $\Gamma_1,\ldots,\Gamma_{\ell}$ and a collection of Wahl singularities $P_0,\ldots,P_{\ell}$ in a surface $W$ such that $P_i, P_{i+1}$ belong to $\Gamma_{i+1}$, and $\Gamma_i, \Gamma_{i+1}$ form a toric boundary at $P_i$ for all $i$. In the case of $i=0$ or $i=\ell$, we have only one part of the toric boundary. The notation is $P_i=\frac{1}{n_i^2}(1,n_i a_i -1)$, where the minimal resolution goes from $\Gamma_i$ to $\Gamma_{i+1}$. In the minimal resolution of all singularities, the proper transforms of $\Gamma_i$ have self-intersection $-c_i$. This situation in $W$ will be denoted by $$\left[{n_0 \choose a_0}\right]-(c_1)-\left[{n_1 \choose a_1}\right]-(c_2)-\ldots -(c_{\ell})- \left[{n_{\ell} \choose a_{\ell}}\right].$$ When $P_i$ is smooth (i.e. $n_i=1$), then we write just $\ldots-(c_{i})-(c_{i+1})-\ldots$.
\label{chainWahlsing}
\end{definition}

\begin{example}
Let $(1<a<b<c)$ be a Markov triple. Consider the weighted projective plane $\P(a^2,b^2,c^2)$. It has Wahl singularities $P_0=\frac{1}{a^2}(1,a w_a-1)$, $P_1=\frac{1}{c^2}(1,c w_c-1)$, and $P_2=\frac{1}{b^2}(1,b w_b-1)$. The toric boundary of $\P(a^2,b^2,c^2)$ is given by nonsingular rational curves $\Gamma_1,\Gamma_2,\Gamma_3$. Say that $P_0,P_1 \in \Gamma_1$ and $P_1,P_2 \in \Gamma_2$. Then we have the chain of Wahl singularities $$\left[{a \choose w_a}\right]-(1)-\left[{c \choose w_c}\right]-(1)- \left[{b \choose w_b}\right].$$ In the case of $a=1<b$, we obtain $(0)-\left[{c \choose w_c}\right]-(1)- \left[{b \choose w_b}\right]$. For $a=b=1<2=c$ we have $(0)-\left[{2 \choose 1}\right]-(0)$.    
\label{manetti}
\end{example}

\begin{definition}
A \textit{W-surface} is a normal projective surface $W$ together with a proper deformation $(W \subset \W) \to (0 \in \D)$ such that
\begin{enumerate}
\item $W$ has at most Wahl singularities.
\item $\W$ is a normal complex $3$-fold with $K_{\W}$ $\Q$-Cartier.
\item The fiber $W_0$ is reduced and isomorphic to $W$.
\item The fiber $W_t$ is nonsingular for $t\neq 0$.
\end{enumerate}
We denote this by $W_t \rightsquigarrow W_0:=W$. 
\label{wsurf}
\end{definition}

We recall that here $\D$ is an arbitrarily small disk, the germ of a nonsingular point on a curve. For a  given W-surface, locally at each of the singularities of $W$ we have what is called a \textit{$\Q$-Gorenstein smoothing}. The invariants $q(W_t):=h^1(\O_{W_t})$, $p_g(W_t):=h^2(\O_{W_t})$, $K_{W_t}^2$, $\chi_{\text{top}}(W_t)$ (topological Euler characteristic) remain constant for every $t \in \D$ (see e.g. \cite[\S 1]{Ma91}). The fundamental group of $W_0$ and $W_t$ may differ. For example $W_t$ could be an Enriques surface and so $\pi_1(W_t)=\Z/2$, and $W_0$ a rational surface and so simply-connected (see e.g. \cite[\S 4]{Urz16b}).

A W-surface is \textit{minimal} if $K_W$ is nef, and so $K_{W_t}$ is nef for all $t$ \cite{Urz16a}. If a W-surface is not minimal, then we can run explicitly the MMP for $K_{\W}$ relative to $\D$, which is fully worked out in \cite{HTU17}. See \cite[\S 2]{Urz16a} for a summary of the results in \cite{HTU17}, and \cite[\S 2]{Urz16b} for details of how MMP is run. (MMP is closed for W-surfaces as shown in the proof of \cite[Theorem 5.3]{HTU17}.) It arrives at either a minimal model or a nonsingular deformation of ruled surfaces or a degeneration of $\P^2$ with only quotient singularities (see \cite[Section 2]{Urz16a}). This last outcome is very relevant in what follows, and it will be described in the next section. When $K_W$ is nef and big, the canonical model of $(W \subset \W) \to (0 \in \D)$ has only T-singularities (i.e. ADE singularities or c.q.s. of type $\frac{1}{dn^2}(1,dna-1)$ with $0<a<n$, gcd$(a,n)=1$, and $d\geq 1$). (C.f. \cite[Section 2]{Urz16a} and \cite[Sections 2 and 3]{Urz16b}.)

\begin{example}
Any $\P(a^2,b^2,c^2)$ can be the central fiber of a W-surface. In that case, the general fiber must be $\P^2$ as $K^2=9$ and $-K$ is ample. The reason is that there are no local-to-global obstructions for global deformations. In fact, in \cite[Proposition 3.1]{HP10}, it is proved that this is the case whenever $-K_W$ is big. Any partial $\Q$-Gorenstein smoothing of $\P(a^2,b^2,c^2)$ will have no local-to-global obstructions for global deformations by the same reason.  
\label{noobstruction}
\end{example}

When $K_W$ is nef, then all fibers are minimal models. If not, there exists a smooth rational curve $\Gamma$ such that $\Gamma \cdot K_W <0$. We have three possibilities: $\Gamma^2<0$, or there is no $\Gamma^2<0$, and so we have $\Gamma^2=0$ (deformations of ruled surfaces), or $\Gamma^2 >0$ (degenerations of $\P^2$). In the first case, the curve $(\Gamma \subset \W)$ is contractible to $(P \in \bbW)$ over $\D$, and this is a $3$-fold \textit{extremal neighborhood} (nbhd) of type \eni~(one singularity) or \enii~(two singularities). In birational geometry of 3-folds, the general definition of extremal neighborhoods of types \eni/\enii ~is here \cite[Section 1, C.4]{kollar1992classification}. It turns out that for our purposes we only need the ``minimal ones" as shown in \cite[Proposition 2.1]{HTU17}. Explicit computations can be found in \cite[\S 2.2 and \S 3.1]{HTU17}. 

In our specific situation and at the level of fibers, we have at each birational transformation of the MMP a contraction on special fibers $\Gamma \subset W \to P\in \bW$, where $P$ is some c.q.s. $\frac{1}{\Delta}(1,\Omega)$. All will be explained below.

\vspace{0.3cm}

\underline{($W \to \bW$ for \eni)}: Fix an \eni ~with Wahl singularity $\frac{1}{n^2}(1,na-1)$. Let $\frac{n^2}{na-1}=[e_1,\ldots,e_r]$ be its continued fraction. Let $E_1,\ldots,E_r$ be the exceptional curves of the minimal resolution $\widetilde{W}$ of $W$ with $E_j^2=-e_j$ for all $j$. Notice that $K_{W} \cdot \Gamma <0$ and $\Gamma \cdot \Gamma <0$ imply that the
strict transform of $\Gamma$ in $\widetilde{W}$ is a $(-1)$-curve intersecting only one curve $E_i$ transversally at one point. These data will be written as $$[e_1,\ldots,\overline{e_i},\ldots,e_r]$$ so that $\frac{\Delta}{\Omega} = [e_1,\ldots,e_i-1,\dots,e_r]$ where $(P \in \bW)$ is $\frac{1}{\Delta}(1,\Omega)$. 
Define $\delta:= -n K_{W} \cdot \Gamma >0$. We have $\Gamma \cdot \Gamma= -\frac{ \Delta}{n^2}<0$.

\vspace{0.4cm}

\underline{($W \to \bW$ for \enii)}: Consider now an \enii ~with Wahl singularities $$\frac{1}{n_0^2}(1,n_0 a_0 -1), \frac{1}{n_1^2}(1,n_1 a_1 -1).$$ Let $E_1,\ldots,E_{r_0}$ and $F_1,\ldots,F_{r_1}$ be the exceptional divisors over $\frac{1}{n_0^2}(1,n_0 a_0 -1)$ and $\frac{1}{n_1^2}(1,n_1 a_1-1)$ respectively, such that $\frac{n_0^2}{n_0 a_0-1}=[e_1,\ldots,e_{r_0}]$ and $\frac{n_1^2}{n_1 a_1 -1}=[f_1,\ldots,f_{r_1}]$ with $E_i^2=-e_i$ and $F_j^2=-f_j$. We know that the strict transform of $\Gamma$ in the minimal resolution $\widetilde{W}$ of $W$ is a $(-1)$-curve intersecting only $E_{r_0}$ and $F_1$ transversally at one point each. (It follows e.g. from the proof of \cite[Theorem 10.6]{kawamata1988crepant} at p.154.) The data for \enii ~will be written as $$[e_{1},\ldots,e_{r_0}]-[f_1,\ldots,f_{r_1}],$$ and $$\frac{\Delta}{\Omega} = [e_{1},\ldots,e_{r_0},1,f_1,\ldots,f_{r_1}]$$ where $(P \in \bW)$ is
$\frac{1}{\Delta}(1,\Omega)$.

We define $\delta:= n_0 a_1- n_1 a_0$, and so $$\Delta= n_0^2 + n_1^2 - \delta n_0 n_1, \ \ \ \Omega= (n_0-\delta n_1)a_0+n_1 a_1 -1.$$ We have
$K_{W} \cdot \Gamma=- \frac{\delta}{n_0 n_1 } <0$ and $\Gamma \cdot \Gamma= -\frac{\Delta}{n_0^2 n_1^2} <0$, see \cite[Proposition 2.6]{M02}.

\vspace{0.4cm}

\underline{($W^+ \to \bW$)}: In analogy to an \enii, an extremal P-resolution has data (see e.g. \cite[Lemma 3.14]{KSB88}, \cite[\S 4]{HTU17}) $$[e_{1},\ldots,e_{r_0}]-c-[f_1,\ldots,f_{r_1}],$$  So that
$$\frac{\Delta}{\Omega}=[e_{1},\ldots,e_{r_0},c,f_1,\ldots,f_{r_1}]$$ where $-c$ is the self-intersection of the strict transform of $\Gamma^+$ in the
minimal resolution of $W^+$, and $(P \in \bW)$ is $\frac{1}{\Delta}(1,\Omega)$. As for an \enii~, here $\frac{{n'}_0^2}{{n'}_0
{a'}_0-1}=[e_1,\ldots,e_{r_0}]$ and $\frac{{n'}_1^2}{{n'}_1 {a'}_1-1}=[f_1,\ldots,f_{r_1}]$. If a Wahl singularity (or both) is (are) actually smooth, then we set ${n'}_0=1$, ${a'}_0=0$ and/or ${n'}_1={a'}_1=1$.

We define $$\delta= (c-1){n'}_0 {n'}_1 + {n'}_1 {a'}_0 - {n'}_0 {a'}_1,$$ and so $\Delta= {n'}_0^2 + {n'}_1^2 + \delta {n'}_0 {n'}_1$ and, when both
${n'}_i \neq 1$, $$\Omega = -{n'}_1^2 (c-1) + ({n'}_0+\delta {n'}_1){a'}_0+{n'}_1 {a'}_1 -1.$$ (One easily computes $\Omega$ when one or both ${n'}_i=1$.) We have $$K_{W^+} \cdot \Gamma^+=\frac{\delta}{{n'}_0 {n'}_1 } >0 \ \ \ \text{and} \ \ \ \Gamma^+ \cdot \Gamma^+= -\frac{\Delta}{{n'}_0^2 {n'}_1^2} <0.$$


\textit{When do we have a divisorial contraction or a flip?} The criterion for any \eni \ or \ \enii \ extremal nbhd uses the Mori recursion \cite{M02} (see \cite{Urz16a} for more details). We check this in 2 steps:

\begin{itemize}
    \item If we have a \eni \ extremal nbhd with a Wahl singularity $Q$, then there is at least one (typically there are two) \enii \ extremal nbhd which has $Q$ as one of its Wahl singularities, and they are over the same c.q.s. The \eni \ and the \enii \ are both divisorial contractions or flip \cite[\S 2.3, \S 3.4]{HTU17}. So it is enough to check it for a \enii.
    
    \item Let $(\Gamma \subset W \subset \W) \to (P \in \bW \subset \bbW)$ be a \enii \ extremal nbhd. Let $\frac{1}{n_0^2}(1,n_0 a_0 -1)$ and $\frac{1}{n_1^2}(1,n_1 a_1 -1)$ be the Wahl singularities of $W$. If $\delta=1$, then this \enii \ is of flipping type. If $\delta>1$, then we consider the Mori recursion (see \cite[\S3.3]{HTU17}) $$ n(0)=n_0, \ \ \ n(1)=n_1, \ \ \ n(i-1)+n(i+1)=\delta n(i)$$ for any $i \in \Z$. If there is $i$ such that $n(i)=0$, then we have a divisorial contraction. Otherwise, it is a flip.
\end{itemize}

Right before some $n(i)$ becomes nonpositive in the Mori recursion \cite{M02}, we obtain an initial \enii \ over the same c.q.s., $\delta$, and flipping or divisorial type as the \enii \ we started with. In \cite{Urz16a} this is called the initial \enii. Let $\frac{1}{n_0^2}(1,n_0 a_0 -1)$ and $\frac{1}{n_1^2}(1,n_1 a_1 -1)$ be the Wahl singularities corresponding to that initial \enii. Assume $\delta n_1-n_0 \leq 0$. (We note that one of them could be a smooth point.) Say that the c.q.s is $\frac{1}{\Delta}(1,\Omega)$, and we have the contraction $$ \left[{n_0 \choose a_0}\right]-(1)-\left[{n_1 \choose a_1}\right] \to \frac{1}{\Delta}(1,\Omega),$$ where the left-hand side is a chain of Wahl singularities. 

For $i \geq 1$, we have the Mori recursions $$ n(0)=n_1, \ \ \ n(1)=n_0, \ \ \ n(i-1)+n(i+1)=\delta n(i)$$ and $a(0)= a_1$, $a(1)=a_0$, $a(i-1) + a(i+1)=\delta a(i)$. When $\delta >1$, for each $i \geq 1$ we have an \enii \ with Wahl singularities defined by the pairs $$(n(i+1),a(i+1)), (n(i),a(i)).$$  We have $n(i+1) > n(i)$. The numbers $\delta$, $\Delta$, and $\Omega$, and the flipping or divisorial type are equal to the ones associated to the initial \enii. We call this sequence of \enii's a \textit{Mori sequence}. If $\delta=1$, then the initial \enii \ is flipping, and the Mori sequence above gives only one more \enii ~with data $n(2)=n_0-n_1, a(2)=a_0-a_1$ and $n(1)=n_0, a(1)=a_0$. 

We now explain the effect of doing either a divisorial contraction or a flip on an arbitrary \eni \ or \enii $$(\Gamma \subset W \subset \W) \to (P \in \bW \subset \bbW),$$ all over the same $\D$. 

\bigskip 
\textbf{(DC):} Assume it is a divisorial contraction. Then $\delta \geq 2$, $\Delta=\delta^2$ and $\Omega=\delta a-1$, for some $a$ coprime to $\delta$. This is, the c.q.s. $P \in \bW$ is a Wahl singularity. The initial \enii \ and the contraction is $$ \left[{\delta^2 \choose \delta a-1}\right]-(1)-\left[{\delta  \choose a}\right] \to \frac{1}{\delta^2}(1,\delta a-1).$$ On the general fibers $W_t \to \bW_t$ we have the contraction of a $(-1)$-curve. In particular, after this divisorial contraction, we have a W-surface $(\bW \subset \bbW) \to (0 \in \D)$. 

\bigskip 
\textbf{(F):} Assume it is a flip. Then its flip is an extremal nbhd $$(\Gamma^+ \subset W^+ \subset \W^+) \to (P \in \bW \subset \bbW),$$ again over $\D$, such that $(\Gamma^+ \subset W^+) \to (P \in \bW)$ is an extremal P-resolution. The deformation $(\bW^+ \subset \bbW^+) \to (0 \in \D)$ is a W-surface again. Between general fibers $W_t$ and $W_t^+$ we have isomorphisms. If we write the extremal P-resolution as a contractible configuration of Wahl singularities $$\left[{n'_0 \choose a'_0}\right]-(c)-\left[{n'_1  \choose a'_1}\right] \to \frac{1}{\Delta}(1,\Omega),$$ then $n'_0=n_1$, $a'_0=a_1$, $n'_1 =n_0-\delta n_1$, and $a'_1=a_0-\delta a_1$ modulo $n'_1$ (We recall that we are assuming $n_0 \geq \delta n_1 \geq n_1$.) To compute $c$ we use the formula of $\delta$ in an extremal P-resolution (see above).

\begin{remark}
As discussed in the introduction and later in Section \ref{s5}, the MMP process we established on $\mathbb{F}_1 \rightsquigarrow \text{Bl}_{\text{pt}}(W)$, where $W=\P(a^2,b^2,c^2)$, consists solely of flips. Nevertheless, we described divisorial contractions to familiarize the reader with the MMP of $W$-surfaces in a broader context.
\end{remark}

\begin{remark}
An interesting question is: \textit{When does a c.q.s. $\frac{1}{\Delta}(1,\Omega)$ admit an extremal P-resolution?} For example, this question is directly related to Markov conjecture (see Proposition \ref{combi3}). In \cite[Section 4]{HTU17}, we show a complete answer in terms of zero continued fractions. Let us consider the dual HJ continued fraction $$\frac{\Delta}{\Delta-\Omega} = [b_1,\ldots,b_s].$$ Then $\frac{1}{\Delta}(1,\Omega)$ admits it if and only if there are $i<j$ such that $$[b_1,\ldots, b_i-1, \ldots, b_j-1, \ldots,b_s]=0.$$ One can read precisely the extremal P-resolution. \textit{How many extremal P-resolutions can a c.q.s. admit?} At most two, and $\delta$'s are the same. This is \cite[Theorems 4.3 and 4.4]{HTU17}. When a c.q.s. admits two extremal P-resolutions we call it \textit{wormhole}. The reason is here \cite{UV22}. There are various open questions on wormhole singularities, their classification is not known. 
\end{remark}

\begin{remark}
As was said in Remark \ref{catalan}, the zero continued fractions of length $s$ are in one-to-one correspondence with triangulations of convex polygons with $s+1$ sides. Hence, given $[b_1,\ldots, b_i-1, \ldots, b_j-1, \ldots,b_s]=0$ as in the previous remark, we have the number of triangles for a given vertex $v_k=b_k$ if $k \neq i,j,0$, $v_i=b_i-1$, and $v_j=b_j-1$. Thus, $v_0=3s-1 - \sum_{k=1}^s b_k$. If this number $v_0$ is equal to $1$, then we can erase the triangle with that vertex, and obtain a new zero continued fraction for a new HJ continued fraction where we subtract $1$ in two positions. After repeating this some number of times, we obtain a vertex $P_0$ with $v_0 \neq 1$. 

In this way, one can think of HJ continued fractions with two $i<j$ positions where we subtract $1$ to obtain a zero continued fraction as constructed by a particular one after adding triangles at the $0$ vertex many times. All of these new continued fractions have the same $\delta$.

Let us take an example, which will be important when describing the birational geometry of Markov triples. Let $0<q<m$ be coprime integers. Consider the Hirzebruch-Jung continued fractions $$\frac{m}{q}=[x_1,\ldots,x_r] \ \ \text{and} \ \ \frac{m}{m-q}=[y_1,\ldots,y_s].$$ Let us define the c.q.s. $\frac{1}{\Delta}(1,\Omega)$ via its dual HJ continued fraction $\frac{\Delta}{\Delta-\Omega}=$ $$[x_1,\ldots,x_r,2,y_s,\ldots,y_1+x_1,\ldots,x_r+y_s,\ldots,y_1+1,\underbrace{2,\ldots,2}_6,x_1+1,\ldots,x_r+y_s,\ldots,y_1].$$ Assume it has $i<j$ positions such that we subtract $1$ in both and we obtain a zero continued fraction. We note that $\sum_{i=1}^r x_i + \sum_{j=1}^s y_j=3r+3s-2$, and so it is easy to verify $v_0=1$ for the corresponding polygon. We now erase that triangle at the zero vertex, and keep going until the corresponding $v_0 \neq 1$. One can verify that the part that survives is precisely the underlined below $$[x_1,\ldots,x_r,2,y_s,\ldots,\underline{y_1+x_1,\ldots,y_1+1,\overbrace{2,\ldots,2}^6},x_1+1,\ldots,x_r+y_s,\ldots,y_1].$$ The new continued fraction that has $i<j$ to subtract to get a zero continued fraction is $$[x_1+1,\ldots,x_r+y_s,\ldots,y_1+1,2,2,2,2,2,2],$$ and it vertex at $0$ has $v_0=2$. Therefore we can modify it to $$\frac{\Delta_0}{\Delta_0 - \Omega_0} =[2,2,2,x_1+1,\ldots,x_r+y_s,\ldots,y_1+1,2,2,2].$$ One can show that $\delta=3m$, and the c.q.s. has HJ continued fraction $$\frac{\Delta_0}{\Omega_0} =[5,x_1,\ldots,x_r,2,y_s,\ldots,y_1,5].$$ We can name it as Markov's c.q.s. because by Proposition \ref{combi3} $m$ is a Markov number in a Markov triple $(a,b<c:=m)$. We recall that Markov conjecture says for a fixed $m$ there is only one $q$ such that $\frac{1}{\Delta_0}(1,\Omega_0)$ admits an extremal P-resolution. For any such case, the strict transform of $\Gamma^+$ in the minimal resolution of $W^+$ is a $(-2)$-curve, and $\delta=3m$.
\label{markovcqs}
\end{remark}

\begin{example} The following are examples of $\frac{\Delta_0}{\Omega_0}=[\bold{5},x_1,\ldots,x_p,\bold{2},y_q,\ldots,y_1,\bold{5}]$, where the bar below a $2$ will indicate the curve $\Gamma^+$ in the extremal P-resolution. We also show the dual continued fraction $\frac{\Delta_0}{\Delta_0 - \Omega_0}$, indicating the two places where we subtract to obtain a zero continued fraction. 

\vspace{0.1cm} 

\begin{itemize}
\item $c=29$, $w_c=22$: $[\bold{5},2,2,2,8,\bold{2},2,2,\underline{2},2,2,2,5,\bold{5}]$

    $[{21 \choose 5}]-2-[{ 9 \choose 7}]$ where $\frac{3c}{70}=[2,2,2,2,10,2]$

    dual fraction $[2,2,2,6,\bar{2},2,2,2,2,10,2,\bar{2},3,2,2,2]$
\vspace{0.3cm}   

\item $c=169$, $w_c=128$: $[\bold{5},2,2,2,10,2,2,2,2,\bold{2},6,2,2,2,\underline{2},2,2,2,5,\bold{5}]$

    $[{ 123 \choose 29 }]-2-[{ 9 \choose 7 }]$ where $\frac{3c}{70}=[8,2,2,2,10,2]$

dual fraction $[2,2,2,6,2,2,2,2,2,2,\bar{2},8,2,2,2,10,2,\bar{2},3,2,2,2]$

  \vspace{0.3cm}   

\item $c=194$, $w_c=163$: $[\bold{5},2,2,2,2,2,5,8,\bold{2},2,2,\underline{2},2,2,2,3,2,2,7,\bold{5}]$

    $[{ 54 \choose 13}]-2-[{24  \choose 19}]$ where $\frac{3c}{269}=[3,2,2,2,2,2,10,5]$

dual fraction $[2,2,2,8,2,\bar{2},3,2,2,2,2,2,10,5,\bar{2},2,2,2,3,2,2,2]$

    \vspace{0.3cm} 
    
\item $c=433$, $w_c=104$: $[\bold{5},5,2,2,2,2,2,10,2,\bold{2},3,2,2,2,\underline{2},2,2,2,8,2,2,2,\bold{5}]$ 

$[{138 \choose 29}]-2-[{ 21 \choose 16}]$ where $\frac{3c}{1120}=[2,2,2,2,2,2,5,10,2,2,2,2]$

dual fraction $[2,2,2,3,2,2,8,\bar{2},2,2,2,2,2,2,5,10,2,2,2,2,\bar{2},6,2,2,2]$

 \vspace{0.3cm} 
 
\item $c=985$, $w_c=746$: $[\bold{5},2,2,2,10,2,2,2,8,\bold{2},2,2,2,2,2,2,6,2,2,2,\underline{2},2,2,2,5,\bold{5}]$

$[{ 717 \choose 169 }]-2-[{ 9 \choose 7}]$ where $\frac{3c}{2378}=[2,2,2,2,10,2,2,2,10,2]$

dual fraction $[2,2,2,6,2,2,2,2,2,2,2,6,\bar{2},2,2,2,2,10,2,2,2,10,2,\bar{2},3,2,2,2]$

\end{itemize}
\end{example}

\begin{theorem} \cite[Theorem 1.1]{HTU17}
Let us consider any \eni \ and \enii \ $(\Gamma \subset W \subset \W) \to (P \in \bW \subset \bbW)$, and so the W-surface $(W \subset \W) \to 0 \in \D$. Then there is a universal irreducible family of surfaces such that $\W \to \bbW$ is a pull-back of it. In the case of a flip  $(\Gamma^+ \subset W^+ \subset \W^+) \to (P \in \bW \subset \bbW)$, there is also a universal irreducible family such that $\W^+ \to \bbW$ is a pull-back of that family. These families are explicitly described using toric geometry, and their two-dimensional basis depends only on $\delta$.   
\label{universal} 
\end{theorem}

The key numerical data that controls the universal family is the infinite HJ continued fraction $$\frac{\delta+ \sqrt{\delta^2-4}}{2}=\delta  - \frac{1}{\delta - \frac{1}{\ddots}}.$$ Depending on the birational type of the extremal nbhd, we encode the numerical data for each \eni \ and \ \enii \ as follows.

\begin{definition}
Let $(P \in \bW)=\frac{1}{\Delta}(1,\Omega)$ be a cyclic quotient singularity which is part of some extremal neighborhood with $\delta>1$. A \textit{Mori train} is the combinatorial data to construct all \eni \ and \ \enii \ over $\frac{1}{\Delta}(1,\Omega)$ of divisorial or flip type. We explain each case separately.
\noindent 

\begin{itemize}
\item[(DC):] Fix the Wahl singularity $\frac{1}{\Delta}(1,\Omega)=\frac{1}{\delta^2}(1,\delta a-1)$. Then its (unique) \textit{Mori train} is the concatenated data of the Wahl chains involved in all \eni \ and \enii \ of divisorial contraction type over $\frac{1}{\delta^2}(1,\delta a-1)$. The \textit{first wagon} corresponds to the Wahl chain $[e_1,\ldots,e_r]$ of $\frac{1}{\delta^2}(1,\delta a-1)$, the next wagons correspond to the Wahl chain in the k1A (and so they have one bar somewhere), and two consecutive wagons correspond to the Wahl chains in the k2A: 
$$[e_1,\ldots,e_r]-[e_{1,1},\ldots,e_{r_1,1}]-[e_{1,2},\ldots,e_{r_2,2}]-\ldots $$

\item[(F):] Fix an extremal P-resolution of $\frac{1}{\Delta}(1,\Omega)$. Its (at most two) \text{Mori trains} are the concatenate data of the Wahl chains involved in all \eni \ and \enii \ of flipping type over $\frac{1}{\Delta}(1,\Omega)$. The \textit{first wagon} corresponds to one of the Wahl chains in the extremal P-resolution, and as before, the next wagons correspond to the Wahl chain in a k1A, and two consecutive wagons correspond to the Wahl chains in a k2A. We put an empty wagon $[]$ if the Wahl singularity is a smooth point. When Wahl singularities in the extremal P-resolution are equal we have only one Mori train.    
\end{itemize}
\label{moritrain}
\end{definition}

Instead of trying general formulas for each wagon of the Mori trains, we give some examples.

\begin{example}
(Divisorial family) Consider the Wahl singularity $(P \in \bW)=\frac{1}{4}(1,1)$. Then $\delta=2$ and the Mori train is $$[4]-[2,\bar{2},6]-[2,2,2,\bar{2},8]-[2,2,2,2,2,\bar{2},10]-\cdots $$ For example, the initial \enii \ is $[4]-[2,2,6]$, the $[2,2,2,2,2,\bar{2},10]$ is an \eni~, and $[2,2,6]-[2,2,2,2,8]$ is another \enii.
\label{exdivfam}
\end{example}

\begin{example} (Flipping family) Let $\frac{1}{11}(1,3)$ be the c.q.s. $(P \in \bW)$. So $\Delta=11$ and $\Omega=3$. Consider the extremal P-resolution $W^+ \to \bW$ defined by $[4]-3-[]$. Here $\delta=3$, and $\Gamma^+$ is a $(-3)$-curve (after minimally resolving). Then the numerical data of any \eni ~and any \enii ~associated with $W^+$ can be read from the Mori trains
$$[]-[\bar{2},5,3]-[2,3,\bar{2},2,7,3]-[2,3,2,2,2,\bar{2},5,7,3]-\cdots
$$ and
$$[4]-[2,\bar{2},5,4]-[2,2,3,\bar{2},2,7,4]-[2,2,3,2,2,2,\bar{2},5,7,4]-\cdots$$ The initial \enii \ are $[]-[\bar{2},5,3]$ and $[4]-[2,\bar{2},5,4]$, corresponding to the smooth point and the Wahl singularity $\frac{1}{4}(1,1)$ in the extremal P-resolution. For particular examples, we have that $[2,3,\bar{2},2,7,3]$ and $[2,\bar{2},5,4]$ are \eni~ whose flips have $W^+$ as central fiber. Or $[2,3,2,2,7,3]-[2,3,2,2,2,2,5,7,3]$ is a \enii \ over $\frac{1}{11}(1,3)$.  
\label{exantiflipfam}
\end{example}

\begin{remark} \cite[\S2.3 and \S3.4]{HTU17}
Via the construction of the universal family in \cite[\S2.3 and \S3.4]{HTU17}, each non-initial wagon $[\Wa_i]$ of the Mori train represents a \eni. The two adjacent wagons $[\Wa_{i-1}]$ (say it is not initial) and $[\Wa_{i+1}]$ give the information of a deformation $\mathbb{W}_i \to \P^1$ which is $\Q$-Gorenstein, and has two fibers with \enii \ $[\Wa_{i-1}]-[\Wa_{i}]$ and $[\Wa_{i}]-[\Wa_{i+1}]$, and all other fibers are isomorphic to the \eni \ defined by $[\Wa_i]$ (with a fixed mark somewhere). So, when $\delta \geq 2$, we obtain an infinite chain of $\P^1$'s connecting all the \eni \ and \enii \ in the Mori train.
\label{defoverP1}
\end{remark}

\section{Degenerations of $\P^2$ as smooth deformations of $\F_1$} \label{s5}

After the work of B\v{a}descu \cite{B86}, Manetti \cite{Ma91}, and Hacking \cite{H04}  on degenerations of rational surfaces, Hacking and Prokhorov \cite{HP10} proved the following theorem for degenerations of $\P^2$ with only quotient singularities.

\begin{theorem} \cite[Corollary 1.2]{HP10}
Let $W$ be a projective surface with quotient singularities that admits a smoothing to $\P^2$. Then $W$ is a $\Q$-Gorenstein deformation of $\P(a^2,b^2,c^2)$, where $(a,b,c)$ is a Markov triple, and the smoothing to $\P^2$ is a W-surface.
\label{hackingprokhorov}
\end{theorem}

Therefore the Markov tree (Figure \ref{f0}) represents the numerical data to understand all of these degenerations of $\P^2$. Each of them over $\D$ is a W-surface, with central fiber a partial $\Q$-Gorenstein smoothing $W$ of some $\P(a^2,b^2,c^2)$, where $(a,b,c)$ is a Markov triple. We understand the minimal resolution of $W$ as particular blow-ups over $\F_{\beta}$ with $\beta=10,7,4$ \cite{Ma91}. Compare with Proposition \ref{combi1}.   

We saw that the MMP on W-surfaces has as one of its ending outcomes the degenerations of $\P^2$ in Theorem \ref{hackingprokhorov}. We call them Markovian planes. Although they are final outcomes of MMP, we can still run it on a birational modification of them to obtain a rich connection with Mori theory. This was done in \cite[Section 3]{Urz16a}, but not explicitly and without any further analysis. This is the purpose of the present paper. 

The trick is very simple, it is the "First W-blow-up" in  \cite[Section 3]{Urz16a}. Given a Markovian plane $\P^2 \rightsquigarrow W$ (Definition \ref{markovian}), let us blow-up a general section. Then we have a W-surface $\F_1 \rightsquigarrow W_0$. Note that $W_0$ has Picard number $2$, and the cone of curves of $W_0$ is generated by two curves $\Gamma_i \simeq \P^1$ $i=1,2$ such that $\Gamma_i \cdot K_{W_0}<0$ and $\Gamma_i^2<0$. These curves are extremal (see e.g. \cite[Lemma 1.22]{KM98}). In the minimal resolution $\widetilde{W}_0$ of $W_0$, the strict transforms of these two curves are the two components of a fiber to the composition $$\widetilde{W}_0 \to \F_{\beta} \to \P^1,$$ and they both are $(-1)$-curves. One of them, say $\Gamma_1$, is a $(-1)$-curve in $W_0$ (not passing through any singularity), and the other $\Gamma_2$ passes through one of the singularities of $W_0$. In the minimal resolution, $\Gamma_2$ touches transversally at one point the only section of $\widetilde{W}_0 \to \F_{\beta} \to \P^1$ which is an exceptional curve of $\widetilde{W}_0 \to W_0$. The situation $\Gamma_2 \subset W_0 \subset \W_0$ defines a \eni \ extremal nbhd of flipping type. See details in \cite[Section 3]{Urz16a}, where it is proved also that after this flip we encounter only flips and at each step they are unique. The reason is that the new generators of the cone of curves are precisely the flipped curve ($K$ positive) and a new flipping curve ($K$ negative). After some finitely many flips, we arrive at a smooth deformation  $\F_1 \rightsquigarrow \F_{m}$ with $m=3,5,7$. This depends on the number of singular points $1,2,3$ of $W$ respectively. This is \cite[Theorem 3.1]{Urz16a}.  

At this point, we have several questions about this MMP process. As $W$ is always a $\Q$-Gorenstein deformation of a $\P(a^2,b^2,c^2)$, we will only consider $W=\P(a^2,b^2,c^2)$. \textit{Is it possible to bound the number of flips for each Markov triple? What is the unique numerical data that we obtain from MMP to each Markov triple? How Markov triples are related via this MMP degenerations?} In the next section we will show answers for all of them.

\begin{remark}
Given a Markovian plane $\P^2 \rightsquigarrow W$ we could have blow-up at the Wahl singularities, using the birational geometry from the previous section. In \cite{Urz16a}, this is called a W-blow-up. In this paper, we do not analyze the effect of considering such situations. Essentially this should be equivalent to what we do, but the "first flip" would now be different.        
\end{remark}

\section{Mutations, Mori trains, and the complete MMP} \label{s6}


Let $\P^2 \rightsquigarrow \P(a^2,b^2,c^2)$ be a Markovian plane, where $(a,b,c)$ is some Markov triple. Consider the general blow-up $$\F_1 \rightsquigarrow \text{Bl}_P(\P(a^2,b^2,c^2))=:W_{a,b,c}$$ of the previous section. In the first subsections, we will explicitly show the MMP for any Markov triple. We do this at least up to a certain flip which will allow us to prove theorems in the final subsection. In this section, we prove all theorems announced in the introduction.

We recall that at each step in the MMP we have a unique flip. We will describe each flip line by line showing how curves change in the relevant chain of Wahl singularities. As described in Example \ref{manetti}, the initial chain of Wahl singularities is either $$\left[{a \choose w_a}\right]-(1)-\left[{c \choose w_c}\right]-(1)- \left[{b \choose w_b}\right],$$ when $(1<a<b<c)$, or $(0)-\left[{c \choose w_c}\right]-(1)- \left[{b \choose w_b}\right]$when $a=1<b<c$, or $(0)-\left[{2 \choose 1}\right]-(0)$ when $a=b=1<2=c$. In the minimal resolution $\widetilde{W}_0$ of $W_{a,b,c}$ we have that the pre-image of the corresponding chain of $\P^1$'s is formed by two fibers and one section of the fibration $\widetilde{W}_0 \to \F_{\beta} \to \P^1$, where $\beta=10,7,4$ respectively. As described in Section \ref{s5}, $\F_1 \rightsquigarrow W_{a,b,c}$ contains the flipping \eni \  neighborhood induced by the curve $\Gamma_2$. The resulting (first) flip is $\F_1 \rightsquigarrow {W_{a,b,c}}^+:=W_1$. From this point, all the (unique) flips will only modify curves and singularities on this chain of Wahl singularities \cite{Urz16a}. The flipped curve (positive for $K$) and the flipping curve (negative for $K$) belongs to the modified chain of Wahl singularities. The unique flipping curve will be highlighted as $\ldots-(1)_{-}-\ldots$, and the flipped curve as $\ldots-(c)_{+}-\ldots$ for some $c\geq 1$. The $i$-th flip will pass from $\F_1 \rightsquigarrow W_{i-1}$ to $\F_1 \rightsquigarrow W_i$. The W-surface $\F_1 \rightsquigarrow W_{0}$ corresponds to $\F_1 \rightsquigarrow W_{a,b,c}$. In the minimal resolution, the previously mentioned flipping curve is touching the unique section of $\widetilde{W}_0$ in the exceptional divisor, so we do not use any notation for that curve. For each flip, we will also indicate the corresponding $\delta$. We always take $w_c\equiv 3a^{-1}b \pmod{c}$, $w_b\equiv 3c^{-1}a \pmod{b}$, and 
$w_a\equiv 3b^{-1}c \pmod{a}$ for the Markov triple $(a<b<c)$. When we move in the corresponding branches, then these numbers may change, even if some $a,b,c$ does not.

\bigskip 

A very simple example is $\F_1 \rightsquigarrow W_{1,1,2}$. In this case the MMP has only one flip:

\bigskip

{\scriptsize $(0)-\left[{2 \choose 1}\right]-(0)$}

Flip 1: $\delta=1$

{\scriptsize $(0)-(3)_{+}-(0)$}

\bigskip

Next we will show explicitly the MMP on all the branches of the Markov tree. As computations is a cumbersome process in each case, we only give details for Subsection \ref{s63} in \ref{app}.

\subsection{MMP on the Fibonacci branch} \label{s61} 

Mori theory has its simplest form in this branch. We recall the Fibonacci branch: 

{\scriptsize  $$(1,1,1)-(1,1,2)-(1,m_0=2,m_1=5)-(1,5,13)-\ldots-(1<m_{k}<m_{k+1})-\ldots$$}

Observe that $w_{m_k}=m_{k-2}$. Set $m_{-1}=m_{-2}=1$. The MMP for $(1<m_{k}<m_{k+1})$ with $k\geq 0$ is:

\bigskip 
 
{\scriptsize$(0)-\left[{ m_{k+1} \choose m_{k-1}}\right]-(1)-\left[{m_{k} \choose m_{k-2}}\right]$}

Flip 1: $\delta=m_{k-1}$

{\scriptsize$(0)-\left[{ m_{k+1}-m_{k-1} \choose m_{k-1}}\right]-(2)_{+}-(1)_{-}-\left[{m_{k} \choose m_{k-2}}\right]$}

Flip 2: $\delta=m_{k-2}$

{\scriptsize$(0)-\left[{ m_{k+1}-m_{k-1} \choose m_{k-1}}\right]-(1)_{-}-\left[{m_{k}-m_{k-2} \choose m_{k-2}}\right]-(2)_{+}$}

Flip 3: $\delta=3$

{\scriptsize$(0)-(5)_{+}-(0)$}

\bigskip 

Here we note that the Fibonacci branch is completely connected via the Mori train corresponding to the extremal P-resolution $[]-5-[]$: $$[]-[\bar{2}, 2, 6]-[2, 2, 2, \bar{2}, 5, 6]-[2, 2, 2, 2, 3, \bar{2}, 2, 7, 6]-
[2, 2, 2, 2, 3, 2, 2, 2, \bar{2}, 5, 7, 6]-\ldots$$

For each vertex in this branch we are choosing one of the \enii \ in this train. After that all the anti-flips are determined. Therefore the key c.q.s. for the Fibonacci branch is $\frac{1}{5}(1,1)$ and $\delta= 3 \cdot 1=3$. 

\subsection{MMP on the Pell branch} \label{s62} At this branch, the MMP already makes a difference between the first vertex and the rest. We will see this in all other branches. Let us first recall the Pell branch:

{\scriptsize  $$(1,1,2)-(1,2,5)-(2,m_0=5,m_1=29)-(2,m_1=29,m_2=169)-\ldots-(2<m_{k}<m_{k+1})-\ldots$$}

\vspace{0.2cm}
The MMP for $(2<5<29)$ is (compare with \cite[Figure 1]{Urz16a}):

\bigskip 

{\scriptsize$\left[{2 \choose 1}\right]-(1)-\left[{ 29\choose 22}\right]-(1)-\left[{5 \choose 4}\right]$}

Flip 1: $\delta=1$

{\scriptsize$\left[{2 \choose 1}\right]-(1)-\left[{25 \choose 19}\right]-(1)_{+}-\left[{4 \choose 3}\right]-(1)_{-}-\left[{5 \choose 4}\right]$}

Flip 2: $\delta=1$

{\scriptsize$\left[{2 \choose 1}\right]-(1)-\left[{ 25 \choose 19}\right]-(1)_{-}-(2)_{+}-\left[{4 \choose 3 }\right]$}

Flip 3: $\delta=6$

{\scriptsize$\left[{2 \choose 1}\right]-(1)_{-}-(2)_{+}-\left[{ 19 \choose 13}\right]-(1)-\left[{4 \choose 3 }\right]$}

Flip 4: $\delta=1$

{\scriptsize $(3)_{+}-(1)_{-}-\left[{ 19 \choose 13}\right]-(1)-\left[{4 \choose 3 }\right]$}

Flip 5: $\delta=13$

{\scriptsize$(0)-\left[{ 6 \choose 5}\right]-(4)_{+}-(1)_{-}-\left[{4 \choose 3 }\right]$}

Flip 6: $\delta=3$

{\scriptsize$(0)-\left[{ 6 \choose 5}\right]-(1)_{-}-(5)_{+}$}

Flip 7: $\delta=5$

{\scriptsize$(0)-(7)_{+}-(0)$}
\bigskip

Let us set $m_{-1}=1$. The MMP for $(2<m_{k}<m_{k+1})$ with $k\geq 1$: 

\bigskip

{\scriptsize$\left[{2 \choose 1}\right]-(1)-\left[{ m_{k+1} \choose w_{m_{k+1}}}\right]-(1)-\left[{m_{k} \choose w_{m_k}}\right]$}

Flip 1: $\delta=m_{k-1}$

{\scriptsize$\left[{2 \choose 1}\right]-(1)-\left[{ m_{k+1}-4m_{k-1} \choose w_{m_{k+1}}-3m_{k-1}}\right]-(1)_{+}-\left[{4 \choose 3}\right]-(1)_{-}-\left[{m_{k} \choose w_{m_k}}\right]$}

Flip 2: $\delta=m_{k-2}$

{\scriptsize$\left[{2 \choose 1}\right]-(1)-\left[{ m_{k+1}-4m_{k-1} \choose w_{m_{k+1}}-3m_{k-1}}\right]-(1)_{-}-\left[{ m_{k}- 4m_{k-2} \choose w_{m_k}-3m_{k-2}}\right]-(1)_{+}-\left[{4 \choose 3}\right]$}

Flip 3: $\delta=6$

{\scriptsize$\left[{2 \choose 1}\right]-(1)_{-}-(2)_{+}-\left[{ 19 \choose 13}\right]-(1)-\left[{4 \choose 3 }\right]$}

\bigskip 

The next flip is equivalent to Flip 3 for $(2,5,29)$, and so we continued from there. This will happen for every other branch as we will see below. Again this MMP depends on one of the Mori trains of the extremal P-resolution $[]-2-[2,2,9,2,2,2,2,4]$. It is the Mori train of the smooth point:
$$[]-[\bar{3}, 2, 9, 2, 2, 2, 2, 4, 2]-[3, 2, 7, \bar{2}, 2, 2, 2, 2, 2, 5, 9, 2, 2, 2, 2, 4, 2]- \ldots $$


Again for each vertex of this Pell branch, we are choosing one \enii \ of this Mori train. The c.q.s. is $\frac{1}{476}(1,361)$ and $\delta= 3 \cdot 2=6$. The other Mori train of this c.q.s. is $$[2, 2, 9, 2, 2, 2, 2, 4]-[2, 2, 7, \bar{2}, 2, 2, 2, 2, 10, 2, 2, 2, 2, 4]- \ldots,$$ but it is not part of the MMP on the Markov tree.


\subsection{MMP on the branches of $(1=a<b<c)$} \label{s63} We recall that for a Markov triple $(1=a<b<c)$ we have two branches. The MMP for each behaves a bit different. 
\label{s6.3}


We start with the branch 

{\scriptsize  $$(1<b<c)-(b<c<3bc-1)-(c<m_0<m_1)-(c<m_1<m_2)-\ldots-(c<m_{k}<m_{k+1})-\ldots$$}

In this case, we have  $w_c=3b-c$ and $w_b=3w_c-b$. Set $m_{-1}=b$ and $m_{-2}=1$. The MMP for $(c<m_k<m_{k+1})$ with $k\geq 0$ is: 

\bigskip 

{\scriptsize$\left[{c \choose w_c}\right]-(1)-\left[{ m_{k+1} \choose w_{m_{k+1}}}\right]-(1)-\left[{m_{k} \choose w_{m_k}}\right]$}

Flip 1: $\delta=m_{k-1}$

{\scriptsize$\left[{c \choose w_c}\right]-(1)-\left[{ m_{k+1}-c^2m_{k-1} \choose w_{m_{k+1}}-m_{k-1}(c w_c+1)}\right]-(1)_{+}-\left[{c^2 \choose c w_c+1}\right]-(1)_{-}-\left[{m_{k} \choose w_{m_k}}\right]$}

Flip 2: $\delta=m_{k-2}$

{\scriptsize$\left[{c \choose w_c}\right]-(1)-\left[{ m_{k+1}-c^2m_{k-1} \choose w_{m_{k+1}}-m_{k-1}(c w_c+1)}\right]-(1)_{-}-\left[{ m_{k}- c^2m_{k-2} \choose w_{m_k}-m_{k-2}(c w_c+1)}\right]-(1)_{+}-\left[{c^2 \choose c w_c+1}\right]$}

Flip 3: $\delta=3c$

{\scriptsize$\left[{c \choose w_c}\right]-(1)_{-}-\left[{b^2 \choose bw_b-1}\right]-(1)_{+}-\left[{ (3c-b)c^2-b \choose (3c-b)(cw_c+1)-w_b}\right]-(1)-\left[{c^2 \choose c w_c+1}\right]$}

Flip 4: $\delta=3b-c$

{\scriptsize$\left[{c-w_c \choose w_c}\right]-(2)_{+}-(1)_{-}-\left[{ (3c-b)c^2-b \choose (3c-b)(cw_c+1)-w_b}\right]-(1)-\left[{c^2 \choose c w_c+1}\right]$}

Flip 5: $\delta=(3c-b)(cw_c+1)-w_b$

{\scriptsize$\left[{c-w_c \choose w_c}\right]-(1)-\left[{ (3c-b)(c(c-w_c)-1)-(b-w_b) \choose (3c-b)(cw_c+1)-w_b}\right]-(2)_+-(1)_{-}-\left[{c^2 \choose c w_c+1}\right]$}

Flip 6: $\delta=cw_c+1$

{\scriptsize$\left[{c-w_c \choose w_c}\right]-(1)-\left[{ (3c-b)(c(c-w_c)-1)-(b-w_b) \choose (3c-b)(cw_c+1)-w_b}\right]-(1)_{-}-\left[{c(c-w_c)-1 \choose c w_c+1}\right]-(2)_+$}

Flip 7: $\delta=3c-b$ \footnote{Excepting the Markov triple $(a=1,b=2,c=5)$ which is work out right after Flip 9.}

{\scriptsize$\left[{c-w_c \choose w_c}\right]-(1)_{-}-\left[{b-w_b \choose w_b}\right]-(3)_{+}-\left[{6\choose 5}\right]-(0)$}

Flip 8: $\delta=3$

{\scriptsize$(5)_{+}-(1)_{-}-\left[{6\choose 5}\right]-(0)$}

Flip 9: $\delta=5$

{\scriptsize$(0)-(7)_{+}-(0)$}

\bigskip

\textbf{Exception:} For $(c=5,m_0=13,m_1=433)$, the Flip 7 is given by:

\bigskip 
{\scriptsize$\left[{4 \choose 1}\right]-(1)_{-}-(4)_{+}-\left[{6\choose 5}\right]-(0)$}
\bigskip

which is precisely Flip 5 of $(2,5,29)$ written in reverse order.

\vspace{0.3cm}

As in the previous cases, the relevant flip that glues all the rest is Flip 3, where we choose \enii's \ from one of the Mori trains of the extremal P-resolution {\scriptsize$$\left[{b^2 \choose b(b-w_b)-1)}\right]-1-\left[{ (3c-b)c^2-b \choose (3c-b)(cw_c+1)-w_b}\right],$$} which has $\delta=3c$.  

\bigskip 

For the other branch 

{\scriptsize  $$(1<b<c)-(1<c<3c-b)-(c<m_0<m_1)-(c<m_1<m_2)-\ldots-(c<m_{k}<m_{k+1})-\ldots$$}

In this case we have $w_{m_0}=3c-2b$. The MMP for $(c<m_0<m_1)$ is: 

\bigskip 

{\scriptsize$\left[{c \choose c-w_c}\right]-(1)-\left[{ m_1\choose w_{m_1}}\right]-(1)-\left[{m_0 \choose w_{m_0}}\right]$}

Flip 1: $\delta=a=1$

{\scriptsize$\left[{c \choose c-w_c}\right]-(1)-\left[{ m_1-c^2 \choose w_{m_1}-(c(c-w_c)+1)}\right]-(1)_{+}-\left[{c^2 \choose c(c-w_c)+1}\right]-(1)_{-}-\left[{m_{0} \choose w_{m_0}}\right]$}

Flip 2: $\delta=b$

{\scriptsize$\left[{c \choose c-w_c}\right]-(1)-\left[{ m_1-c^2 \choose w_{m_1}-(c(c-w_c)+1)}\right]-(1)_{-}-(2)_{+}-\left[{m_0-b \choose w_{m_0}-b }\right]$}

Flip 3: $\delta=bm_0+1$

{\scriptsize$\left[{c \choose c-w_c}\right]-(1)_{-}-(2)_{+}-\left[{m_0w_{m_0}-1 \choose 2(m_0w_{m_0}-1)-m_0^2}\right]-(1)-\left[{m_0-b \choose w_{m_0}-b }\right]$}

Flip 4: $\delta=w_c$

{\scriptsize$(2)_{+}-\left[{c-w_c \choose c-2w_c}\right]-(1)_{-}-\left[{m_0w_{m_0}-1 \choose 2(m_0w_{m_0}-1)-m_0^2}\right]-(1)-\left[{m_0-b \choose w_{m_0}-b }\right]$}

Flip 5: $\delta=3m_0-c$

{\scriptsize$(0)-\left[{6\choose 1}\right]-(3)_{+}-\left[{c-w_c \choose c-2w_c}\right]-(1)_{-}-\left[{m_0-b \choose w_{m_0}-b }\right]$}

Flip 6: $\delta=3$

{\scriptsize$(0)-\left[{6\choose 1}\right]-(1)_{-}-(5)_{+}$}

Flip 7: $\delta=5$

{\scriptsize$(0)-(7)_{+}-(0)$}

\bigskip 

For $k\geq 1$, the MMP requires two extra steps to reach the result of the previous Flip 2. Set $m_{-1}=1$. The MMP for $(c<m_{k}<m_{k+1})$ with $k\geq 1$ is: 

\bigskip 

{\scriptsize$\left[{c \choose c-w_c}\right]-(1)-\left[{ m_{k+1} \choose w_{m_{k+1}}}\right]-(1)-\left[{m_{k} \choose w_{m_k}}\right]$}

Flip 1: $\delta=m_{k-1}$

{\scriptsize$\left[{c \choose c-w_c}\right]-(1)-\left[{ m_{k+1}-c^2m_{k-1} \choose w_{m_{k+1}}-m_{k-1}(c(c-w_c)+1)}\right]-(1)_{+}-\left[{c^2 \choose c(c-w_c)+1}\right]-(1)_{-}-\left[{m_{k} \choose w_{m_k}}\right]$}

Flip 2: $\delta=m_{k-2}$

{\scriptsize$\left[{c \choose c-w_c}\right]-(1)-\left[{ m_{k+1}-c^2m_{k-1} \choose w_{m_{k+1}}-m_{k-1}(c(c-w_c)+1)}\right]-(1)_{-}-\left[{ m_{k}- c^2m_{k-2} \choose w_{m_k}-m_{k-2}(c(c-w_c)+1)}\right]-(1)_{+}-\left[{c^2 \choose c(c-w_c)+1}\right]$}

Flip 3: $\delta=3c$

{\scriptsize$\left[{c \choose c-w_c}\right]-(1)-\left[{ m_1-c^2 \choose w_{m_1}-(c (c-w_c)+1)}\right]-(1)_{+}-\left[{ bc^2-m_0 \choose b(c(c-w_c)+1)-w_{m_0}}\right]-(1)_{-}-\left[{c^2 \choose c(c-w_c)+1}\right]$}

Flip 4: $\delta=b$

{\scriptsize$\left[{c \choose c-w_c}\right]-(1)-\left[{ m_1-c^2 \choose w_{m_1}-(c(c-w_c)+1)}\right]-(1)_{-}-(2)_{+}-\left[{m_0-b \choose w_{m_0}-b }\right]$}
\bigskip

We note that on Flip 3 we have the key extremal P-resolution {\scriptsize $$\left[{ m_1-c^2 \choose w_{m_1}-(c(c-w_c)+1)}\right]-1-\left[{ bc^2-m_0 \choose b(c(c-w_c)+1)-w_{m_0}}\right]$$} which defines a Mori train which is in bijection with this branch. It has again $\delta=3c$. This c.q.s. will be discussed later in the general case.

\subsection{MMP for general Markov triples} \label{s64} We now partially describe the MMP for any general Markov triple $(1<a<b<c)$. It turns out that the partial MMP describe here will glue to the MMP of the next smaller branch, and so on until it finishes. That will be proved in the next subsection.

A general Markov triple $(1<a<b<c)$ has two branches, and so we describe MMP for each of them.

We start with the branch 

{\scriptsize $$
(a<b<c)-(b<c<3 b c-a)-\left(c<m_0<m_1\right)-\left(c<m_1<m_2\right)-\ldots-\left(c<m_k<m_{k+1}\right)-\ldots$$}

The MMP for $(c<m_0<m_1)$ is: 

\bigskip 

{\scriptsize$\left[{c \choose w_c}\right]-(1)-\left[{ m_1\choose w_{m_1}}\right]-(1)-\left[{m_0 \choose w_{m_0}}\right]$}

Flip 1: $\delta=b$

{\scriptsize$\left[{c \choose w_c}\right]-(1)-\left[{ m_1-bc^2 \choose w_{m_1}-b(c w_c+1)}\right]-(1)_{+}-\left[{c^2 \choose c w_c+1}\right]-(1)_{-}-\left[{m_{0} \choose w_{m_0}}\right]$}

Flip 2: $\delta=a$

{\scriptsize$\left[{c \choose w_c}\right]-(1)-\left[{ m_1-bc^2 \choose w_{m_1}-b(c w_c+1)}\right]-(1)_{-}-\left[{b^2 \choose bw_b-1}\right]-(1)_{+}-\left[{m_0-ab^2 \choose w_{m_0}-a(bw_b-1)}\right]$}

Flip 3: $\delta=a(3c-ab)+b$

{\scriptsize$\left[{c \choose w_c}\right]-(1)_{-}-\left[{b^2 \choose bw_b-1}\right]-(1)_{+}-\left[{(m_0-ab^2)(3c-ab)-b \choose (w_{m_0}-a(b w_b-1))(3c-ab)-w_b }\right]-(1)-\left[{m_0-ab^2 \choose w_{m_0}-a(b w_b-1)}\right]$}





Flip 4: Let $u:=3ab-c$ and $w_u:\equiv3a^{-1}b \pmod{u}$. $\delta=u$

\vspace{0.1cm}

{\scriptsize$\left[{c-u a^2 \choose w_c- u (aw_a+1)}\right]-(1)_{+}-\left[{a^2 \choose aw_a+1}\right]-(1)_{-}-\left[{(m_0-ab^2)(3c-ab)-b \choose (w_{m_0}-a(b w_b-1))(3c-ab)-w_b }\right]-(1)-\left[{m_0-ab^2 \choose w_{m_0}-a(b w_b-1)}\right]$}

\vspace{0.1cm} 

Flip 5: Let $v:=3au-b$ and $w_v:\equiv3u^{-1}a \pmod{v}$. $\delta=(3c-ab)(u(3b-au)+a)-v$

\vspace{0.1cm}
{\scriptsize $\left[{c-ua^2 \choose w_c-u(aw_a+1)}\right]-(1)-\left[{(3b-au)((c-ua^2)(3c-ab)-3a^2)+v\choose (3b-au)((w_c-u(aw_a+1))(3c-ab)-3aw_a)+v+w_v}\right]-(1)_{+}- $ 

$\hspace{7.5 cm}\left[{a^2 \choose aw_a+1}\right]-(1)_{-}-\left[{m_0-ab^2 \choose w_{m_0}-a(b w_b-1)}\right]$}

Flip 6: $\delta=u(3b-a u)+a$ \footnote{Excepting $(u=1,a=2)$.\label{u=1}}

{\scriptsize$\left[{c-u a^2 \choose w_c-u (aw_a+1)}\right]-(1)-\left[{(3b-au)((c-ua^2)(3c-ab)-3a^2)+v\choose (3b-au)((w_c-u(aw_a+1))(3c-ab)-3aw_a)+v+w_v}\right]-(1)_{-}-$

$\hspace{6.3cm} \left[{(3b-au)(c-u a^2)-a \choose (w_c-u (aw_a+1))(3b-a u)-w_a}\right]-
(1)_{+}-\left[{a^2 \choose aw_a+1}\right]$}

Flip 7: $\delta=3c-ab$. We have two cases:

If $u<a$ \footnote{Excepting the $v=1$ case.\label{f=1}}, then

{\scriptsize$\left[{c-u a^2 \choose w_c-u(aw_a+1)}\right]-(1)-\left[{(c-u a^2)(3b-a u)-a \choose (w_c- u (aw_a+1))(3b-a u)-w_a}\right]-(1)_{+}-\left[{v a^2-b \choose v (aw_a+1)-w_b}\right]-
(1)_{-}-\left[{a^2 \choose aw_a+1}\right]$}

\vspace{0.3cm}
If $a<u$, then

{\scriptsize$\left[{c-u a^2 \choose w_c- u (aw_a+1)}\right]-(1)-\left[{b-v a^2 \choose w_b-v (aw_a+1)}\right]-(1)_{+}-\left[{M \choose Q}\right]-
(1)_{-}-\left[{a^2 \choose aw_a+1}\right]$}

\vspace{0.1cm}
where

{\scriptsize $\bullet$ $M=((c-ua^2)(3b-au)-a)-(3c-ab)(b-va^2),$}

\vspace{-0.15cm}
{\scriptsize $\bullet$ $Q=((w_c-u(aw_a+1))(3b-au)-w_a)-(3c-ab)(w_b-v(aw_a+1)).$}

\vspace{0.3cm}
For the case $u<a$ we ought to compute one more flip to be used later.

\vspace{0.1cm}
Flip 8: $\delta=v$. We have two cases: 

\vspace{0.1cm}
If $1<u$, then

{\scriptsize$\left[{c-u a^2 \choose w_c-u (aw_a+1)}\right]-(1)-\left[{(c-u a^2)(3b-a u)-a \choose (w_c- u (aw_a+1))(3b-a u)-w_a}\right]-(1)_{-}-\left[{u^2 \choose u(u-w_{u})-1}\right]$

$\hspace{8cm}-(1)_{+}- \left[{b-v u^2 \choose (b-w_b)-v (u(u-w_{u})-1)}\right]$}

If $u=1$, then

{\scriptsize$\left[{c-a^2 \choose w_c-(aw_a+1)}\right]-(1)-\left[{(c-a^2)(3b-a)-a \choose (w_c-(aw_a+1))(3b-a)-w_a}\right]-(1)_{-}-
(2)_{+}-\left[{b-v \choose (b-w_b)-v}\right]$}

\bigskip \bigskip

For $k\geq 1$, the MMP requires two extra steps to reach the result of the previous Flip 2. Set $m_{-1}=b$. The MMP for $(c<m_{k}<m_{k+1})$ with $k\geq 1$ is:

\bigskip 

{\scriptsize$\left[{c \choose w_c}\right]-(1)-\left[{ m_{k+1} \choose w_{m_{k+1}}}\right]-(1)-\left[{m_{k} \choose w_{m_k}}\right]$}

Flip 1: $\delta=m_{k-1}$

{\scriptsize$\left[{c \choose w_c}\right]-(1)-\left[{ m_{k+1}-c^2m_{k-1} \choose w_{m_{k+1}}-m_{k-1}(c w_c+1)}\right]-(1)_{+}-\left[{c^2 \choose c w_c+1}\right]-(1)_{-}-\left[{m_{k} \choose w_{m_k}}\right]$}

Flip 2: $\delta=m_{k-2}$

{\scriptsize$\left[{c \choose w_c}\right]-(1)-\left[{ m_{k+1}-c^2m_{k-1} \choose w_{m_{k+1}}-m_{k-1}(c w_c+1)}\right]-(1)_{-}-\left[{ m_{k}- c^2m_{k-2} \choose w_{m_k}-m_{k-2}(c w_c+1)}\right]-(1)_{+}-\left[{c^2 \choose c w_c+1}\right]$}

Flip 3: $\delta=3c$

{\scriptsize$\left[{c \choose w_c}\right]-(1)-\left[{ m_1-b c^2 \choose w_{m_1}-b(c w_c+1)}\right]-(1)_{+}-\left[{ ac^2-m_0 \choose a(cw_c+1)-w_{m_0}}\right]-(1)_{-}-\left[{c^2 \choose c w_c+1}\right]$}

Flip 4: $\delta=a$

{\scriptsize$\left[{c \choose w_c}\right]-(1)-\left[{ m_1-bc^2 \choose w_{m_1}-b(c w_c+1)}\right]-(1)_{-}-\left[{b^2 \choose bw_b-1}\right]-(1)_{+}-\left[{m_0-ab^2 \choose w_{m_0}-a(bw_b-1)}\right]$}

\bigskip 
This last line is exactly Flip 2 for the $(c<m_0<m_1)$ case.

\vspace{0.5cm}
\textbf{We now analyze the exceptions from the footnotes.}

For $v=1$ and $u<a$, note that $b-va^2>0$. Then, the seventh flip does not behave as the $u<a$ case, instead as the $a<u$ situation. From $v=1$, we derive the relations $b=u^2+a^2$ and $b-w_b=uw_u+a(a-w_a)$ which let us write this flip as follows:

\vspace{0.15cm}
Flip 7: $\delta=3c-ab$

{\scriptsize$\left[{c-u a^2 \choose w_c- u (aw_a+1)}\right]-(1)_{-}-\left[{u^2 \choose u(u-w_u)-1}\right]-(1)_{+}-\left[{M \choose Q}\right]-
(1)-\left[{b-u^2 \choose w_b-(u(u-w_u)-1)}\right]$}

\vspace{0.15cm}
Through operations involving the Markov equations $1+u^2+a^2=3ua$ and $u^2+a^2+b^2=3uab$, we can infer that $M=(3a-u)(a^2(3b-ua)-3u^2)+1$. Also, $M-Q=(3a-u)((b-w_b-(uw_u+1))(3b-au)-3uw_u)+1$. This reveals that the chain of Wahl singularities agrees precisely to Flip 5 of the MMP for the triple $(1<b<c<m_0)$, but written in reverse order. The position of the extremal neighborhood is also justified by this observation.

\vspace{0.15cm}
For example, consider $(c=433< m_0=37666<m_1=48928105)$. Here we have $v=1$ and $u=2<a=5$. This gives us:

\vspace{0.15cm}
Flip 7: $\delta=1154$

{\scriptsize$\left[{383 \choose 92}\right]-(1)_{-}-\left[{4 \choose 1}\right]-(1)_{+}-\left[{24870 \choose 5929}\right]-(1)-\left[{25 \choose 6}\right]$}

\bigskip
Now let us consider the exception $(u=1,a=2)$. In this case, we get $b=5$, $c=29$, $m_0=433$, and $m_1=37666$. From Flip 5, it follows that:

\vspace{0.15cm}
Flip 5: $\delta=1154$

{\scriptsize$\left[{25 \choose 19}\right]-(1)_{-}-\left[{24870 \choose 18941}\right]-(1)_{+}-\left[{4 \choose 3}\right]-(1)_{-}-\left[{383 \choose 291}\right]$}

Flip 6: $\delta=15$

{\scriptsize$\left[{25 \choose 19}\right]-(1)-\left[{24870 \choose 18941}\right]-(1)_{-}-\left[{323 \choose 246}\right]-(1)_{+}-\left[{4 \choose 3}\right]$}

Flip 7: $\delta=77$

{\scriptsize$\left[{25 \choose 19}\right]-(1)_{-}-(2)_{+}-\left[{249 \choose 146}\right]-(1)-\left[{4 \choose 3}\right]$}

Flip 8: $\delta=6$

{\scriptsize$(2)_{+}-\left[{19 \choose 13}\right]-(1)_{-}-\left[{249 \choose 146}\right]-(1)-\left[{4 \choose 3}\right]$}.
\bigskip 

Observe that Flip 7 of $(433,37666,48928105)$ agrees to Flip 5 of $(29,433,37666)$ written in reverse order. Also, Flip 7 of $(29,433,37666)$ corresponds to Flip 5 of $(5,29,433)$ written in reverse order as well.

\vspace{0.2cm}

\begin{remark}
We have that the key extremal P-resolution 
{\scriptsize$$\left[{ m_1-b c^2 \choose w_{m_1}- b(c w_c+1)}\right]-1-\left[{ ac^2-m_0 \choose a(cw_c+1)-w_{m_0}}\right]$$} 

\noindent contracts to $\frac{1}{\Delta}(1,\Omega)$ where $\Delta=c^2(c^2 D -(c-1)^2)=c^2(9c^4-5c^2+2c-1)$ and $\Omega=c^2+(cw_c+1)(c^2 D -(c-1)^2)$, and $D=9c^2-4$. This c.q.s. connects this whole Markov branch to one of its Mori trains. The HJ continued fraction of this singularity appeared already in Remark \ref{markovcqs} as $$[x_1,\ldots,x_r,2,y_s,\ldots,y_1+x_1,\ldots,x_r+y_s,\ldots,y_1+1,\underbrace{2,\ldots,2}_6,x_1+1,\ldots,x_r+y_s,\ldots,y_1],$$ where $m=c$ and $q=w_c$. The relevant c.q.s., which given the bijection with a Mori train, has ``inverted" c.q.s. as $m=c$ but $q=c-w_c$, and so it is $$[y_1,\ldots,y_s,2,x_r,\ldots,x_1+y_1,\ldots,y_s+x_r,\ldots,x_1+1,\underbrace{2,\ldots,2}_6,y_1+1,\ldots,y_s+x_r,\ldots,x_1].$$ As we saw in Remark \ref{markovcqs}, the study of extremal P-resolutions can be reduced to $$\frac{\Delta_0}{\Omega_0} =[5,x_1,\ldots,x_r,2,y_s,\ldots,y_1,5],$$ where $\delta=3c$. It admits an extremal P-resolution with a $(-2)$-middle-curve.
\label{markovcqsII}
\end{remark}

We now look at the other branch of $(1<a<b<c)$ 

{\scriptsize $$
(a<b<c)-(a<c<3 ac-b)-\left(c<m_0<m_1\right)-\left(c<m_1<m_2\right)-\ldots-\left(c<m_k<m_{k+1}\right)-\ldots$$} 

Let us also consider the minimum possible Markov tripe $(a<p_0<p_1)$ which can be obtained mutating $(a<b<c)$ by keeping $a$ as the smallest number in the triples (i.e. $(a<p_0<p_1)$ is an initial vertex for the branch corresponding to $(a<b<c)$ keeping $a$). Then the MMP for $(c<m_0<m_1)$ is:

\bigskip 

{\scriptsize$\left[{c \choose c-w_c}\right]-(1)-\left[{ m_1\choose w_{m_1}}\right]-(1)-\left[{m_0 \choose w_{m_0}}\right]$}

Flip 1: $\delta=a$

{\scriptsize$\left[{c \choose c-w_c}\right]-(1)-\left[{ m_1-ac^2 \choose w_{m_1}-a(c (c-w_c)+1)}\right]-(1)_{+}-\left[{c^2 \choose c(c-w_c)+1}\right]-(1)_{-}-\left[{m_{0} \choose w_{m_0}}\right]$}

Flip 2: $\delta=b$

{\scriptsize$\left[{c \choose c-w_c}\right]-(1)-\left[{ m_1-ac^2 \choose w_{m_1}-a(c (c-w_c)+1)}\right]-(1)_{-}-\left[{a^2 \choose a(a-w_a)-1}\right]-(1)_{+}-\left[{m_0-ba^2 \choose w_{m_0}-b(a(a-w_a)-1)}\right]$}

Flip 3: $\delta=b(3c-ab)+a$

{\scriptsize$\left[{c \choose c-w_c}\right]-(1)_{-}-\left[{a^2 \choose a(a-w_a)-1}\right]-(1)_{+}-\left[{(m_0-ba^2)(3c-ab)-a \choose (w_{m_0}-b(a(a- w_a)-1))(3c-ab)-(a-w_a) }\right]$

$\hspace{8.5cm}-(1)-\left[{m_0-ba^2 \choose w_{m_0}-b(a(a-w_a)-1)}\right]$}




Flip 4: $\delta=u=3ab-c$

{\scriptsize$\left[{a^2 \choose a(a-w_a)-1}\right]-(1)_{+}-\left[{c-u a^2 \choose (c-w_c)-u (a(a-w_a)-1)}\right]-(1)_{-}-\left[{(m_0-ba^2)(3c-ab)-a \choose (w_{m_0}-b(a(a- w_a)-1))(3c-ab)-(a-w_a) }\right]$

$\hspace{8.3cm}-(1)-\left[{m_0-ba^2 \choose w_{m_0}-b(a(a-w_a)-1)}\right]$}

Flip 5: $\delta=3m_0-ac$

{\scriptsize$\left[{a^2 \choose a(a-w_a)-1}\right]-(1)-\left[{M \choose Q}\right]-(1)_{+}-\left[{c-u a^2 \choose (c-w_c)- u (a(a-w_a)-1)}\right]-(1)_{-}-\left[{m_0-ba^2 \choose w_{m_0}-b(a(a-w_a)-1)}\right]$}

\vspace{0.1cm}
where

{\scriptsize $\bullet$ $M=((m_0-ba^2)(3c-ab)-a)-(3m_0-ac)(c-ua^2),$}

\vspace{-0.15cm}
{\scriptsize $\bullet$ $Q=((w_{m_0}-b(a(a- w_a)-1))(3c-ab)-(a-w_a))-(3m_0-ac)((c-w_c)-u(a(a-w_a)-1)).$}

\vspace{0.2cm}
Flip 6: Let $e:=3ap_0-p_1$, $w_e:\equiv 3p_0^{-1}a \pmod{e}$, and $f:=3ae-p_0$ \footnote{Excepting the $f=1$ case and  $(a=2,e=1)$.}. $\delta=3a$

{\scriptsize$\left[{a^2 \choose a(a-w_a)-1}\right]-(1)-\left[{M \choose Q}\right]-(1)_{-}-\left[{fa^2-p_0 \choose f(a(a-w_a)-1)-(p_0-w_{p_0}) }\right]-(1)_{+}-\left[{p_1-ea^2 \choose (p_1-w_{p_1})-e(a(a-w_a)-1)}\right]$}
\vspace{0.1cm}

Flip 7: $\delta=3a(3p_0-ea)-(3e-fa)$

{\scriptsize$\left[{a^2 \choose a(a-w_a)-1}\right]-(1)_{-}-\left[{fa^2-p_0 \choose f(a(a-w_a)-1)-(p_0-w_{p_0}) }\right]-(1)_{+}-\left[{(p_1-ea^2)(3p_0-ea)-a \choose ((p_1-w_{p_1})-e(a(a-w_a)-1))(3p_0-ae)-(a-w_a)}\right]$

$\hspace{8.7cm}-(1)-\left[{p_1-ea^2 \choose (p_1-w_{p_1})-e(a(a-w_a)-1)}\right]$}

Flip 8: $\delta=f$. We have two cases: 

If $e,f>1$, then

\vspace{0.1cm}
{\scriptsize$\left[{p_0-fe^2 \choose (p_0-w_{p_0})-f(ew_e+1)}\right]-(1)_{+}-\left[{e^2 \choose ew_e+1 }\right]-(1)_{-}-\left[{(p_1-ea^2)(3p_0-ea)-a \choose ((p_1-w_{p_1})-e(a(a-w_a)-1))(3p_0-ae)-(a-w_a)}\right]$

\vspace{0.2cm}
$\hspace{8.7cm}-(1)-\left[{p_1-ea^2 \choose (p_1-w_{p_1})-e(a(a-w_a)-1)}\right]$}

If $e=1$, then

{\scriptsize$\left[{p_0-f \choose p_0-w_{p_0}}\right]-(2)_{+}-(1)_{-}-\left[{(p_1-a^2)(3p_0-a)-a \choose ((p_1-w_{p_1})-(a(a-w_a)-1))(3p_0-a)-(a-w_a)}\right]-(1)-\left[{p_1-a^2 \choose (p_1-w_{p_1})-(a(a-w_a)-1)}\right]$}
\bigskip 

For $k\geq 1$, the MMP requires two extra steps to reach the result of the previous Flip 2. Set $m_{-1}=a$. The MMP for $(c<m_{k}<m_{k+1})$ with $k\geq 1$ is:

\bigskip 

{\scriptsize$\left[{c \choose c-w_c}\right]-(1)-\left[{ m_{k+1} \choose w_{m_{k+1}}}\right]-(1)-\left[{m_{k} \choose w_{m_k}}\right]$}

Flip 1: $\delta=m_{k-1}$

{\scriptsize$\left[{c \choose c-w_c}\right]-(1)-\left[{ m_{k+1}-c^2m_{k-1} \choose w_{m_{k+1}}-m_{k-1}(c(c-w_c)+1)}\right]-(1)_{+}-\left[{c^2 \choose c(c-w_c)+1}\right]-(1)_{-}-\left[{m_{k} \choose w_{m_k}}\right]$}

Flip 2: $\delta=m_{k-2}$

{\scriptsize$\left[{c \choose c-w_c}\right]-(1)-\left[{ m_{k+1}-c^2m_{k-1} \choose w_{m_{k+1}}-m_{k-1}(c(c-w_c)+1)}\right]-(1)_{-}-\left[{ m_{k}- c^2m_{k-2} \choose w_{m_k}-m_{k-2}(c(c-w_c)+1)}\right]-(1)_{+}-\left[{c^2 \choose c(c-w_c)+1}\right]$}

Flip 3: $\delta=3c$

{\scriptsize$\left[{c \choose c-w_c}\right]-(1)-\left[{ m_1-a c^2 \choose w_{m_1}-a(c(c-w_c)+1)}\right]-(1)_{+}-\left[{ bc^2-m_0 \choose b(c(c-w_c)+1)-w_{m_0}}\right]-(1)_{-}-\left[{c^2 \choose c(c-w_c)+1}\right]$}

Flip 4: $\delta=b$

{\scriptsize$\left[{c \choose c-w_c}\right]-(1)-\left[{ m_1-ac^2 \choose w_{m_1}-a(c(c-w_c)+1)}\right]-(1)_{-}-\left[{a^2 \choose a(a-w_a)-1}\right]-(1)_{+}-\left[{m_0-ba^2 \choose w_{m_0}-b(a(a-w_a)-1)}\right]$}
\bigskip

This is exactly Flip 2 for the $(c<m_0<m_1)$ case.
\vspace{0.5cm}

\textbf{We now analyze the exceptions in the previous footnotes.}
\vspace{0.15cm}

For $(f=1<e<a)$, we observe that Flip 5 has the same extremal P-resolution as the Flip 2 of a triple $(a<p_k<p_{k+1})$ in the branch
{\scriptsize  $$(1<e<a)-(e<a<p_0)-(a<p_0<p_1)-\ldots-(a<c=p_k<m_0=p_{k+1})-\ldots$$.}
Therefore, we find that

\bigskip 
Flip 6: $\delta=3a$

{\scriptsize$\left[{a^2 \choose a(a-w_a)-1}\right]-(1)-\left[{M \choose Q}\right]-(1)_{-}-\left[{(3a-e)a^2-e \choose (3a-e)(a(a-w_a)-1)-w_e }\right]-(1)_{+}-\left[{e^2 \choose ew_e+1}\right]$}

\bigskip 
where $w_e\equiv 3p_0^{-1}a \pmod{e}$. After some computations we are able to prove that {\scriptsize $M=(3a-e)((p_0-e^2)(3p_0-ea)-3e^2)+1$} and {\scriptsize $Q=(3a-e)((p_0-w_{p_0}-(ew_e+1))(3p_0-ea)-3ew_e)+1$}. Indeed, that chain of Wahl singularities is exactly the same as Flip 6 of the Markov triple $(p_0<p_1<3p_0p_1-a)$.

\vspace{0.3cm}
 For $(e=1,a=2)$, the Markov triples that satisfy this situation are of the form $(c<m_0<m_1=3m_0c-2)$, where $c$ and $m_0$ are Pell numbers. We obtain,

\bigskip 
Flip 6: $\delta=6$

{\scriptsize$\left[{4 \choose 1}\right]-(1)-\left[{246 \choose 77}\right]-(1)_{-}-\left[{19 \choose 6}\right]-(2)_{+}$}

\bigskip 

From the exceptions of the previous case, we note that this chain of Wahl singularities agrees to Flip 6 of the Markov triple $(5,13,433)$.

\vspace{0.5cm}
The computations from Sections 7.1 to 7.4 give a proof of Theorem \ref{branches-trains}.

\subsection{The complete MMP on an arbitrary Markov triple} \label{s65}

In this section, we are going to prove Theorem \ref{changes-branches} and Theorem \ref{complete-MMP}. In addition, we will have a full description of MMP for any Markov triple $(a,b,c)$. 

The idea is the following. Let $(a<b<c)$ be an arbitrary Markov triple. Then it belongs to a branch where the label by $a_0:=a$. Hence we have the a situation  {\scriptsize $$ \left(a_0<m_{0,0}<m_{1,0} \right)-\left(a_0<m_{1,0}<m_{2,0}\right)-\ldots-\left(a_0<m_{k,0}=b<m_{k+1,0}=c\right)-\ldots$$} 
for some $k\geq 0$, where $3a m_{0,0}-m_{1,0} <a$. This is, this branch starts with the vertex $(a_0<m_{0,0}<m_{1,0})$. 

By Theorem \ref{branches-trains}, proved in the previous sections, we have that MMP stabilizes to the MMP of $(a_0<m_{0,0}<m_{1,0})$ for all such pairs after at most 4 flips. If $a_0=1$, then we are in the Fibonacci branch and we know how MMP ends.

If $a_0>1$, then we consider the unique mutation $$(a_1:=3a_0 m_{0,0}-m_{1,0}<a_0<m_{0,0})$$ from $(a_0<m_{0,0}<m_{1,0})$ which decreases the first coordinate. We will prove in Lemma \ref{lemma8} and Lemma \ref{lemma12} that the MMP of $(a_0<m_{0,0}<m_{1,0})$ stabilizes to the MMP of $(a_1<a_0<m_{0,0})$ after a few flips. Thus we consider the branch corresponding to $(a_1<a_0<m_{0,0})$ labelled by $a_1$, and we repeat. In this process, we will be considering parts of branches {\scriptsize $$ \left(a_i<m_{0,i}<m_{1,i} \right)-\left(a_i<m_{1,i}<m_{2,i}\right)-\ldots-\left(a_i<m_{k_i,i}<m_{k_i+1,i}\right)-\ldots$$} for various $i$, until we arrive to $a_\nu=1$ for some $\nu$.


In the next two lemmas, we will change branches from  $(a=a_0<b=m_{0,0}<c=m_{1,0})$ to $(a_1<m_{k_1,1}=a<m_{k_1+1,1}=b)$ to simplify notation. 

\begin{lemma}
Let $(a=a_0<b=m_{0,0}<c=m_{1,0})$ be a Markov triple such that $3ab-c<a$. If the mutation $(a_1=3ab-c<a_0<m_{0,0})$ satisfies that $a_2=3a_1a-b<a_1$, then the MMP of $\F_1  \rightsquigarrow  W_{a,b,c}$ stabilizes to the MMP of $\F_1  \rightsquigarrow W_{a_1,a,b}$ in at most 8 flips.
\label{lemma8}
\end{lemma}

\begin{figure}[htbp]
\centering
\includegraphics[width=12cm]{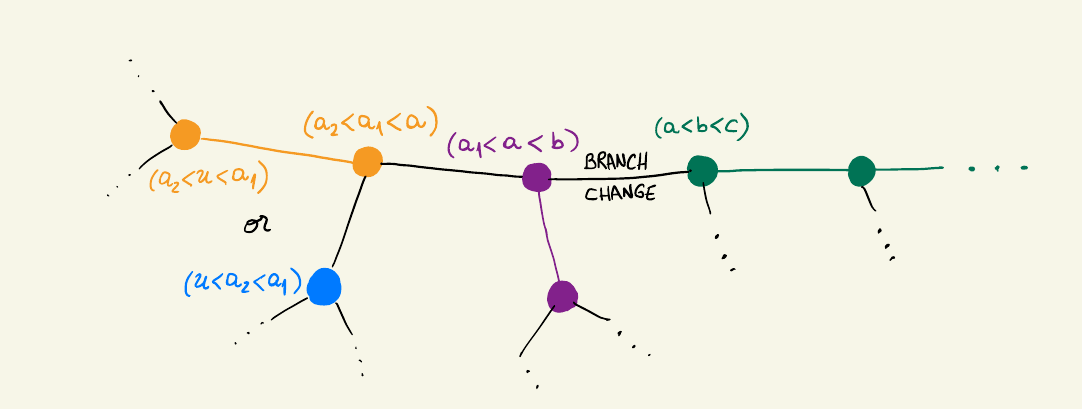}
\caption{Guiding diagram for Lemma \ref{lemma8}} 
\label{f2}
\end{figure}

\begin{proof}
See Figure \ref{f2} for a guide through this proof. 

Let $u=3a_2a_1-a_0$ and $v=3a_2u-a_1$. Based on the notation previously introduced, observe that $u$ plays the role of $a_3$ when $u<a_2$, and of $m_{i-1,2}$ for some $i$ otherwise. We proceed by analyzing $\min\{a_2,u\}$. First we study the cases having $\min\{a_2,u\}=1$.

\vspace{0.2cm}
\noindent 
\textbf{Case $a_2=1$:} We fall into the situation
{\scriptsize  $$(1<m_{i-1,2}<a_1)-(1<a_1<a_0)-(a_1<a<b)-(a<b<c)$$}

Following the conventions of \ref{s63}, we obtain the following chain of Wahl singularities at Flip 7 for the MMP on $\F_1  \rightsquigarrow  W_{a,b,c}$.

{\scriptsize $$\left[{a_0-w_{a_0} \choose w_{a_0}}\right]-(1)_{-}-\left[{a_1-w_{a_1} \choose w_{a_1}}\right]-(3)_{+}-\left[{6\choose 5}\right]-(0).$$}

 Note that the T-weight of $a_1$ in $(1<m_{i-1,2}<a_1)$ is equal to $w_{a_1}$, also the T-weight of $a_0$ in $(a_1<a<b)$ is $a-w_{a}$ and $m_{i-1,2}=w_{a}$. Therefore, this step corresponds precisely to Flip 5 of the MMP on $\F_1  \rightsquigarrow W_{a_1,a,b}$ but presented in reverse order.

\vspace{0.2cm}
\noindent 
\textbf{Case $u=1$:} We fall into the situation 

{\scriptsize $$
(1<a_2<a_1)-(a_2<a_1<a)-(a_1<a<b)-(a<b<c).$$}

Following the conventions of \ref{s64}, we get the following chain of Wahl singularities at Flip 8 of the MMP on $\F_1  \rightsquigarrow W_{a,b,c}$.

{\scriptsize $$\left[{a-a_2^2 \choose w_a-(a_2w_{a_2}+1)}\right]-(1)-\left[{(a-a_2^2)(3a_1-a_2)-a_2 \choose (w_a-(a_2w_{a_2}+1))(3a_1-a)-w_{a_2}}\right]-(1)_{-}-
(2)_{+}-\left[{a_1-v \choose (a_1-w_{a_1})-v}\right],$$}where $v=3a_2-a_1$. Since $u=1$, from the Markov equation and \cite[Lemma 5.5]{Per22} (see also Proposition \ref{prop}) we obtain the equations $a=a_2^2+a_1^2$ and $w_a=a_2w_{a_2}+a_1w_{a_1}$, which allow us to rewrite the chain as

{\scriptsize $$\left[{a_1^2 \choose a_1w_{a_1}-1}\right]-(1)-\left[{(3a_1-a_2)a_1^2-a_2 \choose (3a_1-a_2)(a_1w_{a_1}-1)-w_{a_2}}\right]-(1)_{-}-
(2)_{+}-\left[{a_1-v \choose (a_1-w_{a_1})-v}\right].$$}

By taking into account that the T-weights of $a_2$ and $a_1$ appearing in $(1<a_2<a_1)$ are $a_1-w_{a_1}$ and $a_2-w_{a_2}$ respectively, we deduce that this is identical to Flip 4 of the MMP on $\F_1  \rightsquigarrow W_{a_1,a,b}$ written in reverse order.

\vspace{0.3cm}
We now assume that $\min\{a,u\}>1$. Notice that the exceptional case $(v=1,u<a_2)$ is included in this category. However, we omit it from the remainder of the proof since we proved already stabilization to the MMP of $\F_1  \rightsquigarrow W_{a_1,a,b}$ in the Section \ref{s64}.

\bigskip
We are now in the MMP for general Markov triples having
{\scriptsize $$
(a_2<a_1<a)-(a_1<a<b)-(a<b<c).$$}

From the computations performed in \ref{s64}, we again treat the cases $u<a_2$ and $a_2<u$ independently. This refers to $(a_1<a<b)$ belonging to the first and second branch of $(a<b<c)$ respectively.

\vspace{0.2cm}
\noindent 
\textbf{Case $u<a_2$:} We note that Flip 8 of the MMP on $\F_1  \rightsquigarrow W_{a,b,c}$ is {\scriptsize $$\left[{a-ua_2^2 \choose w_a-u(a_2w_{a_2}+1)}\right]-(1)-\left[{(a-ua_2^2)(3a_1-a_2u)-a_2 \choose (w_a-u(a_2w_{a_2}+1))(3a_1-a_2u)-w_{a_2}}\right]-(1)_{-}-\left[{u^2 \choose u(u-w_{u})-1}\right]$$

$$\hspace{7.4cm}-(1)_{+}-\left[{a_1-vu^2 \choose (a_1-w_{a_1})-v(u(u-w_{u})-1)}\right].$$} By reversing the order of this chain, we obtain

\vspace{0.1cm}
{\scriptsize $$\hspace{-5.5cm}\left[{a_1-vu^2 \choose w_{a_1}-u(uw_{u}+1)}\right]-(1)_{+}-\left[{u^2 \choose uw_u+1}\right]-(1)_{-}-$$

$$\left[{(a-ua_2^2)(3a_1-a_2u)-a_2 \choose (a-w_{a})-u(a_2(a_2-w_{a_2})-1))(3a_1-a_2u)-(a_2-w_{a_2})}\right]-(1)-\left[{a-a_2a_1^2 \choose (a-w_{a})-u(a_2(a_2-w_{a_2})-1)}\right].$$}

The T-weight of $a$ in $(a_1<a<b)$ is $a-w_a$, and the corresponding of $a_2$ in $(u<a_2<a_1)$ is $a_2-w_{a_2}$. Thus, we infer that this chain of Wahl singularities is exactly Flip 4 of the MMP on $\F_1  \rightsquigarrow W_{a_1,a,b}$.

\vspace{0.2cm}
\noindent 
\textbf{Case $u>a_2$:} We note that Flip 7 of the MMP on $\F_1  \rightsquigarrow W_{a,b,c}$ is given by

{\scriptsize $$\left[{a-ua_2^2 \choose w_a-u(a_2w_{a_2}+1)}\right]-(1)-\left[{a_1-va_2^2 \choose w_{a_1}-v(a_2w_{a_2}+1)}\right]-(1)_{+}-\left[{M \choose Q}\right]-
(1)_{-}-\left[{a_2^2 \choose a_2w_{a_2}+1}\right] $$}

\vspace{0.1cm}
where

{\scriptsize $\bullet$ $M=((a-ua_2^2)(3a_1-a_2u)-a_2)-(3a-a_2a_1)(a_1-va_2^2),$}

\vspace{-0.15cm}
{\scriptsize $\bullet$ $Q=((w_a-u(a_2w_{a_2}+1))(3a_1-a_2u)-w_{a_2})-(3a-a_2a_1)(w_{a_1}-v(a_2w_{a_2}+1))$.}

\vspace{0.1cm}
As in the previous case, we proceed to write the chain in reverse order. Then,

{\scriptsize $$\left[{a_2^2 \choose a_2(a_2-w_{a_2})-1}\right]-(1)-\left[{M \choose M-Q}\right]-(1)_{+}-\left[{a_1-va_2^2 \choose (a_1-w_{a_1})-v(a_2(a_2-w_{a_2})-1)}\right]$$

$$\hspace{7cm}-(1)_{-}-\left[{a-a_2u^2 \choose w_{a_2}-u(a_2(a_2-w_{a_2})-1)}\right]$$}

\vspace{0.1cm}
where

{\scriptsize $\bullet$ $M-Q=((w_{a_2}-u(a_2(a_2-w_{a_2})-1))(3a_1-a_2)-(a_2-w_{a_2}))-(3a-a_2a_1)((a_1-w_{a_1})-u(a_2(a_2-w_{a_2})-1)).$}

\vspace{0.2cm}

In this case, the T-weights associated to $a_2$ and $a_1$ in $(a_2<u<a_1)$ are the same $w_{a_2}$ and $w_{a_1}$ of $(a_2<a_1<a_0)$ respectively. Thus, this chain of Wahl singularities corresponds exactly to Flip 5 of the MMP on $\F_1  \rightsquigarrow W_{a_1,a,b}$.
\end{proof}

\begin{lemma}
Let $(a=a_0<b=m_{0,0}<c=m_{1,0})$ be a Markov triple such that $3ab-c<a$. If the mutation $(a_1=3ab-c<a_0<m_{0,0})$ satisfies that $a_2=3a_1a-b>a_1$, then the MMP of $\F_1  \rightsquigarrow  W_{a,b,c}$ stabilizes to the MMP of $\F_1  \rightsquigarrow W_{a_1,a,b}$ in at most 12 flips.

\label{lemma12}
\end{lemma}

\begin{figure}[htbp]
\centering
\includegraphics[width=12cm]{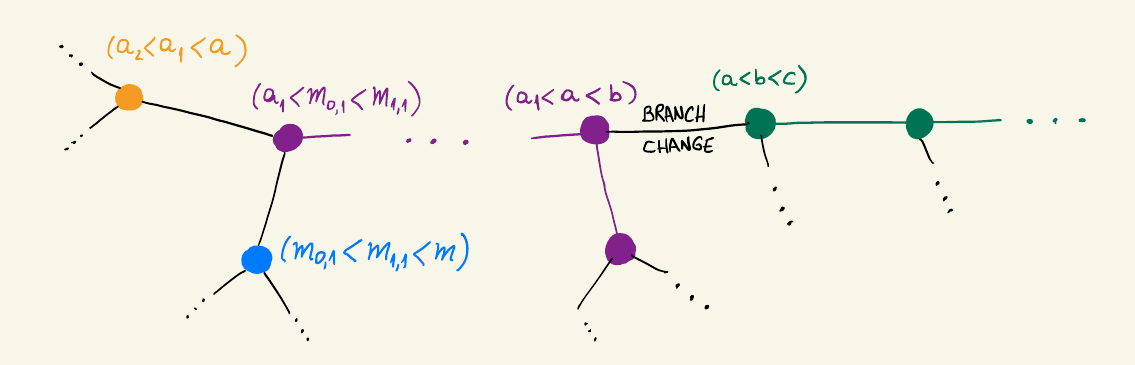}
\caption{Guiding diagram for Lemma \ref{lemma12}} 
\label{f3}
\end{figure}

\begin{proof}
See Figure \ref{f3} for a guide through this proof. 

Using the notation introduced at the beginning of this section, we have $a_2=3a_1m_{0,1}-m_{1,1}$ and we define $f:=3a_2a_1-m_{0,1}$. The proof proceeds by cases, similar to Lemma \ref{lemma8}, depending on the minimum of $\{a_2,f\}$. The key element of this proof is the stabilization provided for the MMP on $\F_1  \rightsquigarrow W_{m_{0,1},m_{1,1},3m_{0,1}m_{1,1}-a_1}$ from Lemma \ref{lemma8}. For convenience, we define $m:=3m_{0,1}m_{1,1}-a_1$. 

\vspace{0.2cm}
We proceed with the cases where $\min\{a_2,f\}=1$.

\vspace{0.2cm}
\noindent 
\textbf{Case $a_2=1$:} The computations in \ref{s64}  let us infer that Flip 8 of the MMP on $\F_1  \rightsquigarrow W_{a,b,c}$ is given by

{\scriptsize $$\hspace{-0.15cm}\left[{m_{0,1}-f \choose m_{0,1}-w_{m_{0,1}}}\right]-(2)_{+}-(1)_{-}-\left[{(m_{1,1}-a_1^2)(3m_{0,1}-a_1)-a_1 \choose ((m_{1,1}-w_{m_{1,1}})-(a_1(a_1-w_{a_1})-1))(3m_{0,1}-a_1)-(a_1-w_{a_1})}\right]$$
$$\hspace{6.45cm}-(1)-\left[{m_{1,1}-a_1^2 \choose (m_{1,1}-w_{m_{1,1}})-(a_1(a_1-w_{a_1})-1)}\right],$$} where $w_{m_{0,1}}$ and $w_{m_{1,1}}$ are the corresponding T-weights in $(a_1<m_{0,1}<m_{1,1})$. Note that in the present situation, one obtains  $m_{1,1}=a_1^2+m_{0,1}^2$ and $w_{m_{1,1}}=m_{0,1}w_{m_{0,1}}+a_1w_{a_1}$. Thus, the chain becomes
{\scriptsize $$\left[{m_{0,1}-f \choose m_{0,1}-w_{m_{0,1}}}\right]-(2)_{+}-(1)_{-}-\left[{(3m_{0,1}-a_1)m_{0,1}^2-a_1 \choose (3m_{0,1}-a_1)(m_{0,1}(m_{0,1}-w_{m_{0,1}})+1)-(a_1-w_{a_1})}\right] \ \ \ \ \ \ \ \ \ \ \ \ \ \ \ \ \ \ \ \ 
$$

$$\hspace{7.9cm}-(1)-\left[{m_{0,1}^2 \choose m_{0,1}(m_{0,1}-w_{m_{0,1}})+1}\right].$$}

Since the T weights of $a_1$ and $m_{0,1}$ in $(1<a_1<m_{0,1})$ are $a_1-w_{a_1}$ and $m_{0,1}-w_{m_{0,1}}$ respectively, we observe that the latter chain is exactly the same as Flip 4 of the MMP on $\F_1  \rightsquigarrow W_{m_{0,1},m_{1,1},m}$. By the ($a_2=1$) case of Lemma \ref{lemma8}, we know we ought to compute three flips more to reach Flip 5 of the MMP on $\F_1  \rightsquigarrow W_{a_1,m_{0,1},m_{1,1}}$.

\vspace{0.2cm}
\noindent 
\textbf{Case $f=1$:} We fall into the exception described in the footnote at the end of \ref{s64}. In this situation, we deduced that Flip 6 of the MMP on $\F_1  \rightsquigarrow W_{a,b,c}$ coincides to Flip 6 of the MMP on $\F_1  \rightsquigarrow W_{m_{0,1},m_{1,1},m}$. Then, by the general form of Lemma \ref{lemma8}, we compute two flips more to reach Flip 4 of the MMP on $\F_1  \rightsquigarrow W_{a_1,m_{0,1},m_{1,1}}$.

\vspace{0.2cm}
Now, we suppose that $1<\min\{a_2,f\}$. In this case we have that Flip 8 of the MMP on $\F_1  \rightsquigarrow W_{a,b,c}$ is given by

{\scriptsize $$\hspace{-3.7cm}\left[{m_{0,1}-fa_2^2 \choose (m_{0,1}-w_{m_{0,1}})-f(a_2w_{a_2}+1)}\right]-(1)_{+}-\left[{a_2^2 \choose a_2w_{a_2}+1 }\right]-(1)_{-}-$$
$$\left[{(m_{1,1}-a_2a_1^2)(3m_{0,1}-a_2a_1)-a_1 \choose ((m_{1,1}-w_{m_{1,1}})-a_2(a_1(a_1-w_{a_1})-1))(3m_{0,1}-a_1a_2)-(a_1-w_{a_1})}\right]-(1)-$$

$$\hspace{6.5cm}\left[{m_{1,1}-a_2a_1^2 \choose (m_{1,1}-w_{m_{1,1}})-a_2(a_1(a_1-w_{a_1})-1)}\right].$$}

where $w_{m_{0,1}}$ and $w_{m_{1,1}}$ are the corresponding T-weights in $(a_1<m_{0,1}<m_{1,1})$, also $w_{a_2}$ is taken in $(a_2<a_1<m_{0,1})$. Observe that in the triple $(a_2<a_1<m_{0,1})$, $a_1$ and $m_{0,1}$ have T-weights $a_1-w_{a_1}$ and $m_{0,1}-w_{m_{0,1}}$ respectively. Therefore, the latter chain corresponds to Flip 4 of the MMP on $\F_1  \rightsquigarrow W_{m_{0,1},m_{1,1},m}$.

\vspace{0.2cm}
Since the triple $(m_{0,1},m_{1,1},m)$ is in the general form of Lemma \ref{lemma8}, we fall into the pair of situations described in the proof of the lemma. If $f<a_2$ we shall perform four more flips to reach Flip 4 of the MMP on $\F_1  \rightsquigarrow W_{a_1,m_{0,1},m_{1,1}}$. Similarly, if $a_2<f$ we shall perform three more flips to reach Flip 5 of the MMP on $\F_1  \rightsquigarrow W_{a_1,m_{0,1},m_{1,1}}$. 
\end{proof}

\bigskip
With these two lemmas we have proved Theorem \ref{changes-branches}. To end this section we proceed to prove Theorem \ref{complete-MMP} which consists of bounding the number of flips for a given triple $(a<b<c)$ in terms of their branch changes required to reach a triple in the Fibonacci branch. For a Markov triple $(1<a<b<c)$, we define $F(a,b,c):=\# \textit{ flips of the MMP on } \F_1  \rightsquigarrow  W_{a,b,c}$. In the following corollaries, we combine the results of the Lemmas \ref{lemma8} and \ref{lemma12} to compute the number of flips after the change of two branches in general form.

\begin{corollary}
Let us consider a Markov triple $(a=a_0<b=m_{0,0}<c=m_{1,0})$ that satisfies the conditions of Lemma \ref{lemma8}. We obtain one of the following changes in branches:
$$(a_3<m_{k_3,3}<m_{k_3+1,3})\to(a_2<m_{0,2}<m_{1,2})\to (a_1<m_{0,1}<m_{1,1})$$
or
$$(a_2<m_{k_2-1,2}<m_{k_2,2})\to(a_2<m_{k_2,2}<m_{k_2+1,2})\to (a_1<m_{0,1}<m_{1,1}).$$

Assuming that $(a_2<m_{0,2}<m_{1,2})$ belongs to the general situations of either Lemma \ref{lemma8} or Lemma \ref{lemma12}, we deduce the following:

\begin{itemize}
    \item If $k_2=0$, then $$F(a,b,c)=F(a_1,m_{0,1},m_{1,1})+4=F(a_2,m_{0,2},m_{1,2})+l,$$ where $l\in\{6,8\}$ is determined by whether $(a_3<m_{k_3,3}<m_{k_3+1,3})$ has $k_3>0$ or $k_3=0$, respectively. 
\vspace{0.1cm}
    \item If $k_2>0$, then 
    $$F(a,b,c)= F(a_1,m_{0,1},m_{1,1})+2=F(a_2,m_{0,2},m_{1,2})+l,$$
    where $l\in \{8,10\}$ is determined by whether $(a_3=3a_1a_2-a<m_{k_3,3}=a_2<m_{k_3+1,3}=m_{0,2})$ has $k_3>0$ or $k_3=0$, respectively.
\end{itemize}
\label{cor8}
\end{corollary}

\vspace{0.2cm}
\begin{corollary}
Let us consider a Markov triple $(a=a_0<b=m_{0,0}<c=m_{1,0})$ that satisfies the conditions of Lemma \ref{lemma12}. Let $m:=3m_{0,1}m_{1,1}-a_1$, we obtain one of the following changes in branches: $$(a_2<m_{0,2}<m_{1,2})\to (a_1<m_{0,1}<m_{1,1})\to (m_{0,1}<m_{1,1}<m)$$
or
$$(a_2<m_{k_2,2}<m_{k_2+1,2})\to (a_1<m_{0,1}<m_{1,1})\to (m_{0,1}<m_{1,1}<m).$$

Assuming that $(a_2<m_{0,2}<m_{1,2})$ belongs to the general situations of either Lemma \ref{lemma8} or Lemma \ref{lemma12}, we deduce the following:

\begin{itemize}
    \item If $k_2=0$, then $$F(a,b,c)=F(a_1,m_{0,1},m_{1,1})+8=F(a_2,m_{0,2},m_{1,2})+l$$ where $l\in\{10,12\}$ is determined by whether $(a_3<m_{k_3,3}<m_{k_3+1,3})$ has $k_3>0$ or $k_3=0$, respectively. 
\vspace{0.1cm}
    \item If $k_2>0$, then 
    $$F(a,b,c)= F(a_1,m_{0,1},m_{1,1})+6=F(a_2,m_{0,2},m_{1,2})+l$$
    where $l\in \{12,14\}$ is determined by whether $(a_3=3a_1a_2-a<m_{k_3,3}=a_2<m_{k_3+1,3}=m_{0,2})$ has $k_3>0$ or $k_3=0$, respectively.
\end{itemize}
\label{cor12}
\end{corollary}

\vspace{0.2cm}
From our calculations, the most effective strategy to construct a triple $(a<b<c)$ such that the MMP of $\F_1  \rightsquigarrow W_{a,b,c}$ has the greatest number of flips is by taking Corollary \ref{cor12} with $k_2>0$ iteratively. Indeed, this is accomplished by making every general branch change to add six flips, until we reach a triple $(a_t,m_{0,t},m_{1,t})$ such that $a_{t+2}=1$. This triple will have an MMP $\F_1  \rightsquigarrow W_{a_t,m_{0,t},m_{1,t}}$ that is a particular case of the preceding lemmas. Under this approach, the subsequent proposition provides us an effective upper limit for the number of flips in the total MMP.

\vspace{0.2cm}
\begin{proposition}
Let $(1<a=a_0<b=m_{i,0}<c=m_{i+1,0})$ be a Markov triple with $\nu$ branch changes to reach $a_{\nu}=1$. Then $F(a,b,c)$ is bounded by $6\nu +3 $. The bound is optimal for infinitely many triples. 

\label{estimacion}
\end{proposition}

\begin{proof}
For $\nu=1$, the bound is direct from the computations in the Sections \ref{s62} and \ref{s63}. For $\nu \geq 2$, we make the estimation of $F(a,b,c)$ by cases. 

\bigskip 
\textbf{Case 1:} Suppose that $\nu-1$ branch change results in a triple $(a_{\nu-1}<m_{k_{\nu-1},\nu-1}<m_{k_{\nu-1}+1,\nu-1})$ having $j>0$. This is obtained from the mutation of a triple $(a_{\nu-2}<m_{0,\nu-2}<m_{1,\nu-2})$. From the case $(a_2=1)$ of Lemma \ref{lemma12}, we infer that the MMP of $\F_1  \rightsquigarrow  W_{a_{\nu-2},m_{0,\nu-2},m_{1,\nu-2}}$ consists of at most 13 flips.

\vspace{0.1cm}
 Taking into account the number of branch changes in the general form of Lemma \ref{lemma8}, we denote $F_{1,1}$ as the number of times where the condition $(u<a_2)$ is satisfied. Similarly, we denote $F_{1,2}$ as the number of instances where the condition $(a_2<u)$ holds. Under the same reasoning for the branch changes in the general form of Lemma \ref{lemma12}, we define $F_{2,1}$ and $F_{2,2}$ accordingly.
 We observe that $\nu-2=F_{1,1}+F_{2,1}+F_{2,1}+F_{2,2}$. Additionally, note that under the notation above, the triple $(a_{\nu-2}<m_{0,\nu-2}<m_{1,\nu-2})$ is reached by a change $1,2$ or $2,2$ from $(a_{\nu-3}<m_{0,\nu-3}<m_{1,\nu-3})$. Thus, $\min\{F_{1,2},F_{2,2}\}>1$.

\vspace{0.2cm}
We obtain the formula,
\begin{equation}
F(a,b,c)\leq 13+2F_{1,2}+4F_{1,1}+6F_{2,2}+8F_{2,1}+2(1-\delta_{0,i}),
\end{equation}
where $\delta_{i,0}$ is the Kronecker delta. The term $2(1-\delta_{0,i})$ arises from the fact that $(a_0,m_{i,0},m_{i+1,0})$ may have $i>0$. Since the the last general branch change is not of the form 2,1, we can infer through successive application of Corollary \ref{cor12} that for a fixed $\nu$, $F$ is maximized by a triple $(a_0,m_{i,0},m_{i+1,0})$ that satisfies the following conditions:
\begin{enumerate}
    \item $a_{\nu-1}>2$. 
    \item $\nu-2=F_{2,2}$.
    \item $i>0$.
\end{enumerate}
In such case, we obtain $F(a,b,c)=15+6(\nu-2)=6\nu+3$.

\bigskip 
\textbf{Case 2:} Now, consider the case where the $\nu-1$ branch change falls into the triple $(a_{\nu-1},m_{0,\nu-1},m_{1,\nu-1})$.

\vspace{0.1cm}
If $\nu=2$, the computations in Section \ref{s63} let us infer that $F(a,b,c)=9$. For $\nu \geq 3$, the MMP of $\F_1  \rightsquigarrow  W_{a_{\nu-3},m_{0,\nu-3},m_{1,\nu-3}}$ consists of at most 13 flips. Similarly to the previous case, we again have that $(a_{\nu-3}<m_{0,\nu-3}<m_{1,\nu-3})$ is reached by a change $1,2$ or $2,2$ from $(a_{\nu-4}<m_{0,\nu-4}<m_{1,\nu-4})$.

\vspace{0.1cm}
Avoiding the situation where $(a_{\nu-4}<m_{0,\nu-4}<m_{1,\nu-4})$ attains the condition $(v=1,a_2<u)$  as in Lemma \ref{lemma8}, the number of branch changes in general form satisfy $\nu-3=F_{1,1}+F_{2,1}+F_{2,1}+F_{2,2}$ and the inequality $(7.1)$ holds as well.

\vspace{0.1cm}
In this situation and with a fixed $\nu$, $F$ is maximized by a triple $(a_0,m_{i,0},m_{i+1,0})$ that satisfies the following conditions:
\begin{enumerate}
    \item $a_{\nu-2}>5$. 
    \item $\nu-3=F_{2,2}$.
    \item $i>0$.
\end{enumerate}
In such case,we derive that $F(a,b,c)=15+6(\nu-3)=6\nu-3$. Given that $F$ is not maximized through the case $(v=1)$, as it only contributes two additional flips compared to $(a_{\nu-3},m_{0,\nu-3},m_{1,\nu-3})$, the proposition follows from Case 1.
\end{proof}

In summary, Lemmas \ref{lemma8} and \ref{lemma12} demonstrate that the number of flips will approach infinity if and only if $\nu$ does as well. The last proposition presents the optimal bound as outlined in Theorem \ref{complete-MMP}. With these points in consideration, we can affirm the proof of this theorem.


\begin{example}

The approach outlined in Proposition \ref{estimacion} for determining the maximum number of flips in the Markov tree is by choosing a \say{zigzag} path starting with $(1<b<c)$ where $b\geq 5$. This is: $$ (1<b<c) - (b<c<3bc-1) - (b<3bc-1<3b(3bc-1)-c) - (3bc-1<\ldots) - \ldots$$ and repeat the patron. For example, The Markov triple $$(1686049, 981277621, 4963446454828093)$$ needs $\nu=3$ branch changes to arrive at the Fibonacci branch. These are given by $a_3=1 \mapsto a_2=5 \mapsto a_1=194 \mapsto a_0=1686049$. The corresponding MMP requires $19=6 \cdot 3+1$ flips to end in $\F_1 \rightsquigarrow \F_7$. Thus, by taking one mutation in the branch of $a_0=1686049$ , we obtain the Markov triple $$(1686049,4963446454828093,25105841795148372846050),$$ whose MMP needs $21=6\cdot3+3$ flips to finish in $\F_1 \rightsquigarrow \F_7$.

\end{example}

\appsection{Computations for the MMP on $(1=a<b<c)$} \label{app}

We now proceed to describe in detail the computations of the MMP for Subsection \ref{s6.3}, omitting the calculation of the self-intersection of the flipped curve and other redundant flip calculations. We utilize the basic properties of T-weights as outlined in \cite[Lemma 5.5]{Per22} and apply \cite[Prop 2.15]{Urz16b} to compute the k1A flips after Flip 1. For each step we list the extremal k1A/k2A neighborhood, the associated $\delta$ and the Mori recursion as described Section \ref{s5}. We recall that for any given Markov number we have the computer program \cite{programa} that runs explicitly the MMP.

\vspace{0.2cm}
As in Subsection \ref{s6.3}, we start with the branch 
{\scriptsize  $$(1<b<c)-(b<c<3bc-1)-(c<m_0<m_1)-(c<m_1<m_2)-\ldots-(c<m_{k}<m_{k+1})-\ldots$$}
Recall that in this case we set $m_{-1}=b$ and $m_{-2}=1$. Additionally, the equations $w_c=3b-c$ and $w_b=3w_c-b$ hold. For $k\geq 1$, the computations of the MMP are as follows:

\vspace{0.1cm}

\noindent \textbf{Step 1:} \\
We replace the corresponding k1A neighborhood with the k2A:\\ 
{\tiny$\left[{c^2 \choose cw_c+1}\right]-(1)_--\left[{ m_{k+1}\choose w_{m_{k+1}}}\right]$}\\
We obtain the data
\begin{itemize}
    \item $\delta=c(cw_{m_{k+1}}-m_{k+1}w_c)-m_{k+1}=3cm_k-m_{k+1}=m_{k-1}$.
    \item $n(1)=m_{k+1}$, $n(0)=c^2$, $n(-1)=m_{k-1}c^2-m_{k+1}<0$.
    \item $a(1)=w_{m_{k+1}}$, $a(0)=cw_c+1$, $a(-1)=m_{k-1}(cw_c+1)-w_{m_{k+1}}$.
\end{itemize}
Therefore, the flipped curve is given by: 
\vspace{-0.1cm}
{\tiny$$\left[{ m_{k+1}-c^2m_{k-1} \choose w_{m_{k+1}}-m_{k-1}(c w_c+1)}\right]-(1)_{+}-\left[{c^2 \choose c w_c+1}\right].$$}

\vspace{0.2cm}
\noindent \textbf{Step 2:} 
{\tiny$\left[{c^2 \choose cw_c+1}\right]-(1)_--\left[{m_k\choose w_{m_k}}\right]$}\\
We obtain the data
\begin{itemize}
    \item $\delta=c(cw_{m_k}-m_kw_c)-m_k=3cm_{k-1}-m_k=m_{k-2}$.
    \item $n(1)=m_k$, $n(0)=c^2$, $n(-1)=m_{k-2}c^2-m_k<0$.
    \item $a(1)=w_{m_k}$, $a(0)=cw_c+1$, $a(-1)=m_{k-2}(cw_c+1)-w_{m_k}$.
\end{itemize}
Therefore, the flipped curve is given by:  

\vspace{-0.1cm}
{\tiny$$\left[{ m_{k}- c^2m_{k-2} \choose w_{m_k}-m_{k-2}(c w_c+1)}\right]-(1)_{+}-\left[{c^2 \choose c w_c+1}\right].$$}
\vspace{0.2cm}

\noindent \textbf{Step 3:}
{\tiny$\left[{ m_{k+1}-c^2m_{k-1} \choose w_{m_{k+1}}-m_{k-1}(c w_c+1)}\right]-(1)_{-}-\left[{ m_{k}- c^2m_{k-2} \choose w_{m_k}-m_{k-2}(c w_c+1)}\right]$}\\
Initially we focus on the $k=0$ case, for which we obtain the data
\begin{itemize}
    \item $\delta=m_1w_{m_0}-m_0w_{m_1}+c(cw_{m_1}-m_1w_c)+bc(m_0w_c-cw_{m_0})+bm_0-m_1$\\ $=3c+3cm_0-m_1-3b^2c+bm_0=3c.$
    \item $n(0)=m_1-bc^2$, $n(1)=m_0-c^2$, $n(2)=b-(3c-b)c^2<0$.
    \item $a(0)=w_{m_1}-b(cw_c+1)$, $a(1)=w_{m_0}-(cw_c+1)$, $a(2)=w_b-(3c-b)(cw_c+1)$.
\end{itemize}

\vspace{0.1cm}
The induced recursions of the $k=0$ neighborhood are $n(1-k) = m_k-c^2 m_{k-2}$ and $a(1-k) = w_{m_k}-m_{k-2}(c w_c+1)$, given that $m_k = 3c m_{k-1}-m_{k-2}$ and $w_{m_k} = 3c w_{m_{k-1}}-w_{m_{k-2}}$. Thus, the collection of neighborhoods for $k\geq 0$ form a Mori train and consequently carry the same birational information. Noting that $m_0-c^2=b^2$ and $w_{m_0}-cw_c-1=bw_b-1$, we observe that the flipped curve for any $k\geq 0$ is given by: 
\vspace{-0.1cm}

{\tiny$$\left[{b^2 \choose bw_b-1}\right]-(1)_{+}-\left[{ (3c-b)c^2-b \choose (3c-b)(cw_c+1)-w_b}\right].$$}
\vspace{0.2cm}

\noindent \textbf{Step 4:} 
{\tiny$\left[{c \choose w_c}\right]-(1)_{-}-\left[{b^2 \choose bw_b-1}\right]$}\\
We obtain the data
\begin{itemize}
    \item $\delta=b(cw_b-bw_c)-c=3b-c=w_c$.
    \item $n(1)=b^2$, $n(0)=c$, $n(-1)=1$, $n(-2)=w_c-c<0$.
    \item $a(1)=bw_b-1$, $a(0)=w_c$, $a(-1)=0$, $a(-2)=-w_c$.
\end{itemize}
Therefore, the flipped curve is given by: {\scriptsize$\left[{c-w_c \choose w_c}\right]-(2)_{+}.$}
\vspace{0.2cm}

\noindent \textbf{Step 5:} 
{\tiny$ (1)_{-}-\left[{ (3c-b)c^2-b \choose (3c-b)(cw_c+1)-w_b}\right]$}\\
We obtain the data
\begin{itemize}
    \item $\delta=(3c-b)(cw_c+1)-w_b$.
    \item $n_1^\prime=(3c-b)(c(c-w_c)-1)-(b-w_b)$.
    \item $a_1^\prime=(3c-b)(cw_c+1)-w_b$.
\end{itemize}
Therefore, the flipped curve is given by: {\scriptsize$\left[{ (3c-b)(c(c-w_c)-1)-(b-w_b) \choose (3c-b)(cw_c+1)-w_b}\right]-(2)_+.$}
\vspace{0.2cm}

\noindent \textbf{Step 6:} 
{\tiny$ (1)_{-}-\left[{c^2 \choose c w_c+1}\right]$}\\
We obtain the data
\begin{itemize}
    \item $\delta=c w_c+1$.
    \item $n_1^\prime=c(c-w_c)-1$.
    \item $a_1^\prime=c w_c+1$.
\end{itemize}
Therefore, the flipped curve is given by: {\scriptsize$\left[{c(c-w_c)-1 \choose c w_c+1}\right]-(2)_+.$}
\vspace{0.2cm}

 \noindent \textbf{Step 7:} 
{\tiny$\left[{ (3c-b)(c(c-w_c)-1)-(b-w_b) \choose (3c-b)(cw_c+1)-w_b}\right]-(1)_{-}-\left[{c(c-w_c)-1 \choose c w_c+1}\right]$}\\
We obtain the data
\begin{itemize}
    \item $\delta=w_b(c(c-w_c)-1)-(cw_c+1)(b-w_b)=c(cw_b-bw_c)-b=3c-b$.
    \item $n(0)=(3c-b)(c(c-w_c)-1)-(b-w_b)$, $n(1)=c(c-w_c)-1$, $n(2)=b-w_b$, $n(3)=2+bw_b+cw_c-3cw_b=3+c^2+w_c(3b-8c)=-6$.
    \item $a(0)=(3c-b)(cw_c+1)-w_b$, $a(1)=c w_c+1$, $a(2)=w_b$, $a(3)=-2-c^2+w_c(8c-3b)=7$.
\end{itemize}

Since $-7\equiv 5 \pmod{6}$, the flipped curve is given by: {\tiny$$\left[{b-w_b \choose w_b}\right]-(3)_{+}-\left[{6\choose 5}\right],$$}
with the exception of the specific case where $b=2$, which is addressed in \ref{s6.3}.
\vspace{0.2cm}

\noindent \textbf{Step 8:} 
{\tiny$\left[{c-w_c \choose w_c}\right]-(1)_{-}-\left[{b-w_b \choose w_b}\right]$} 

\vspace{0.2cm}
We observe that this neighborhood lies in the Mori train governed by the antiflip family of the P-resolution over $\frac{1}{5}(1,1)$ with $n_0^\prime = n_1^\prime = 1$ and $c = 5$. Indeed, we set $a_0^\prime = 0$ and $a_1^\prime = 1$, which implies $\delta = 3$. The Mori train is defined by the recursive relations:
\begin{itemize}
    \item $n(i+1) = 3n(i) - n(i-1)$ with  $n(0) = 1$ and $n(-1) = -1$.
    \item $a(i+1) = 3a(i) - a(i-1)$
with $a(0) = 0$ and $a(-1) = -\delta = -3$.
\end{itemize}
Therefore, the flip is given by:
{\scriptsize$(5)_{+}.$}
\vspace{0.2cm}

\noindent \textbf{Step 9:} 
{\tiny$ \left[{6\choose 1}\right]-(1)_{-}$}\\
We obtain the data
\begin{itemize}
    \item $\delta=5$.
    \item $n_1^\prime=1$.
    \item $a_1^\prime=1$.
\end{itemize}
Therefore, the flipped curve is given by: {\scriptsize$(7)_{+}.$}
\vspace{0.5cm} 

Now, we proceed to exhibit the computations of the MMP associated to the other branch
{\scriptsize  $$(1<b<c)-(1<c<3c-b)-(c<m_0<m_1)-(c<m_1<m_2)-\ldots-(c<m_{k}<m_{k+1})-\ldots$$}

In this situation the equations $w_{m_0}=3c-2b$, $w_c=3b-c$ and $w_b=3w_c-b$ hold. For the MMP, we initially focus on the $k=0$ case.

\vspace{0.1cm}
\noindent \textbf{Step 1:}\\
We replace the corresponding k1A neighborhood with the k2A:\\ 
{\tiny$\left[{c^2 \choose c(c-w_c)+1}\right]-(1)_--\left[{ m_1\choose w_{m_1}}\right]$}\\
We obtain the data
\begin{itemize}
    \item $\delta=c(m_1w_c-c(m_1-w_{m_1}))-m_1=3cm_0-m_1=1$.
    \item $n(1)=m_1$, $n(0)=c^2$, $n(-1)=c^2-m_1<0$.
    \item $a(1)=w_{m_1}$, $a(0)=c(c-w_c)+1$, $a(-1)=c(c-w_c)+1-w_{m_1}$.
\end{itemize}
Therefore, the flipped curve is given by:
\vspace{-0.1cm}
{\tiny$$\left[{ m_{k+1}-c^2m_{k-1} \choose w_{m_{k+1}}-m_{k-1}(c w_c+1)}\right]-(1)_{+}-\left[{c^2 \choose c w_c+1}\right].$$}
\vspace{0.2cm}

\noindent \textbf{Step 2:} 
{\tiny$\left[{c^2 \choose c(c-w_c)+1}\right]-(1)_{-}-\left[{m_{0} \choose w_{m_0}}\right]$}\\
We obtain the data
\begin{itemize}
    \item $\delta=c(m_0w_c-c(m_0-w_{m_0}))-m_0=3c-m_0=b$.
    \item $n(0)=c^2$, $n(1)=m_0$, $n(2)=3bc-b^2-c^2=1$, $n(3)=b-m_0<0$.
    \item $n(0)=c(c-w_c)+1$, $a(1)=w_{m_0}$, $a(2)=b(3c-2b)-c(2c-3b)-1=1$, $a(3)=b-w_{m_0}$.
\end{itemize}
Therefore, the flipped curve is given by:
{\tiny$(2)_{+}-\left[{m_0-b \choose w_{m_0}-b }\right].$}
\vspace{0.2cm}

\noindent \textbf{Step 3:} 
{\tiny$\left[{ m_1-c^2 \choose w_{m_1}-(c(c-w_c)+1)}\right]-(1)_{-}$}\\
We obtain the data
\begin{itemize}
    \item $\delta=m_1-w_{m_1}+cw_c+1=m_0(m_0-w_{m_0})+1=bm_0+1.$
    \item $n_1^\prime=w_{m_1}-c(c-w_c)-1=m_0w_{m_0}-1$.
    \item $a_1^\prime=2(m_0w_{m_0}-1)-m_1+c^2=2(m_0w_{m_0}-1)-m_0^2$.
\end{itemize}
Therefore, the flipped curve is given by:
{\tiny$(2)_{+}-\left[{m_0w_{m_0}-1 \choose 2(m_0w_{m_0}-1)-m_0^2}\right].$}
\vspace{0.2cm}

\noindent \textbf{Step 4:} 
{\tiny$\left[{c \choose c-w_c}\right]-(1)_{-}$}\\
We obtain the data
\begin{itemize}
    \item $\delta=w_c.$
    \item $n_1^\prime=c-w_c$.
    \item $a_1^\prime=2(c-w_c)-c$.
\end{itemize}
Therefore, the flipped curve is given by:
{\tiny$(2)_{+}-\left[{c-w_c \choose c-2w_c}\right].$}
\vspace{0.2cm}

\noindent \textbf{Step 5:} 
{\tiny$\left[{c-w_c \choose c-2w_c}\right]-(1)_{-}-\left[{m_0w_{m_0}-1 \choose 2(m_0w_{m_0}-1)-m_0^2}\right]$}\\
We obtain the data
\begin{itemize}
    \item $\delta=c(m_0w_{m_0}-1)-m_0^2(c-w_c)=m_0(m_0w_c-(m_0-w_{m_0})c)-c=3m_0-c$.
    \item $n(1)=m_0w_{m_0}-1$, $n(0)=c-w_c$, $n(-1)=2+cw_c-m_0w_b=-6$.
    \item $a(1)=2(m_0w_{m_0}-1)-m_0^2$, $a(0)=c-2w_c$, $a(-1)=-5+bm_0-(3m_0-c)w_c=-13$.
\end{itemize}
Since $13\equiv 1 \pmod{6}$, the flipped curve is given by:
{\tiny$$\left[{6\choose 1}\right]-(3)_{+}-\left[{c-w_c \choose c-2w_c}\right].$$}

\noindent \textbf{Step 6:} {\tiny$\left[{c-w_c \choose c-2w_c}\right]-(1)_{-}-\left[{m_0-b \choose w_{m_0}-b }\right]$}

\vspace{0.1cm}
This neighborhood lies in the same Mori train of Step 8 in the previous branch. Therefore, the flipped curve is given by: {\scriptsize$(5)_{+}.$}
\vspace{0.2cm}

\noindent \textbf{Step 7:} {\tiny$\left[{6\choose 1}\right]-(1)_{-}$}\\
The flipped curve is given by: {\scriptsize$(7)_{+}.$}

\vspace{0.2cm}
Now we consider the $k\geq 1$ case. Indeed, the first three steps are very similar to those in the previous branch situation. We set $m_{-1}=1$.

\noindent \textbf{Step 1:} \\
We replace the corresponding k1A neighborhood to the k2A:\\ 
{\tiny$\left[{c^2 \choose c(c-w_c)+1}\right]-(1)_--\left[{ m_{k+1}\choose w_{m_{k+1}}}\right]$}\\
We obtain the data
\begin{itemize}
    \item $\delta=c(m_{k+1}w_c-c(m_{k+1}-w_{m_{k+1}}))-m_{k+1}=3cm_k-m_{k+1}=m_{k-1}$.
    \item $n(1)=m_{k+1}$, $n(0)=c^2$, $n(-1)=m_{k-1}c^2-m_{k+1}<0$.
    \item $a(1)=w_{m_{k+1}}$, $a(0)=c(c-w_c)+1$, $a(-1)=m_{k-1}(c(c-w_c)+1)-w_{m_{k+1}}$.
\end{itemize}
Therefore, the flipped curve is given by: 
\vspace{-0.1cm}
{\tiny$$\left[{ m_{k+1}-c^2m_{k-1} \choose w_{m_{k+1}}-m_{k-1}(c(c-w_c)+1)}\right]-(1)_{+}-\left[{c^2 \choose c(c-w_c)+1}\right].$$}

\vspace{0.2cm}
\noindent \textbf{Step 2:} 
{\tiny$\left[{c^2 \choose cw_c+1}\right]-(1)_--\left[{m_k\choose w_{m_k}}\right]$}\\
We obtain the data
\begin{itemize}
    \item $\delta=c(m_kw_c-c(m_k-w_{m_k}))-m_k=3cm_{k-1}-m_k=m_{k-2}$.
    \item $n(1)=m_k$, $n(0)=c^2$, $n(-1)=m_{k-2}c^2-m_k<0$.
    \item $a(1)=w_{m_k}$, $a(0)=c(c-w_c)+1$, $a(-1)=m_{k-2}(c(c-w_c)+1)-w_{m_k}$.
\end{itemize}
Therefore, the flipped curve is given by:  

\vspace{-0.1cm}
{\tiny$$\left[{ m_{k}- c^2m_{k-2} \choose w_{m_k}-m_{k-2}(c(c-w_c)+1)}\right]-(1)_{+}-\left[{c^2 \choose c(c-w_c)+1}\right].$$}
\vspace{0.2cm}

\noindent \textbf{Step 3:}
{\tiny$\left[{ m_{k+1}-c^2m_{k-1} \choose w_{m_{k+1}}-m_{k-1}(c(c-w_c)+1)}\right]-(1)_{-}-\left[{ m_{k}- c^2m_{k-2} \choose w_{m_k}-m_{k-2}(c(c- w_c)+1)}\right]$}

\vspace{0.1cm}
We repeat the technique of the Step 3 of the previous branch, but for the $k=1$ case. To minimize space, we skip the computation of $\delta$.
\begin{itemize}
    \item $\delta=3c.$
    \item $n(0)=m_2-m_0c^2$, $n(1)=m_1-c^2$, $n(2)=m_0-bc^2<0$.
    \item $a(0)=w_{m_2}-b(c(c-w_c)+1)$, $a(1)=w_{m_1}-(c(c-w_c)+1)$, $a(2)=w_0-b(c(c-w_c)+1)$.
\end{itemize}
Again, these neighborhoods determine the Mori train defined by the recursions: $n(2-k) = m_k-c^2 m_{k-2}$ and $a(2-k) = w_{m_k}-m_{k-2}(c w_c+1)$. Therefore, the flipped curve is given by:  
\vspace{-0.1cm}
{\tiny$$\left[{ m_1-c^2 \choose w_{m_1}-(c (c-w_c)+1)}\right]-(1)_{+}-\left[{ bc^2-m_0 \choose b(c(c-w_c)+1)-w_{m_0}}\right].$$}

\vspace{0.2cm}
\noindent \textbf{Step 4:} 
{\tiny$\left[{ bc^2-m_0 \choose b(c(c-w_c)+1)-w_{m_0}}\right]-(1)_{-}-\left[{c^2 \choose c(c-w_c)+1}\right]$}\\
We obtain the data
\begin{itemize}
    \item $\delta=c(cw_{m_0}-m_0(c-w_c))-m_0=3c-m_0=b$.
    \item $n(0)=bc^2-m_0$, $n(1)=c^2$, $n(2)=m_0$, $n(3)=1$, $n(4)=b-m_0<0$.
    \item $n(0)=b(c(c-w_c)+1)-w_{m_0}$, $a(1)=c(c-w_c)+1$, $a(2)=w_{m_0}$, $a(3)=1$, $a(4)=b-w_{m_0}$.
\end{itemize}
Therefore, the flipped curve is given by:
{\tiny$(2)_{+}-\left[{m_0-b \choose w_{m_0}-b }\right].$}
\vspace{0.2cm}

\bibliographystyle{plain}
\bibliography{math}

\end{document}